\documentclass{amsart}
\usepackage{amsmath,amssymb,amsthm}
\usepackage[msc-links,abbrev]{amsrefs}
\usepackage{mathtools}
\usepackage{mathrsfs}
\usepackage{tikz-cd}

\DeclareMathOperator{\Pf}{Pf}
\DeclareMathOperator{\Ric}{Ric}
\DeclareMathOperator{\Rm}{Rm}
\DeclareMathOperator{\cRm}{\widetilde{Rm}}
\DeclareMathOperator{\tr}{tr}

\DeclareMathOperator{\dvol}{dV}
\DeclareMathOperator{\darea}{dA}

\DeclareMathOperator{\odivsymbol}{\overline{div}}
\DeclareMathOperator{\Id}{Id}
\DeclareMathOperator{\Area}{Area}

\newcommand{\bR}{\mathbb{R}}
\newcommand{\mE}{\mathcal{E}}
\newcommand{\mG}{\mathcal{G}}
\newcommand{\mN}{\mathcal{N}}
\newcommand{\mS}{\mathcal{S}}
\newcommand{\cmE}{\widetilde{\mathcal{E}}}
\newcommand{\cmG}{\widetilde{\mathcal{G}}}
\newcommand{\cmP}{\widetilde{\mathcal{P}}}
\newcommand{\cmS}{\widetilde{\mathcal{S}}}
\newcommand{\cg}{\widetilde{g}}
\newcommand{\ch}{\widetilde{h}}
\newcommand{\cjmath}{\widetilde{\jmath}\,}
\newcommand{\cD}{\widetilde{D}}
\newcommand{\cH}{\widetilde{H}}
\newcommand{\cI}{\widetilde{I}\,}
\newcommand{\cM}{\widetilde{M}}

\newcommand{\cR}{\widetilde{R}}
\newcommand{\cT}{\widetilde{T}}
\newcommand{\cU}{\widetilde{U}}
\newcommand{\cV}{\widetilde{V}}
\newcommand{\cW}{\widetilde{W}}
\newcommand{\cX}{\widetilde{X}}
\newcommand{\cL}{\widetilde{L}}
\newcommand{\cdelta}{\widetilde{\delta}}
\newcommand{\cgamma}{\widetilde{\gamma}}
\newcommand{\crho}{\widetilde{\rho}}
\newcommand{\cnabla}{\widetilde{\nabla}}
\newcommand{\cDelta}{\widetilde{\Delta}}
\newcommand{\cGamma}{\widetilde{\Gamma}}
\newcommand{\cSigma}{\widetilde{\Sigma}}
\newcommand{\hg}{\widehat{g}}
\newcommand{\hj}{\widehat{\jmath}\,}

\newcommand{\hSigma}{\widehat{\Sigma}}
\newcommand{\hM}{\widehat{M}}
\newcommand{\kc}{\mathfrak{c}}
\newcommand{\rv}{\rvert}
\newcommand{\lv}{\lvert}
\newcommand{\onabla}{\overline{\nabla}}
\newcommand{\oRm}{\overline{\Rm}}
\newcommand{\Cphg}{C_{\mathrm{ph}}}
\newcommand{\mF}{\mathcal{F}}
\newcommand{\mFphg}{\mathcal{F}_{\mathrm{ph}}}

\newcommand{\defn}[1]{{\boldmath\bfseries#1}}

\newcommand{\suchthat}{\mathrel{}:\mathrel{}}

\def\sideremark#1{\ifvmode\leavevmode\fi\vadjust{\vbox to0pt{\vss% the remark
 \hbox to 0pt{\hskip\hsize\hskip1em%                          will appear only
 \vbox{\hsize3cm\tiny\raggedright\pretolerance10000%          on the side
  \noindent #1\hfill}\hss}\vbox to8pt{\vfil}\vss}}}%
\newtheorem{theorem}{Theorem}[section]
\newtheorem{lemma}[theorem]{Lemma}
\newtheorem{proposition}[theorem]{Proposition}
\newtheorem{corollary}[theorem]{Corollary}

\theoremstyle{definition}
\newtheorem{definition}[theorem]{Definition}

\theoremstyle{remark}
\newtheorem{remark}[theorem]{Remark}

\numberwithin{equation}{section}

\newcommand{\intr}{\sideset{^R}{}\int}
\newcommand{\smallintr}{{}^R\!\!\int}
\newcommand{\intl}{\mathring{L}}

\begin{document}

\title{Local and global conformal invariants of submanifolds}
% \title[Computing renormalized extrinsic curvature integrals]{Renormalized curvature integrals on minimal submanifolds of conformally compact Einstein manifolds}
\author[J.\ S.\ Case]{Jeffrey S.\ Case}
\address{Department of Mathematics \\ Penn State University \\ University Park, PA 16802 \\ USA}
\email{jscase@psu.edu}
\author[A. Khaitan]{Ayush Khaitan}
\address{Department of Mathematics \\ Rutgers University \\ Hill Center for the Mathematical Sciences \\ 110 Frelinghuysen Rd. \\ Piscataway, NJ 08854 \\ USA}
\email{ayush.khaitan@rutgers.edu}
\author[Y.-J.\ Lin]{Yueh-Ju Lin}
\address{Department of Mathematics, Statistics, and Physics \\ Wichita State University \\ Wichita, KS 67260 \\ USA}
\email{yueh-ju.lin@wichita.edu}
\author[A.\ J.\ Tyrrell]{Aaron J.\ Tyrrell}
\address{Department of Mathematics \\ University of Notre Dame \\ Notre Dame, IN 46556 \\ USA}
\email{atyrrell@nd.edu}
\author[W.\ Yuan]{Wei Yuan}
\address{Department of Mathematics \\ Sun Yat-sen University \\ Guangzhou, Guangdong 510275 \\ China}
\email{yuanw9@mail.sysu.edu.cn}
\begin{abstract}
 We develop methods for constructing and computing conformal invariants of submanifolds, with a particular emphasis on conformal submanifold scalars and conformally invariant integrals of natural submanifold scalars.
 These methods include a direct construction of the extrinsic ambient space,
 a construction of global invariants of conformally compact minimal submanifolds of conformally compact Einstein manifolds via renormalized extrinsic curvature integrals,
 and the introduction of a large class of conformal submanifold scalars that are easily computed at minimal submanifolds of Einstein manifolds.
 As an application, we derive an explicit Gauss--Bonnet--Chern-type formula relating the renormalized area of a conformally compact $k$-dimensional minimal submanifold of a conformally compact Einstein manifold to its Euler characteristic and the integral of a conformal submanifold scalar of weight $-k$.
 As another application, we prove a rigidity result for conformally compact minimal submanifolds of conformally compact hyperbolic manifolds.
\end{abstract}
\maketitle

\section{Introduction}
\label{sec:intro}

Minimal submanifolds of Einstein manifolds, and especially of spaceforms, have long captured the attention of mathematicians, and conformal invariants of submanifolds have played an important role in their study.
The simplest conformal invariant is the trace-free part of the second fundamental form~\cite{Fialkow1944}.
Two applications of it, in the guise of a holomorphic quadratic differential, are the facts that a compact minimal surface in an Einstein three-manifold (a) is umbilic if it has genus zero~\cites{Almgren1966,Chern1983}, and (b) is either umbilic or has no umbilic points if it has genus one~\cite{Lawson1970}.
Another important conformal invariant is the Willmore energy~\cites{LiYau1982,MontielUrbano2002,Willmore1965,Weiner1978} $\int \bigl( \lambda + \lv H \rv^2 \bigr) \darea$ of a compact surface in an Einstein manifold $(M^n,g)$ with $\Ric = (n-1)\lambda g$.
This formula shows that minimal surfaces are critical points of the Willmore energy, which plays an essential role in the resolution of the Willmore Conjecture in dimension three~\cite{MarquesNeves2014}.
In higher dimensions, the renormalized area is an important global invariant of even-dimensional conformally compact minimal submanifolds of conformally compact Einstein manifolds~\cites{GrahamWitten1999,RyuTakayanagi2006a,RyuTakayanagi2006b}.
While difficult to compute in general, the renormalized area can be expressed as a linear combination of the Euler characteristic of the minimal submanifold and the convergent integral of a conformal submanifold scalar in dimension two~\cite{AM2010}, dimension four~\cites{CGKTW24,Tyrrell2023}, and under the assumption of a conjectural Alexakis-type decomposition (cf.\ \cites{Alexakis,CGKTW24,MondinoNguyen2018}) in higher dimensions~\cite{CGKTW24}.

In this paper we develop methods that systematically construct and compute local and global conformal invariants of submanifolds.
Our approach to constructing local invariants is via the extrinsic analogue of the (Fefferman--Graham) ambient space~\cite{FeffermanGraham2012} introduced by Case, Graham, and Kuo~\cite{CGK23};
our approach to constructing global invariants is via an extrinsic analogue of the renormalized curvature integrals of Albin~\cite{Albin2009};
and our approach to computing local and global invariants, which applies to minimal submanifolds of Einstein manifolds, generalizes work of Case, Khaitan, Lin, Tyrrell, and Yuan~\cite{CKLTY2024}.
We use these results to compute a large class of renormalized extrinsic curvature integrals on conformally compact minimal submanifolds of conformally compact Einstein manifolds.
In particular, we compute the renormalized area independent of an Alexakis-type decomposition, significantly improving a result of Case, Graham, Kuo, Tyrrell, and Waldron~\cite{CGKTW24}.
Our methods also lead to a rigidity result for minimal submanifolds of hyperbolic manifolds.

Our first main result is a direct construction of the extrinsic ambient space~ and its application to the construction of conformal submanifold scalars:

\begin{theorem}
    \label{thm:rough-ambient-space}
    Let $j \colon \Sigma^k \to (M^n,\kc)$, $k < n$ and $n \geq 3$, be a conformal submanifold and let $(\cmG,\cg)$ be an ambient space for $(M,\kc)$.
    There is a formally unique, formally minimal, dilation-equivariant submanifold $\cjmath \colon \cmS \to (\cmG,\cg)$ such that $(\cmS,\cjmath^\ast\cg)$ is a pre-ambient space for $(\Sigma,j^\ast\kc)$ and $\widetilde{\jmath}$ restricts to the tautological immersion of the metric bundle of $(\Sigma,j^\ast\kc)$ into that of $(M,\kc)$.
    Moreover, if $\widetilde{I}$ is a natural submanifold scalar of homogeneity $w \geq -k$ on $(k+2)$-submanifolds of $(n+2)$-manifolds, then $\cI$ descends to a conformal submanifold scalar $\iota^\ast\cI$ of weight $w$ on $k$-submanifolds of $n$-manifolds.
\end{theorem}

We call $\cjmath \colon \cmS \to (\cmG,\cg)$ the \emph{extrinsic ambient space} of $j$.
In our terminology, the conformal class of $\Sigma$ is not specified as part of the data of a conformal submanifold;
rather, $\Sigma$ inherits the conformal structure $j^\ast\kc$.
The map $j$ is an immersion, not necessarily injective, and the invariants we construct are local invariants of the unparameterized submanifold $j(\Sigma)$.
See Section~\ref{sec:bg} for definitions of conformal submanifolds, natural submanifold scalars, and conformal submanifold scalars.

Theorem~\ref{thm:rough-ambient-space} combines two results.
The first result, stated as Theorem~\ref{thm:extrinsic-ambient-space} below, is the existence and uniqueness of the extrinsic ambient space;
see Section~\ref{sec:ambient} for a precise formulation.
If $k$ is even, then there is an obstruction, regarded here as a natural submanifold section of the conormal bundle $N^\ast\Sigma$, to $\cjmath$ being smooth to all orders.
This obstruction was studied in detail by Graham and Reichert~\cite{GrahamReichert2020} via Poincar\'e spaces;
see Theorem~\ref{thm:obstruction} for a treatment via the extrinsic ambient space.
The second result, stated as Theorem~\ref{thm:construction-of-scalars} below, is that natural submanifold scalars descend to conformal submanifold scalars under suitable assumptions.
We expect that, in analogy with a result of Bailey, Eastwood, and Graham~\cite{BaileyEastwoodGraham1994}, \emph{all} (even) conformal submanifold scalars of weight $w \geq -k$ arise from the construction of Theorem~\ref{thm:rough-ambient-space};
the bound on $w$ stems from the aforementioned obstruction.

The first part of Theorem~\ref{thm:rough-ambient-space} is originally due to Case, Graham, and Kuo~\cite{CGK23}, who constructed the extrinsic ambient space as the homogeneous lift of the formally minimal extension of $j(\Sigma)$ into a Poincar\'e space for $(M^n,\kc)$;
the latter extension is due to Graham and Witten~\cite{GrahamWitten1999} and, in a more invariant way, Graham and Reichert~\cite{GrahamReichert2020}.
We present an independent proof of Theorem~\ref{thm:rough-ambient-space} for three reasons.
First, Case, Graham, and Kuo primarily focus on the extrinsic ambient space of a minimal submanifold of an Einstein manifold.
We clarify their results for general conformal submanifolds.
Second, our proof is direct, in that it does not require Poincar\'e spaces or minimal immersions therein.
This yields a conceptual simplification to our construction and computation of conformal submanifold invariants.
Third, our presentation focuses on the immersion $j$, which is necessary when considering global invariants.
This global perspective is only implicit in the work of Case, Graham, and Kuo, which concerned a construction of local invariants.

Curry, Gover, and Snell~\cite{CurryGoverSnell2023} developed a different approach to constructing conformal submanifold invariants based on the tractor calculus.
The primary benefit of our approach is that it enables us to compute conformal submanifold invariants via straightening.
We expect that there is a close link between our extrinsic ambient space and their extrinsic tractor calculus (cf.\ \cite{CapGover2003}).

Our second main result is a method for systematically computing conformal submanifold scalars.
For example:

\begin{theorem}
    \label{thm:compute-straight}
    Fix positive integers $k,n$ such that $n > k \geq 2$.
    Let $a,b,c$ be nonnegative integers such that $a+2b+2c \leq k$, and let $\cmP_{a,b}$ be a scalar polynomial of degree $a$ in the second fundamental form and degree $b$ in the Riemann curvature tensor, regarded as a natural submanifold scalar on $(k+2)$-submanifolds of $(n+2)$-manifolds.
    If $j \colon \Sigma^k \to (M^n,g)$ is a minimal submanifold of an Einstein manifold with $\Ric = (n-1)\lambda g$, then
    \begin{equation*}
        \bigl( \iota^\ast \cDelta^c \cmP_{a,b} \bigr)^{j^\ast g} = \left( \prod_{s=0}^{c-1} \bigl( \Delta^{j^\ast g} + (a+2b+2s)(k-a-2b-2s-1)\lambda \bigr) \right) \bigl( \iota^\ast \cmP_{a,b} \bigr)^{j^\ast g} .
    \end{equation*}
\end{theorem}

Here the Riemann curvature tensor is that of $(\cmG,\cg)$ and $\Delta := -\nabla^a\nabla_a$.
Direct computation implies that the scalar conformal invariant $\iota^\ast\cmP_{a,b}$ is the same polynomial $P_{a,b}$ of degree $a$ in the trace-free part of the second fundamental form and degree $b$ in the Weyl tensor, where the Weyl tensor is that of the target manifold $(M^n,g)$.
The key point of Theorem~\ref{thm:compute-straight} is that it explicitly expresses a conformal submanifold scalar of higher order in terms of a conformal submanifold scalar of low order modulo natural divergences, when evaluated at a minimal submanifold of an Einstein manifold.
This is particularly useful for computing global invariants, including renormalized extrinsic curvature integrals.

Theorem~\ref{thm:compute-straight} is the extrinsic analogue of a recent result of Case, Khaitan, Lin, Tyrrell, and Yuan~\cite{CKLTY2024}, and our proof is analogous:
If $j \colon \Sigma^k \to (M^n,g)$ is a minimal submanifold and $\Ric = (n-1)\lambda g$, then, as observed by Case, Graham, and Kuo~\cite{CGK23}, $\cjmath(t,x,\rho) := \bigl(t,j(x),\rho\bigr)$ defines an extrinsic ambient space
\begin{align*}
    \cjmath & \colon \bR_+ \times \Sigma^k \times (-\varepsilon,\varepsilon ) \to \bR_+ \times M^n \times (-\varepsilon,\varepsilon ) , \\
    \cg & := 2\rho\,dt^2 + 2t\,dt\,d\rho + \tau^2g , \\
    \tau & := t(1+\lambda\rho/2) .
\end{align*}
Direct computation~\cites{CLY23,Matsumoto2013} shows that there is a $P_{a,b} \in C^\infty(\Sigma)$ such that
\begin{equation*}
    \cmP_{a,b} = \tau^{-a-2b}\varpi^\ast P_{a,b} ,
\end{equation*}
where $\varpi \colon \mathbb{R}_+ \times \Sigma \times (-\varepsilon,\varepsilon) \to \Sigma$ is the canonical projection, and that
\begin{equation*}
    \cDelta^{\cjmath^\ast\cg}\bigl( \tau^w\varpi^\ast u \bigr) = \tau^{w-2}\varpi^\ast\bigl( \Delta^{j^\ast g} - w(k+w-1)\lambda \bigr)u
\end{equation*}
for all $u \in C^\infty(\Sigma)$ and $w \in \mathbb{R}$.
See Section~\ref{sec:straightenable} for details.

Our third main result is a general construction of global invariants of conformally compact minimal submanifolds of conformally compact Einstein manifolds:

\begin{theorem}
    \label{thm:albin}
    Fix integers $2 \leq k < n$ with $k$ even.
    Let $I$ be a natural submanifold scalar on $k$-submanifolds of $n$-manifolds.
    If $j \colon \Sigma^k \to (M^n,g_+)$ is a conformally compact minimal submanifold of a conformally compact Einstein manifold and if $r$ is a geodesic defining function for $\partial_\infty M$, then the integral $\int I \darea$ has an asymptotic expansion
    \begin{equation*}
        \int_{j^{-1}(\{ r>\varepsilon \})} I^{j,g_+} \darea_{j^\ast g_+} = a_{(0)}\varepsilon^{1-k} + a_{(2)}\varepsilon^{3-k} + \dotsm + a_{(k-2)}\varepsilon^{-1} + \mathscr{I} + o(1)
    \end{equation*}
    as $\varepsilon \to 0^+$, where $a_{(0)},\dotsc,a_{(k-2)},\mathscr{I} \in \mathbb{R}$.
    Moreover, $\mathscr{I}$ is independent of the choice of $r$, and hence defines a global invariant of $j \colon \Sigma^k \to (M^n,g_+)$.
\end{theorem}

See Section~\ref{sec:albin} for definitions of conformally compact (sub)manifolds, including of the conformal infinity $j_\infty \colon \partial_\infty \Sigma \to \partial_\infty M$.

Theorem~\ref{thm:albin} is the extrinsic analogue of a result of Albin~\cite{Albin2009}.
It allows one to define the \defn{renormalized extrinsic curvature integral} $\smallintr I \darea$ by
\begin{equation*}
 \intr I \darea := \mathscr{I} .
\end{equation*}
When $I=1$, this recovers the \defn{renormalized area} $\mathscr{A}$ of Graham and Witten~\cite{GrahamWitten1999}.
When $I = \Pf(\oRm)$ is the Pfaffian of the Riemann curvature tensor of $j^\ast g_+$, a result of Albin~\cite{Albin2009} yields the Gauss--Bonnet--Chern-type formula
\begin{equation}
    \label{eqn:renormalized-gbc}
    \intr \Pf(\oRm) \darea = (2\pi)^{k/2}\chi(\Sigma) .
\end{equation}

Similar to Albin, we prove Theorem~\ref{thm:albin} by carefully studying the asymptotic expansions of natural submanifold tensors in terms of a geodesic defining function for $\partial_\infty M$.
Indeed, we prove a general result about asymptotic expansions of integrals in all dimensional parities that depends only on the asymptotic behavior of the metrics $g_+$ and $j^\ast g_+$.
In Lemma~\ref{lem:renormalized-divergence-is-zero}, we also establish that the renormalized integral of a natural divergence vanishes (cf.\ \cite{CKLTY2024}).

Theorems~\ref{thm:rough-ambient-space}, \ref{thm:compute-straight}, and~\ref{thm:albin} allow us to compute a large class of renormalized extrinsic curvature integrals;
see Theorem~\ref{thm:straight}.
Specializing to Equation~\eqref{eqn:renormalized-gbc} yields the following Gauss--Bonnet--Chern-type formula involving the renormalized area:

\begin{corollary}
    \label{cor:renormalized-area}
    Let $j \colon \Sigma^k \to (M^n,g_+)$, $k < n$ and $k$ even, be a conformally compact minimal submanifold of a conformally compact Einstein manifold.
    Then
    \begin{equation}
        \label{eqn:renormalized-area}
        (2\pi)^{k/2}\chi(\Sigma) = (-1)^{k/2}(k-1)!!\mathscr{A} + \sum_{r=1}^{k/2} 2^{r-k/2}\frac{(r-1)!}{(k/2-1)!}\int_\Sigma \mathcal{P}_{r,k} \darea ,
    \end{equation}
    where $\mathcal{P}_{r,k} := \iota^\ast \bigl( \widetilde{\Delta}^{k/2-r}\Pf_r(\widetilde{\overline{\Rm}}) \bigr)$.
\end{corollary}

Here $\widetilde{\overline{\Rm}}$ is the Riemann curvature tensor of the induced metric $\cjmath^\ast\cg$ on $\cmS$, where $\cjmath \colon \cmS \to (\cmG,\cg)$ is the extrinsic ambient space;
see Section~\ref{subsec:pseudo-riemannian} for the definition of the Pfaffian-like polynomial $\Pf_r$.
Remarkably, Equation~\eqref{eqn:renormalized-area} is the \emph{same} formula computed by Case, Khaitan, et al.~\cite{CKLTY2024} for the renormalized volume of an even-dimensional conformally compact Einstein manifold, except that it is stated in terms of extrinsic invariants.
Since $\cjmath^\ast\cg$ need not be Ricci-flat, $\mathcal{P}_{1,k}$ need not vanish.

There are four key points to Corollary~\ref{cor:renormalized-area}.
First, $\mathcal{P}_{r,k}$ is a conformal submanifold scalar of weight $-k$, and hence its integral is convergent.
Second, our result is valid in all even dimensions without any additional assumptions.
In particular, this improves the aforementioned result of Case, Graham, et al.~\cite{CGKTW24} by removing its dependence on the conjectural Alexakis-type decomposition in dimensions $k \geq 6$.
Third, Equation~\eqref{eqn:renormalized-area} gives an explicit formula for the conformal submanifold scalar.
Case, Graham, et al.~\cite{CGKTW24} proved that in each even dimension $k \geq 4$ there are conformal submanifold scalars of weight $-k$ on $k$-submanifolds of $n$-manifolds that are natural divergences.
Thus there is some freedom in how one writes Equation~\eqref{eqn:renormalized-area}.
Fourth, if $j$ is an immersion into a locally conformally flat manifold, then the ambient space $(\cmG,\cg)$ may be taken to be flat~\cite{FeffermanGraham2012}.
The Gauss equation then yields the simplification
\begin{equation*}
    \mathcal{P}_{r,k} = 2^{-r}\iota^\ast \bigl( \cDelta^{k/2-r} \Pf_r( \widetilde{L} \wedge \widetilde{L} ) \bigr) ,
\end{equation*}
where $L \wedge L$ denotes the normal trace of the Kulkarni--Nomizu product:
\begin{equation*}
    (L \wedge L)_{\alpha\beta\gamma\delta} := 2L_{\alpha\gamma\epsilon'}L_{\beta\delta}{}^{\epsilon'} - 2L_{\alpha\delta\epsilon'}L_{\beta\gamma}{}^{\epsilon'} .
\end{equation*}

Theorem~\ref{thm:albin} enters the proof of Corollary~\ref{cor:renormalized-area} only when manipulating renormalized extrinsic curvature integrals, especially when eliminating divergences.
The same algebraic manipulations yield a Gauss--Bonnet--Chern-type formula on compact minimal submanifolds of Einstein manifolds:

\begin{corollary}
    \label{cor:compact-area}
    Let $j \colon \Sigma^k \to (M^n,g)$, $k < n$ and $k$ even, be a compact minimal submanifold of an Einstein manifold with $\Ric = (n-1)\lambda g$.
    Then
    \begin{equation*}
        (2\pi)^{k/2}\chi(\Sigma) = (k-1)!!\lambda^{k/2}\Area_{j^\ast g}(\Sigma) + \sum_{r=1}^{k/2} 2^{r-k/2}\frac{(r-1)!}{(k/2-1)!}\int_\Sigma \mathcal{P}_{r,k} \darea ,
    \end{equation*}
    where $\mathcal{P}_{r,k}$ is as in Corollary~\ref{cor:renormalized-area}.
\end{corollary}

We expect that Theorems~\ref{thm:rough-ambient-space}, \ref{thm:compute-straight}, and~\ref{thm:albin} have broad applications to rigidity results for conformally compact minimal submanifolds of conformally compact Einstein manifolds.
The following result should be prototypical:

\begin{theorem}
    \label{thm:rigidity}
    Let $j \colon \Sigma^k \to (M^n,g_+)$, $4 \leq k < n$ and $k$ even, be a conformally compact minimal submanifold of a conformally compact hyperbolic manifold.
    Suppose additionally that the conformal infinity $j_\infty \colon \partial_\infty\Sigma \to \partial_\infty M$ is umbilic.
    \begin{enumerate}
        \item For each $\ell \in \{ 1, \dotsc, k/2 \}$, it holds that
        \begin{equation}
            \label{eqn:totally-geodesic}
            \int_\Sigma \iota^\ast\left( (-\cDelta)^{k/2-\ell}\lv\widetilde{L}\rv^{2\ell} \right) \darea \geq 0
        \end{equation}
        with equality if and only if $j$ is totally geodesic.
        \item It holds that
        \begin{equation}
            \label{eqn:einstein}
            \int_\Sigma \iota^\ast\left( (-\cDelta)^{k/2-2}\lvert\widetilde{L}^2\rvert^2 \right) \darea \geq \frac{1}{k}\int_\Sigma \iota^\ast\left( (-\cDelta)^{k/2-2}\lvert\widetilde{L}\rvert^4 \right) \darea
        \end{equation}
        with equality if and only if $j$ is totally geodesic.
        \item If $n = k+1$, then
        \begin{equation}
            \label{eqn:lcf}
            \int_\Sigma \iota^\ast\left( (-\cDelta)^{k/2-2}\lvert\widetilde{L}^2\rvert^2 \right) \darea \leq \frac{k^2-3k+3}{k(k-1)}\int_\Sigma \iota^\ast\left( (-\cDelta)^{k/2-2} \lvert\widetilde{L}\rvert^4 \right) \darea
        \end{equation}
        with equality if and only if $(\Sigma,j^\ast g_+)$ is locally conformally flat.
    \end{enumerate}
\end{theorem}

Here $L_{\alpha\beta}^2 := L_{\alpha\gamma\gamma'}L_\beta{}^{\gamma\gamma'}$.
Locally conformally flat hypersurfaces of hyperbolic $n$-space, $n \geq 5$, are classified~\cite{DoCarmoDajczer1983}.

There are two key ingredients in the proof of Theorem~\ref{thm:rigidity}.
First, the classification of umbilic submanifolds of hyperbolic space and the fact~\cite{GrahamWitten1999} that $j$ mod $O(r^{k+1})$ is locally determined together imply that $\lvert L \rvert \in L^p(\Sigma)$ for all $p \in [1,\infty]$.
Second, our main results imply that the integrals appearing in Inequalities~\eqref{eqn:totally-geodesic}, \eqref{eqn:einstein}, and~\eqref{eqn:lcf} are proportional to the integrals of appropriate powers of $\lv L \rv^2$ and $\lv L^2\rv^2$.
The characterization of equality follows from the Gauss equations.

This paper is organized as follows:

In Section~\ref{sec:bg} we recall necessary background and fix our conventions.
This includes a discussion of natural invariants of Riemannian and conformal submanifolds.

In Section~\ref{sec:ambient} we give precise definitions of the extrinsic ambient space and extrinsic ambient equivalence, and then prove the first statement of Theorem~\ref{thm:rough-ambient-space}.

In Section~\ref{sec:invariants} we prove the second statement of Theorem~\ref{thm:rough-ambient-space}.

In Section~\ref{sec:straightenable} we introduce the notions of straight and straightenable submanifold tensors, and then give a systematic construction of straight submanifold scalars.
A special case of these results proves Theorem~\ref{thm:compute-straight}.
We also prove Corollary~\ref{cor:compact-area}.

In Section~\ref{sec:albin} we carefully discuss renormalized integrals on even asymptotically hyperbolic manifolds.
We also study the asymptotics of natural submanifold scalars on conformally compact minimal submanifolds of conformally compact Einstein manifolds.
We use this to prove Theorem~\ref{thm:albin} and the fact that the renormalized curvature integral of a natural divergence is zero.

In Section~\ref{sec:renormalize} we prove Corollary~\ref{cor:renormalized-area}.
We also compute the renormalized extrinsic curvature integral of a straightenable submanifold scalar.

In Section~\ref{sec:applications} we prove Theorem~\ref{thm:rigidity}.

\section{Background}
\label{sec:bg}

In this section we introduce relevant background about (immersed) submanifolds (with multiplicity) of pseudo-Riemannian and conformal manifolds, formulated via immersions.
Our conventions follow Case, Graham, Kuo, Tyrrell, and Waldron~\cite{CGKTW24}.
We also prove two technical results needed in Theorem~\ref{thm:rough-ambient-space}.
The first, stated as Proposition~\ref{prop:natural}, shows that our notion of natural submanifold tensors agrees with other definitions in the literature.
The second, stated as Proposition~\ref{prop:tubular-neighborhood}, identifies certain one-parameter families of immersions with one-parameter families of sections of the normal bundle.
Both results are known in the context of embeddings, but we could not find statements for immersions in the literature.

\subsection{Pseudo-Riemannian manifolds}
\label{subsec:pseudo-riemannian}

In this subsection we introduce some important Riemannian invariants.
The main purpose is to fix our conventions.

A \defn{pseudo-Riemannian manifold} $(M^n,g)$ is a pair of a smooth\footnote{
    By \defn{smooth}, we mean of class $C^\infty$.
}
$n$-manifold $M$ and a smooth section $g$ of $S^2T^\ast M$, called the pseudo-Riemannian metric, such that $g_p$ defines a nondegenerate inner product on $T_pM$ for each $p \in M$.
We say that $(M,g)$ and $g$ are \defn{Riemannian} if $g_p$ is positive definite for each $p \in M$.
With the exception of Section~\ref{sec:applications}, all of the results in this paper hold in general signature.

We perform computations using abstract index notation, using lowercase Latin letters ($a,b,c,\dotsc$) to denote factors of $T^\ast M$ (when subscripts) or $TM$ (when superscripts), and with repeated indices denoting a contraction via the canonical pairing of $TM$ and $T^\ast M$.
For example, we write $T_{abc}$ to denote a section of $\otimes^3 T^\ast M$ and $X^a$ to denote a vector field.
We denote evaluation of $T$ at vector fields $X,Y,Z$ by
\begin{equation*}
 T(X,Y,Z) = X^a Y^b Z^c T_{abc} .
\end{equation*}
We use square brackets and round parentheses to denote skew-symmetrization and symmetrization, respectively. For example,
\begin{align*}
 T_{[abc]} & := \frac{1}{6}(T_{abc}-T_{acb}+T_{bca}-T_{bac}+T_{cab}-T_{cba}) , \\
 T_{(abc)} & := \frac{1}{6}(T_{abc}+T_{acb}+T_{bca}+T_{bac}+T_{cab}+T_{cba}) .
\end{align*}
We use $g_{ab}$ and its inverse $g^{ab}$ to lower and raise indices, respectively.
We denote by $\Rm$ or $R_{abcd}$ the Riemann curvature tensor, defined by the convention
\begin{equation*}
 \nabla_a \nabla_b \tau_c - \nabla_b \nabla_a \tau_c = R_{abc}{}^d\tau_d .
\end{equation*}
If $n \geq 2$, then the \defn{Schouten scalar} is
\begin{equation*}
 J := \frac{R}{2(n-1)} ,
\end{equation*}
where $R := R_a{}^a$ is the scalar curvature and $R_{ab} := R_{acb}{}^c$ is the Ricci tensor.
If $n \geq 3$, then the \defn{Schouten tensor} is
\begin{equation*}
 P_{ab} := \frac{1}{n-2}\left( R_{ab} - J g_{ab} \right) .
\end{equation*}
Note that $J = P_a{}^a$.
The \defn{Kulkarni--Nomizu product} of two symmetric $(0,2)$-tensors $S$ and $T$ is
\begin{equation*}
    (S \wedge T)_{abcd} := 2S_{a[c}T_{d]b} - 2S_{b[c}T_{d]a} .
\end{equation*}
The \defn{Weyl tensor} is $W := \Rm - P \wedge g$.
Equivalently,
\begin{equation*}
 W_{abcd} := R_{abcd} - 2P_{a[c}g_{d]b} + 2P_{b[c}g_{d]a} .
\end{equation*}
The Weyl tensor is conformally invariant: $W^{e^{2u}g} = e^{2u}W^g$.
It vanishes when $n=3$.
If $n \geq 4$, then $(M^n,g)$ is locally conformally flat if and only if $W^g=0$.

Fix nonnegative integers $k,n$ and denote by
\begin{equation*}
    \delta_{b_1 \dotsm b_k}^{a_1 \dotsm a_k} := \delta_{[b_1}^{[a_1} \dotsm \delta_{b_k]}^{a_k]}
\end{equation*}
the identity map on $\Lambda^kT^\ast M$, where $M$ is an $n$-manifold.
Direct calculation yields
\begin{equation}
    \label{eqn:contract-generalized-kronecker}
    \delta_{b_1 \dotsm b_k}^{a_1 \dotsm a_k}\delta_{a_k}^{b_k} = \frac{n-k+1}{k}\delta_{b_1 \dotsm b_{k-1}}^{a_1 \dotsm a_{k-1}}
\end{equation}
on $n$-manifolds.
Given a nonnegative integer $\ell$, define $\Pf_\ell$ on $(2,2)$-tensors $T$ by
\begin{equation*}
    \Pf_\ell(T) := 2^{-\ell}(2\ell-1)!!\delta_{b_1 \dotsm b_{2\ell}}^{a_1 \dotsm a_{2\ell}}T_{a_1a_2}^{b_1b_2} \dotsm T_{a_{2\ell-1}a_{2\ell}}^{b_{2\ell-1}b_{2\ell}} ,
\end{equation*}
with the convention $\Pf_0(T) := 1$.
Here $(2\ell-1)!! := (1)(3)\dotsm(2\ell-1)$, with the convention $(-1)!! := 1$.
The \defn{Pfaffian} of an even-dimensional pseudo-Riemannian manifold $(M^n,g)$ is $\Pf^g := \Pf_{n/2}(\Rm^g)$, where $\Rm_{ab}^{cd} := R_{ab}{}^{cd}$.
Regard $\Pf_\ell$ as a multilinear map via polarization and denote $(g \wedge g)_{ab}^{cd} := (g \wedge g)_{ab}{}^{cd}$.
Equation~\eqref{eqn:contract-generalized-kronecker} implies that if $n$ is even, then
\begin{equation}
    \label{eqn:Pf-with-g}
    \Pf_{n/2}\Bigl( T^{\otimes s} \otimes (g \wedge g)^{\otimes(n/2-s)} \Bigr) = 2^{n/2-s}\binom{n/2}{s}^{-1}(n-2s-1)!!\Pf_s(T) .
\end{equation}

\subsection{Pseudo-Riemannian submanifolds}
\label{subsec:bg/submanifold}

In this subsection we discuss the geometry of submanifolds of pseudo-Riemannian manifolds.
Our main goals are to characterize their local invariants and their one-parameter families.

A \defn{nondegenerate submanifold} is a smooth immersion $j \colon \Sigma^k \to (M^n,g)$ into a pseudo-Riemannian manifold $(M,g)$ such that $k<n$ and $j^\ast g$ defines a pseudo-Riemannian metric.
We do not require that $j$ is injective.
Note that $j^\ast g$ is automatically Riemannian if $g$ is Riemannian.

Let $j \colon \Sigma^k \to (M^n,g)$ be a nondegenerate submanifold and let $\pi \colon E \to M$ be a vector bundle.
We denote by
\begin{equation*}
    j^{-1}E := \left\{ (p,X) \suchthat p \in \Sigma , X \in E_{j(p)} \right\}
\end{equation*}
the \defn{pullback bundle} $\pi \colon j^{-1}E \to \Sigma$ with its canonical smooth structure.
Since $j$ is an immersion, the map $j_\ast \colon T\Sigma \to j^{-1}TM$ defined by
\begin{equation*}
    j_\ast(X_p) := \bigl( p, dj_p(X_p) \bigr)
\end{equation*}
is an injective bundle morphism.
We abuse notation and identify $T\Sigma \cong j_\ast(T\Sigma)$, with the distinction clear by context.
The \defn{normal bundle} is the unique subbundle $N\Sigma \subset j^{-1}TM$ of rank $n-k$ over $\Sigma$ that is $g$-orthogonal to $T\Sigma$.
Hence
\begin{equation*}
    j^{-1}TM = T\Sigma \oplus N\Sigma .
\end{equation*}
Note that $N\Sigma$ depends only on the conformal class $[g]$.
The fibers of $T\Sigma$ and $N\Sigma$ over $p \in \Sigma$ are denoted $T_p\Sigma$ and $N_p\Sigma$, respectively.

Pick local coordinates $(x^\alpha)_{\alpha=1}^k$ for $\Sigma$ and a local frame $(e_{\alpha'})_{\alpha'=k+1}^{n}$ for $N\Sigma$, defined on a common open set $U \subset \Sigma$.
By shrinking $U$ if necessary, we may assume that $j\rvert_U \colon U \to M$ is an embedding.
Define $(z^a)_{a=1}^n = (x^\alpha,u^{\alpha'})$ by
\begin{equation*}
    (x^\alpha,u^{\alpha'}) \mapsto \exp_{j(x^\alpha)}^g \left( \sum_{\alpha'=k+1}^{n} u^{\alpha'}e_{\alpha'} \right) .
\end{equation*}
The Tubular Neighborhood Theorem~\cite{LeeIRM}*{Theorem~5.25} implies that these define a coordinate system, called \defn{Fermi coordinates}, on a neighborhood of $j(U) \subset M$.

Define $j^\ast \colon j^{-1}T^\ast M \to T^\ast\Sigma$ by
\begin{equation*}
    j^\ast(p,\alpha_{j(p)})(X_p) := \alpha_{j(p)}\bigl(dj_p(X_p)\bigr) .
\end{equation*}
The \defn{conormal bundle} is
\begin{equation*}
    N^\ast\Sigma := \ker \left( j^\ast \colon j^{-1}T^\ast M \to T^\ast\Sigma \right) .
\end{equation*}
It is clear that $N^\ast\Sigma$ annihilates $T\Sigma$ via the canonical pairing of $j^{-1}T^\ast M$ and $j^{-1}TM$.
We abuse notation and denote by $T^\ast\Sigma$ the $g$-orthogonal complement of $N^\ast\Sigma \subset j^{-1}T^\ast M$.
Hence
\begin{equation*}
    j^{-1}T^\ast M = T^\ast\Sigma \oplus N^\ast\Sigma .
\end{equation*}
This splitting depends only on the conformal class $[g]$.

The \defn{second fundamental form} is the section $L^g$ of $S^2T^\ast\Sigma \otimes N^\ast\Sigma$ defined by
\begin{equation*}
 L^g(U_p,V_p,\xi_p) = g\bigl( \nabla^g_{U_p} V_p , \xi_p \bigr)
\end{equation*}
for all $U_p , V_p \in T_p\Sigma$ and all $\xi_p \in N_p\Sigma$.
The \defn{mean curvature} is the section $H$ of $N^\ast\Sigma$ determined by $H^g := \frac{1}{k}\tr_{j^\ast g}L^g$.
A nondegenerate submanifold is \defn{minimal} if its mean curvature is zero.

Fix integers $0 \leq k < n$ and $r,s \geq 0$.
A \defn{natural submanifold tensor} of bi-rank $(r,s)$ on $k$-submanifolds of $n$-manifolds is an assignment $T$ to each nondegenerate submanifold $j \colon \Sigma^k \to (M^n,g)$ of a section $T^{j,g}$ of $(T^\ast\Sigma)^{\otimes r} \otimes (N^\ast\Sigma)^{\otimes s}$
that can be universally expressed as an $\mathbb{R}$-linear combination of partial contractions of tensors
\begin{equation}
    \label{eqn:tensor-to-be-contracted}
    \pi_1(\nabla^{I_1} \Rm)\otimes \cdots \otimes \pi_p(\nabla^{I_p} \Rm)\otimes (\overline{\nabla}{}^{J_1} L)\cdots \otimes (\overline{\nabla}{}^{J_q} L)\otimes\pi(g^{\otimes K}) .
\end{equation}
Here $I_s$, $J_t$, and $K$ denote powers, all factors in~\eqref{eqn:tensor-to-be-contracted} are regarded as covariant, $\pi$ and $\pi_s$ denote projection to either $T^\ast\Sigma$ or $N^\ast\Sigma$ in each index, and $\nabla$ and $\overline{\nabla}$ denote the Levi-Civita connection of $g$ and the induced connections on $T^\ast\Sigma$ and $N^\ast\Sigma$, respectively.
All contractions are performed using the projections of $g^{-1}$ to $S^2T\Sigma$ and $S^2N\Sigma$, as appropriate.
A \defn{natural submanifold scalar} is a natural submanifold tensor of bi-rank $(0,0)$.
A natural submanifold tensor $T$ has \defn{homogeneity} $w \in \mathbb{R}$ if $T^{j,c^2g} = c^{w}T^{j,g}$ for all $c>0$.
For example, as covariant tensors, $L$ and the various projections of $\Rm$ have homogeneity $2$.

Let $j \colon \Sigma \to M$ and $\widehat{\jmath} \colon \hSigma \to \hM$ be smooth maps.
Suppose that $\Psi \colon \Sigma \to \hSigma$ and $\Phi \colon M \to \hM$ are smooth maps such that $\widehat{\jmath} \mathop{\circ} \Psi = \Phi \mathop{\circ} j$.
Then
\begin{equation*}
    (\Psi,\Phi)^\ast\bigl( \Psi(p) , \omega_{\widehat{\jmath}(\Psi(p))} \bigr) := \bigl(p , \Phi_{j(p)}^\ast\omega_{\Phi(j(p))} \bigr)
\end{equation*}
defines a vector bundle homomorphism $(\Psi,\Phi)^\ast \colon \widehat{\jmath}\,^{-1}T^\ast\hM \to j^{-1} T^\ast M$.
Given an integer $r \in \mathbb{N}_0 := \{ 0, 1, 2, \dotsc \}$, extend this to a vector bundle homomorphism
\begin{equation}
    \label{eqn:submanifold-pullback}
    (\Psi,\Phi)^\ast \colon \bigl( \widehat{\jmath}\,^{-1}T^\ast\hM \bigr)^{\otimes r} \to \bigl( j^{-1}T^\ast M \bigr)^{\otimes r}
\end{equation}
by acting factor-wise.
This $(\Psi,\Phi)^\ast$ is the pullback associated to the commutative diagram
\begin{equation*}
    \begin{tikzcd}
        \Sigma \arrow[r,"\Psi"] \arrow[d,"j"] & \widehat{\Sigma} \arrow[d,"\widehat{\jmath}"] \\
        M \arrow[r,"\Phi"] & \widehat{M}.
    \end{tikzcd}
\end{equation*}
These pullbacks allow us to relate our notion of natural submanifold tensors to the usual definition in terms of coordinate charts:

\begin{proposition}
    \label{prop:natural}
    Fix integers $r,s \geq 0$.
    Let $T$ be an assignment to each nondegenerate submanifold $j \colon \Sigma^k \to (M^n,g)$ of a section $T^{j,g}$ of $(T^\ast\Sigma)^{\otimes r} \otimes (N^\ast\Sigma)^{\otimes s}$.
    Then $T$ is a natural submanifold tensor if and only if the following two conditions hold:
    \begin{enumerate}
        \item There are polynomials $\mathcal{P}_{AB'}$ such that if $j \colon \Sigma^k \to (M^n,g)$ is a nondegenerate submanifold and $(z^a)=(x^\alpha,u^{\alpha'})$ are Fermi coordinates around $j(p) \in M$, then
        \begin{equation*}
            T^{j,g}(p) = \mathcal{P}_{AB'}\bigl( h^{\alpha\beta}(p) , h^{\alpha'\beta'}(p) , \partial_{a_1\dotsm a_k}^kg_{bc}(j(p)) \bigr) \, dx^A \, du^{B'} ,
        \end{equation*}
        where $h^{\alpha\beta}$ and $h^{\alpha'\beta'}$ denote the components of the induced metrics on $T^\ast\Sigma$ and $N^\ast\Sigma$, respectively, $A \in \{ 1, \dotsc, k \}^r$ and $B' \in \{ k+1 , \dotsc , n \}^{s}$ are multi-indices, and $dx^A := dx^{\alpha_1} \dotsm dx^{\alpha_r}$ and $du^{B'} := du^{\beta_1'} \dotsm du^{\beta_s'}$.
        \item If $j \colon \Sigma^k \to (M^n,g)$ is a nondegenerate submanifold and if $\Psi \colon \Sigma \to \widehat{\Sigma}$ and $\Phi \colon M \to \widehat{M}$ are diffeomorphisms, then
        \begin{equation*}
            (\Psi,\Phi)^\ast\bigl( T^{\hj , (\Phi^{-1})^\ast g} \bigr) = T^{j,g}
        \end{equation*}
        for $\hj := \Phi \circ j \circ \Psi^{-1} \colon \hSigma \to \hM$.
    \end{enumerate}
\end{proposition}

\begin{proof}
    Suppose first that $T$ is a natural submanifold tensor.
    Let $j \colon \Sigma^k \to (M^n,g)$ be a nondegenerate submanifold and let $p \in \Sigma$.
    Pick Fermi coordinates $(x^\alpha,u^{\alpha'})$ around $j(p)$.
    Then $j(x) = (x,0)$, and hence the components of $j^\ast g$ are $h_{\alpha\beta} := g_{\alpha\beta} \circ j$.
    Moreover, $g_{\alpha\alpha'}\circ j = 0$, and $h_{\alpha'\beta'} := g_{\alpha'\beta'} \circ j$ is the induced metric on $N\Sigma$.
    Direct computation implies that the $g$-orthogonal projection onto $T^\ast\Sigma$ is
    \begin{equation*}
        \Pi_{T^\ast\Sigma}(\omega_a \, dz^a) := \omega_{\alpha} \, dx^{\alpha} .
    \end{equation*}
    The $g$-orthogonal projection onto $N^\ast\Sigma$ is $\Pi_{N^\ast\Sigma} := \Id - \Pi_{T^\ast\Sigma}$.
    Since the induced connections on $T^\ast\Sigma$ and $N^\ast\Sigma$ are obtained from the Levi-Civita connection of $g$ and projection, we deduce from the standard coordinate formulas for the Levi-Civita connection and Riemann curvature tensor of $g$ that $T^{j,g}$ satisfies Property~(1).
    It is straightforward to check that if $\Psi \colon \Sigma \to \hSigma$ and $\Phi \colon M \to \hM$ are diffeomorphisms, then the map
    \begin{equation*}
        (\Psi , \Phi)_\ast \colon j^{-1}TM \to \hj^{-1}T\hM 
    \end{equation*}
    defined by
    \begin{equation*}
        (\Psi , \Phi)_\ast\bigl( p, X_{j(p)} \bigr) := \bigl( \Psi(p) , d\Phi_{j(p)}(X_{j(p)}) \bigr)
    \end{equation*}
    is a vector bundle isomorphism and, moreover, that $(\Psi , \Phi)_\ast(T\Sigma) = T\hSigma$.
    Combining this with the naturality of the Levi-Civita connection implies that $T$ satisfies Property~(2).

    Suppose next that $T$ satisfies Properties~(1) and~(2).
    Let $\Sigma^k \subset M^n$ be an embedded submanifold and denote by $j \colon \Sigma \to M$ the canonical inclusion.
    Suppose that $\Phi \colon (M,g) \to (\hM,\hg)$ is an isometry.
    Set $\hSigma := \Phi(\Sigma)$ and let $\hj \colon \hSigma \to \hM$ be the canonical inclusion.
    Since $T$ satisfies Property~(2), we see that
    \begin{equation*}
        T^{j,g} = \Phi^\ast T^{\hj,\hg} .
    \end{equation*}
    Property~(1) implies that $T$ is a natural submanifold tensor~\cite{GrahamKuo2024}*{Theorem~1.3}.
\end{proof}
 
We frequently use abstract index notation to compute with natural submanifold tensors.
In this context, we use lowercase Latin letters ($a,b,c,\dotsc$) to label sections of $j^{-1} TM$ or its dual, we use lowercase Greek letters ($\alpha,\beta,\gamma,\dotsc$) to label sections of $T\Sigma$ or its dual, and we use primed lowercase Greek letters ($\alpha',\beta',\gamma',\dotsc$) to label sections of $N\Sigma$ or its dual.
For example, $L_{\alpha\beta\alpha'}$ and $H_{\alpha'}$ denote the second fundamental form and mean curvature, respectively.
We also use lowercase Greek indices, unprimed and primed, to denote projections from $j^{-1} TM$ to $T\Sigma$ or $N\Sigma$, respectively.
For example, the Gauss equation~\cite{DajczerTojeiro2019}*{Section~1.3} is
\begin{equation}
    \label{eqn:gauss}
    R_{\alpha\beta\gamma\delta} = \overline{R}_{\alpha\beta\gamma\delta} - L_{\alpha\gamma\alpha'}L_{\beta\delta}{}^{\alpha'} + L_{\alpha\delta\alpha'}L_{\beta\gamma}{}^{\alpha'} ,
\end{equation}
where $\overline{R}_{\alpha\beta\gamma\delta}$ denotes the curvature of the induced connection on $T\Sigma$.
More generally, we use bars to denote intrinsic Riemannian invariants of $(\Sigma,j^\ast g)$;
e.g.\ if $k \geq 3$, then $\overline{P}_{\alpha\beta}$ denotes the Schouten tensor of $j^\ast g$.
Our definition of the second fundamental form is such that if $u \in C^\infty(M)$, then
\begin{equation*}
    \nabla_\alpha\nabla_\beta u = \onabla_\alpha \onabla_\beta u - L_{\alpha\beta\alpha'}\nabla^{\alpha'}u .
\end{equation*}
In particular, with our convention $\Delta := -\nabla^a\nabla_a$, it holds that
\begin{equation*}
    -\nabla^\alpha \nabla_\alpha u = \overline{\Delta} u + kH^{\alpha'}\nabla_{\alpha'}u .
\end{equation*}

We conclude this subsection with a technical result that allows us to express a one-parameter family of submanifolds in terms of a section of the normal bundle.
To that end, given a nondegenerate submanifold $j \colon \Sigma^k \to (M^n,g)$, denote by $\exp^\perp \colon N\Sigma \to M$ the \defn{normal exponential map}, defined by
\begin{equation}
    \label{eqn:normal-exponential-map}
    \exp^\perp \xi := \exp_{j(p)}^g \xi
\end{equation}
for any $\xi \in N_p\Sigma$.
The Tubular Neighborhood Theorem immediately gives the correspondence between variations of embeddings and sections of the normal bundle, as has been used for previous holographic constructions of conformal submanifold tensors (e.g.\ \cites{GrahamWitten1999,GrahamReichert2020,CGK23}).
For one-parameter families of submanifolds, one can locally apply the Tubular Neighborhood Theorem, using a fixed member of the family to take the inverse of the normal exponential map.

\begin{proposition}
    \label{prop:tubular-neighborhood}
    Let $j \colon \Sigma^k \to (M^n,g)$ be a nondegenerate submanifold.
    Let $I$ be an interval containing $0$, set $\cM := M \times I$, and denote by $\Pi_2 \colon \cM \to I$ the canonical projection.
    Suppose that there is an embedding $\iota \colon \Sigma \to \cSigma^{k+1}$ and an immersion $\cjmath \colon \cSigma \to \cM$ such that $d(\Pi_2 \circ \cjmath)$ is nowhere-vanishing and
    \begin{equation*}
        \begin{tikzcd}
            \Sigma \arrow[r,"j"] \arrow[d,"\iota"] & M \arrow[d,"\iota"] \\
            \cSigma \arrow[r,"\cjmath"] & \cM
        \end{tikzcd}
    \end{equation*}
    commutes, where $\iota \colon M \to \cM$ is the inclusion $\iota(x) := (x,0)$.
    Then there are
    \begin{enumerate}
        \item an open set $\cU \subset \Sigma \times I$ containing $\Sigma \times \{ 0 \}$,
        \item a smooth map $\Phi \colon \cU \to \cSigma$ that is a diffeomorphism onto its image and satisfies $\Phi(x,0) = \iota(x)$ for all $x \in \Sigma$, and
        \item a smooth map $\xi \colon \cU \to N\Sigma$ such that $\pi \circ \xi = \Pi_1$ and
        \begin{equation*}
            ( \cjmath \circ \Phi )(x,\rho) = \bigl( \exp^\perp \xi(x,\rho) , \rho \bigr)
        \end{equation*}
        for all $(x,\rho) \in \cU$, where $\Pi_1 \colon \cU \to \Sigma$ is the canonical projection.
    \end{enumerate}
\end{proposition}

\begin{proof}
 Set $\crho := \Pi_2 \circ \cjmath$.
 Pick an auxiliary Riemannian metric $h$ on $\cSigma$.
 On the one hand, the assumption that $\cjmath \circ \iota = \iota \circ j$ implies that $\iota(\Sigma) \subset \crho^{-1}(\{0\})$.
 On the other hand, the assumption that $d\crho$ is nowhere-vanishing implies that the vector field
 \begin{equation*}
  \widetilde{X} := \frac{1}{h(d\crho,d\crho)}(d\crho)^{\sharp}
 \end{equation*}
 is globally defined on $\cSigma$, where $(d\crho)^{\sharp}$ is the vector field on $\cSigma$ dual to $d\crho$ with respect to $h$.
 Observe that if $\cgamma$ is an integral curve of $\widetilde{X}$, then
 \begin{equation*}
  \frac{d}{dt} d\crho\bigl(\cgamma(t)\bigr) = d\crho\bigl(\widetilde{X}_{\cgamma(t)}\bigr) = 1 .
 \end{equation*}
 Applying the Flowout Theorem~\cite{LeeISM}*{Theorem~9.20(d)} to $\widetilde{X}$ along $\iota(\Sigma)$ yields neighborhoods $\cV' \subset \Sigma \times I$ and $\cW \subset \cSigma$ of $\Sigma \times \{ 0 \}$ and $\iota(\Sigma)$, respectively, and a diffeomorphism $\Psi \colon \cV' \to \cW$, such that $(\crho \circ \Psi)(x,\rho) = \rho$ and $\Psi(x,0) = \iota(x)$ for all $x \in \Sigma$ and all $(x,\rho) \in \cV'$.
 Define $J \colon \cV' \to M$ by
 \begin{equation*}
  (\cjmath \circ \Psi)(x,\rho) = \bigl( J(x,\rho) , \rho \bigr) .
 \end{equation*}
 Then $J(x,0) = j(x)$ for all $x \in \Sigma$.
 
 Since $j \colon \Sigma \to (M,g)$ is locally an embedding, for each $p \in \Sigma$ the Tubular Neighborhood Theorem~\cite{LeeIRM}*{Theorem~5.25} produces a neighborhood $\mN_p \subset N\Sigma$ of $0_p$ and a neighborhood $W_p \subset M$ of $j(p)$ such that each fiber of $\mN_p$ is starshaped about $0$ and the normal exponential map restricts to a diffeomorphism $\exp^\perp \rvert_{\mN_p} \colon \mN_p \to W_p$.
 By introducing an auxiliary Riemannian metric and picking balls of half radius (cf.\ \cite{LeeIRM}*{p.\ 135}), for each $p \in \Sigma$ we may pick neighborhoods $\mN_p' \subset \mN_p$ and $W_p' \subset W_p$ of $p$ and $j(p)$, respectively, such that $\exp^\perp\rv_{\mN_p'} \colon \mN_p' \to W_p'$ is a diffeomorphism, each fiber of $\mN_p'$ is starshaped about $0$, and $\mN_p' \subset \mN_{q}$ whenever $\mN_p' \cap \mN_q' \not= \emptyset$.
 By shrinking each $\mN_p'$ if necessary, we may also assume that if $J(x,\rho) \in W_p'$, then $J(x,0) \in W_p'$.
 In particular, if $J(x,\rho) \in W_p' \cap W_q'$, then $\mN_p' \cap \mN_q' \not= \emptyset$.
 
 Given $p \in \Sigma$, set
 \begin{equation*}
  \cV_p := ( J \rv_{\cV' \cap ( \pi(\mN_p') \times I )})^{-1}( W_p' ) .
 \end{equation*}
 Then $\cV_p \subset \cV'$ is a neighborhood of $\pi(\mN_p') \times \{ 0 \}$.
 Define $\zeta_p \colon \cV_p \to N\Sigma$ by
 \begin{equation*}
  \zeta_p := ( \exp^\perp \rv_{\mN_p'} )^{-1} \circ J .
 \end{equation*}
 Clearly $\zeta_p$ is smooth.
 Suppose that $(x,\rho) \in \cV_p \cap \cV_q$ for some $p,q \in \Sigma$.
 Then $J(x,\rho) \in W_p' \cap W_q'$.
 Therefore $\zeta_p(x,\rho),\zeta_q(x,\rho) \in \mN_p \cap \mN_q$ have the same image under $\exp^\perp$.
 Since $\exp^\perp\rv_{\mN_p}$ and $\exp^\perp\rv_{\mN_q}$ are injective, we deduce that $\zeta_p = \zeta_q$ on $\cV_p \cap \cV_q$.
 We may thus glue the maps $\zeta_p$ to define a smooth map $\zeta \colon \cV \to N\Sigma$ on $\cV := \bigcup \cV_p$.
 By construction,
 \begin{equation*}
  (\cjmath \circ \Psi)(x,\rho) = \bigl( \exp^\perp \zeta(x,\rho) , \rho \bigr)
 \end{equation*}
 for all $(x,\rho) \in \cV$.
 
 Finally, consider the smooth map $F := (\pi \circ \zeta) \times \Pi_2 \colon \cV \to \Sigma \times I$;
 i.e.
 \begin{equation*}
  F(x,\rho) := \bigl( (\pi \circ \zeta)(x,\rho) , \rho \bigr) .
 \end{equation*}
 Note that $F$ restricts to the identity on $\Sigma \times \{ 0 \}$.
 It readily follows that $dF \rv_{(x,0)}$ is invertible for all $x \in \Sigma$.
 Hence, by shrinking $\cV$ if necessary, we may assume that $\cV \subset \Sigma \times I$ is an open neighborhood of $\Sigma \times \{ 0 \}$ and that $F$ is a diffeomorphism onto its image.
 Set $\cU := F(\cV)$ and $\Phi := \Psi \circ F^{-1}$ and $\xi := \zeta \circ F^{-1}$.
 Then $\Phi$ and $\xi$ are the desired maps.
\end{proof}

\subsection{Conformal submanifolds}

In this subsection we discuss submanifolds of conformal manifolds.
The key objectives are to introduce two types of local invariants of such spaces, one which depends on a choice of metric for the induced conformal structure on the submanifold and one which does not, and to define some important examples of these invariants.

A \defn{conformal manifold} is a pair $(M^n,\mathfrak{c})$ of a smooth $n$-manifold and a conformal class $\mathfrak{c}$;
i.e.\ an equivalence class of pseudo-Riemannian metrics on $M$ with respect to the relation $g \sim g'$ if and only if $g' = e^{2u}g$ for some $u \in C^\infty(M)$.

A \defn{conformal submanifold} is a smooth immersion $j \colon \Sigma^k \to (M^n,\mathfrak{c})$ from a smooth manifold $\Sigma$ to a conformal manifold $(M,\mathfrak{c})$ such that $j \colon \Sigma \to (M,g)$ is a nondegenerate submanifold for some, and hence any, $g \in \mathfrak{c}$.
We denote by $j^\ast\mathfrak{c}$ the induced conformal structure on $\Sigma$;
i.e.\ $j^\ast\mathfrak{c} := [j^\ast g]$ for some, and hence any, $g \in \mathfrak{c}$.
If $h \in j^\ast\mathfrak{c}$, then locally we may choose $g \in \mathfrak{c}$ such that $h = j^\ast g$.
In this case we call $g$ a \defn{local extension} of $h$;
we call $g$ a \defn{global extension} if it is defined on all of $M$.
Note that $g$ is not uniquely determined and, unless $j$ is an embedding, $g$ may not be globally defined.

A \defn{conformal submanifold tensor} of rank $(r,s)$ on $k$-submanifolds of $n$-manifolds is a natural submanifold tensor $T$ of bi-rank $(r,s)$ for which there is a $w \in \mathbb{R}$ such that
\begin{equation*}
 T^{j,e^{2\Upsilon}g} = e^{wj^\ast\Upsilon} T^{j,g}
\end{equation*}
for all nondegenerate submanifolds $j \colon \Sigma^k \to (M^n,g)$ and all $\Upsilon \in C^\infty(M)$.
In this case we call $w$ the \defn{weight} of $T$.
A \defn{conformal submanifold scalar} is a conformal submanifold tensor of bi-rank $(0,0)$.

Fundamental examples of conformal submanifold tensors are the various projections of the restriction of the Weyl tensor of $(M,g)$ to $\Sigma$ and the trace-free part $\intl_{\alpha\beta\gamma'} := L_{\alpha\beta\gamma'} - H_{\gamma'}g_{\alpha\beta}$ of the second fundamental form.
Denote
\begin{align*}
    \intl^2_{\alpha\beta} & := \intl_{\alpha\gamma\gamma'}\intl_\beta{}^{\gamma\gamma'} , \\
    \lvert \intl \rvert^2 & := \intl_{\alpha\beta\gamma'}\intl^{\alpha\beta\gamma'} ,
\end{align*}
both of which are conformal submanifold tensors.
Two other examples of conformal submanifold tensors are the \defn{Fialkow scalar}
\begin{equation*}
    G := \frac{1}{2(k-1)}\left( \lvert\intl\rvert^2 - W_{\alpha\beta}{}^{\alpha\beta} \right) ,
\end{equation*}
defined when $k \geq 2$, and the \defn{Fialkow tensor}
\begin{equation*}
    F_{\alpha\beta} := \frac{1}{k-2} \left( \intl^2_{\alpha\beta} - W_{\alpha\gamma\beta}{}^\gamma - Gg_{\alpha\beta} \right) , 
\end{equation*}
defined when $k \geq 3$.
Note that $G = \tr_{j^\ast g}F$ when $k \geq 3$.
These are related to the pullback $W_{\alpha\beta\gamma\delta}$ to $\Sigma$ of the Weyl tensor of $g$ and the intrinsic Weyl tensor $\overline{W}_{\alpha\beta\gamma\delta}$ of $j^\ast g$ by the Gauss equation
\begin{equation}
    \label{Gausseun}
    W_{\alpha\beta\gamma\delta} = \overline{W}_{\alpha\beta\gamma\delta} - \intl_{\alpha\gamma\alpha'}\intl_{\beta\delta}{}^{\alpha'} + \intl_{\alpha\delta\alpha'}\intl_{\beta\gamma}{}^{\alpha'} - 2F_{\alpha[\gamma}g_{\delta]\beta} + 2F_{\beta[\gamma}g_{\delta]\alpha} .
\end{equation}

A natural submanifold tensor $T$ of bi-rank $(r,s)$ on $k$-submanifolds of $n$-manifolds is an \defn{extrinsic tensor invariant} if $T^{j,g_1} = T^{j,g_2}$ for every conformal submanifold $j \colon \Sigma^k \to (M^n,\kc)$ and every pair $g_1,g_2 \in \kc$ such that $j^\ast g_1 = j^\ast g_2$.
Since natural submanifold tensors are locally defined, an extrinsic tensor invariant defines an assignment $T$ to each conformal submanifold $j \colon \Sigma^k \to (M^n,\kc)$ and each metric $h \in j^\ast\kc$ of a section $T^h$ of $(T^\ast\Sigma)^{\otimes r} \otimes (N^\ast\Sigma)^{\otimes s}$ by the formula
\begin{equation*}
    T^h := T^{j,g} ,
\end{equation*}
where $g$ is a local extension of $h$.
For example, conformal submanifold tensors are extrinsic tensor invariants, but the mean curvature is not.
An \defn{extrinsic scalar invariant} is an extrinsic tensor invariant of bi-rank $(0,0)$.

A fundamental example of an extrinsic tensor invariant that is not conformally invariant is the extrinsic Schouten tensor
\begin{equation}
    \label{eqn:defn-mP}
    \mathcal{P}_{\alpha\beta} := P_{\alpha\beta} + H^{\alpha'}\mathring{L}_{\alpha\beta\alpha'} + \frac{1}{2}H^{\alpha'}H_{\alpha'}g_{\alpha\beta} .
\end{equation}
This tensor and its properties were first described by Case, Graham, Kuo, Tyrrell, and Waldron~\cite{CGKTW24}*{Lemma~4.1}, though a variant involving an intrinsic tensor was first introduced by Blitz, Gover, and Waldron~\cite{BlitzGoverWaldron2024}*{Lemma~6.1}.
Notably, when $n \geq 3$ the Gauss equation~\eqref{eqn:gauss} implies~\cite{CGKTW24}*{Equation~(4.9b)} that
\begin{equation*}
    \overline{P}_{\alpha\beta} = \mathcal{P}_{\alpha\beta} - F_{\alpha\beta} .
\end{equation*}

\section{The extrinsic ambient space}
\label{sec:ambient}

The (Fefferman--Graham) ambient space~\cite{FeffermanGraham2012} is a formally Ricci flat $(n+2)$-manifold canonically associated to a conformal $n$-manifold.
In this section, we give a direct construction of the extrinsic ambient space for submanifolds of conformal manifolds, originally due to Case, Graham, and Kuo~\cite{CGK23}*{Section~6}.
In so doing, we clarify the ambiguities of the extrinsic ambient space.

We begin with a quick review of the ambient space.
Let $(M,\mathfrak{c})$ be a conformal manifold.
Consider the \defn{metric bundle}
\begin{equation*}
    \mG := \left\{ (x,g_x) \suchthat x \in M , g \in \kc \right\} \subset S^2T^\ast M .
\end{equation*}
This is a principal $\mathbb{R}_+$-bundle with projection $\pi \colon \mG \to M$, $\pi(x,g_x) := x$, and dilations $\delta_s \colon \mG \to \mG$, $\delta_s(x,g_x) := (x,s^2g_x)$ for $s \in \bR_+ := (0,\infty)$.
Define the tautological section $\boldsymbol{g}$ of $S^2T^\ast\mG$ by
\begin{equation*}
    \boldsymbol{g}(X,Y) := g_x(\pi_\ast X , \pi_\ast Y)
\end{equation*}
for all $X,Y \in T_{(x,g_x)}\mG$.
Note that $\delta_s^\ast\boldsymbol{g} = s^2\boldsymbol{g}$ for all $s>0$.

Define dilations $\cdelta_s \colon \mG \times \bR \to \mG \times \bR$ by $\cdelta_s(z,\rho) := (\delta_s(z),\rho)$.
A \defn{pre-ambient space} $(\cmG,\cg)$ for $(M^n,\kc)$ is a $\cdelta_s$-invariant open neighborhood $\cmG \subset \mG \times \bR$ of $\mG \times \{ 0 \}$ together with a pseudo-Riemannian metric $\cg$ on $\cmG$ such that
\newcounter{ambient-counter}
\begin{enumerate}
    \item $\iota^\ast\cg = \boldsymbol{g}$, and
    \item $\cdelta_s^\ast\cg = s^2\cg$ for all $s \in \bR_+$,
    \setcounter{ambient-counter}{\value{enumi}}
\end{enumerate}
where $\iota \colon \mG \to \cmG$ is the inclusion $\iota(z) := (z,0)$.
Note that $\cdelta_s \circ \iota = \iota \circ \delta_s$ for all $s \in \mathbb{R}_+$ and that if $(\cmG,\cg)$ is a pre-ambient space, then so too is $(\widetilde{\mathcal{U}},\cg\rvert_{\widetilde{\mathcal{U}}})$ for any $\cdelta_s$-invariant neighborhood $\widetilde{\mathcal{U}} \subset \cmG$ of $\mG \times \{ 0 \}$.

Let $(\cmG,\cg)$ be a pre-ambient space for $(M^n,\kc)$.
Given a vector bundle $E \to \cmG$, we denote by $O(\rho^m)$ the space of sections $T$ of $E$ such that $\rho^{-m}T$ extends continuously to $\{ \rho = 0 \}$.
Set $O(\rho^\infty) := \bigcap_{m \in \mathbb{Z}} O(\rho^m)$.
We say that $(\cmG,\cg)$ is an \defn{ambient space} if additionally
\begin{enumerate}
    \setcounter{enumi}{\value{ambient-counter}}
    \item $\Ric(\cg) = O^+(\rho^{(n-2)/2})$ if $n \geq 4$ is even, and $\Ric(\cg) = O(\rho^\infty)$ otherwise.
\end{enumerate}
Here $O^+(\rho^m)$ is the subspace of those sections $T \in O(\rho^m)$ of $S^2T^\ast\cmG$ such that if $z = (x,g_x) \in \mG$, then there is a $\tau \in S^2T_x^\ast M$ such that $\iota_z^\ast(\rho^{-m}T) = \pi_z^\ast\tau$ and $\tr_{g_x}\tau = 0$.
Two pre-ambient spaces $(\cmG_i,\cg_i)$, $i \in \{ 1, 2 \}$, for $(M,\kc)$ are \defn{ambient equivalent} if, after shrinking $\cmG_1$ and $\cmG_2$ if necessary, there is a $\cdelta_s$-equivariant\footnote{
    A diffeomorphism $\Phi \colon \cmG_1 \to \cmG_2$ is \defn{$\cdelta_s$-equivariant} if $\Phi \circ \cdelta_s = \cdelta_s \circ \Phi$ for all $s \in \bR_+$.
}
diffeomorphism $\Phi \colon \cmG_1 \to \cmG_2$ such that
\begin{enumerate}
    \item $\Phi \circ \iota_1 = \iota_2$, where $\iota_i \colon \mG \to \cmG_i$ are the canonical inclusions, and
    \item $\Phi^\ast\cg_2 - \cg_1 \in O^+(\rho^{n/2})$ if $n$ is even, and $\Phi^\ast\cg_2 - \cg_1 \in O(\rho^\infty)$ otherwise.
\end{enumerate}
In this case we call $\Phi$ an \defn{ambient equivalence}.
A fundamental result of Fefferman and Graham~\cite{FeffermanGraham2012}*{Theorem~2.3} states that every conformal manifold admits an ambient space and, moreover, it is unique up to ambient equivalence.

We now turn to the extrinsic ambient space.
Let $j \colon \Sigma^k \to (M^n,\kc)$ be a conformal submanifold.
Denote by $\mS$ the metric bundle of $(\Sigma,j^\ast\kc)$ and define the \defn{tautological immersion} $\jmath \colon \mS \to \mG$ by
\begin{equation*}
    \jmath(p,h_p) := \bigl( j(p) , g_{j(p)} \bigr) ,
\end{equation*}
where $g$ is a local extension of $h$.
Note that $\jmath$ is well-defined and $\delta_s$-equivariant.

\begin{definition}
    An \defn{extrinsic pre-ambient space} for a conformal submanifold $j \colon \Sigma^k \to (M^n,\kc)$ is a nondegenerate $\cdelta_s$-equivariant immersion $\cjmath \colon \cmS \to (\cmG,\cg)$ such that
    \begin{enumerate}
        \item $(\cmG,\cg)$ is a pre-ambient space for $(M,\kc)$,
        \item $(\cmS,\cjmath^\ast\cg)$ is a pre-ambient space for $(\Sigma,j^\ast\kc)$, and
        \item $\cjmath \circ \iota = \iota \circ \jmath$, where $\iota$ is the appropriate canonical inclusion.
    \end{enumerate}
\end{definition}

That is, an extrinsic pre-ambient space is a nondegenerate $\cdelta_s$-equivariant immersion built from pre-ambient spaces and for which the diagram
\begin{equation*}
    \begin{tikzcd}
        \mS \arrow[r,"\jmath"] \arrow[d,"\iota"] & \mG \arrow[d,"\iota"] \\
        \cmS \arrow[r,"\widetilde{\jmath}"] & \cmG
    \end{tikzcd}
\end{equation*}
commutes.

\begin{definition}
    An \defn{extrinsic ambient space} is an extrinsic pre-ambient space $\cjmath \colon \cmS \to (\cmG,\cg)$ for a conformal submanifold $j \colon \Sigma^k \to (M^n,\kc)$ such that $(\cmG,\cg)$ is an ambient space for $(M,\kc)$ and the mean curvature vector $\cH$ of $\widetilde{\jmath}$ satisfies
    \begin{enumerate}
        \item $\cH = O(\rho^{k/2})$, if $k$ is even, and
        \item $\cH = O(\rho^\infty)$, if $k$ is odd.
    \end{enumerate}
\end{definition}

We emphasize that, because of the Gauss equations, $(\cmS,\cjmath^\ast\cg)$ may not be formally Ricci flat.
Hence $(\widetilde{S},\cjmath^\ast\widetilde{g})$ need not be an ambient space for $(\Sigma,j^\ast\kc)$.
Also, while the dimensional parity of $M$ is not encoded directly in the formal vanishing of the mean curvature, it is included in the constraint on $(\cmG,\cg)$, and hence in the notion of extrinsic ambient equivalence:

\begin{definition}
    \label{defn:extrinsic-ambient-space}
    Two extrinsic pre-ambient spaces $\widetilde{\jmath}_\ell \colon \cmS_\ell \to (\cmG_\ell , \cg_\ell )$, $\ell \in \{ 1 , 2 \}$, for a conformal submanifold $j \colon \Sigma^k \to (M^n , \kc)$ are \defn{extrinsic ambient equivalent} if, after shrinking $\cmS_\ell$ and $\cmG_\ell$ if necessary, there are $\cdelta_s$-equivariant diffeomorphisms $\Psi \colon \cmS_1 \to \cmS_2$ and $\Phi \colon \cmG_1 \to \cmG_2$ such that
    \begin{enumerate}
        \item $\Phi \colon (\cmG_1 , \cg_1) \to (\cmG_2 , \cg_2)$ is an ambient equivalence,
        \item $\Psi \circ \iota_1 = \iota_2$, where $\iota_\ell \colon \mS \to \cmS_\ell$ is the canonical inclusion, and
        \item the difference $\widetilde{D} := d(\widetilde{\jmath}_2 \circ \Psi) - d(\Phi \circ \widetilde{\jmath}_1)$ satisfies
        \begin{enumerate}
            \item $\widetilde{D} \in O^+(\rho^{k/2})$, if $k$ is even,
            \item $\widetilde{D} \in O^+(\rho^{n/2})$, if $k$ is odd and $n$ is even, and
            \item $\widetilde{D} \in O(\rho^\infty)$, if $k$ and $n$ are odd.
        \end{enumerate}
    \end{enumerate}
    We call $(\Psi,\Phi)$ an \defn{extrinsic ambient equivalence}.
\end{definition}

Here $O^+(\rho^m)$ denotes the subspace of sections $T \in O(\rho^m)$ of $T^\ast\cmS \otimes \cjmath^{-1}T\cmG$ such that $\iota^\ast(\rho^{-m}T)=0$, where $\iota^\ast$ acts only on the $T^\ast\cmS$ factor.
Note that extrinsic ambient equivalence is an equivalence relation.

We do not assume that $\cjmath_2 \circ \Psi = \Phi \circ \cjmath_1$ in Definition~\ref{defn:extrinsic-ambient-space}, but rather only that these two maps formally agree to an order depending on the parities of $k$ and $n$.
Thus an extrinsic ambient equivalence is a dilation-equivariant diagram
\begin{equation*}
    \begin{tikzcd}
        & \mS \arrow[ld,swap,"\iota_1"] \arrow[dd,near start,"\jmath"] \arrow[rd,"\iota_2"] & \\
        \cmS_1 \arrow[rr,crossing over,"\Psi" near start] \arrow[dd,swap,"\widetilde{\jmath}_1"] & & \cmS_2 \arrow[dd,"\widetilde{\jmath}_2"] \\
        & \mG \arrow[ld,swap,"\iota_1"] \arrow[rd,"\iota_2"]  & \\
        \cmG_1 \arrow[rr,"\Phi" near start] & & \cmG_2
    \end{tikzcd}
\end{equation*}
for which the front face formally commutes and all other faces commute.
Note that if $(\Psi,\Phi)$ is an extrinsic ambient equivalence, then $\Psi$ is an ambient equivalence.

The main result of this section constructs extrinsic ambient spaces:

\begin{theorem}
    \label{thm:extrinsic-ambient-space}
    Let $j \colon \Sigma^k \to (M^n,\kc)$ be a conformal submanifold.
    There is an extrinsic ambient space $\cjmath \colon \cmS \to (\cmG,\cg)$ for $j$.
    Moreover, $\cjmath$ is unique up to extrinsic ambient equivalence.
\end{theorem}

Like the construction of the ambient metric~\cite{FeffermanGraham2012}, it is illuminating to split the proof of Theorem~\ref{thm:extrinsic-ambient-space} into two parts.
First we prove the existence and uniqueness of extrinsic ambient spaces in a canonical form.
Then we prove that any extrinsic ambient space is extrinsic ambient equivalent to such an extrinsic ambient space.

Let $(\cmG,\cg)$ be a pre-ambient space for $(M^n,\kc)$.
We say that $(\cmG,\cg)$ is \defn{straight} if the infinitesimal generator $\cX$ of dilation satisfies $\cnabla\cX = \Id$.
Pick $g \in \kc$ and identify $\mG \cong \bR_+ \times M$ by $(x,t^2g_x) \cong (t,x)$.
We say that $(\cmG,\cg)$ is in \defn{normal form} with respect to $g$ if
\begin{enumerate}
    \item for each $z \in \mG$, the set $\{ \rho \in \mathbb{R} \mathrel{}:\mathrel{} (z,\rho) \in \cmG \}$ is an open interval containing $0$,
    \item the map $\rho \mapsto (z,\rho)$ is a geodesic for each $z \in \mG$, and
    \item $\cg = t^2g + 2t\,dt\,d\rho$ along $\iota(\mG)$.
\end{enumerate}
Let $(\cmG,\cg)$ be in normal form with respect to $g$.
Then $(\cmG,\cg)$ is straight if and only if there is a one-parameter family $g_\rho$ of pseudo-Riemannian metrics such that
\begin{equation*}
    \cg = 2\rho\,dt^2 + 2t\,dt\,d\rho + t^2g_\rho
\end{equation*}
and $g_0 = g$~\cite{FeffermanGraham2012}*{Lemma~3.1 and Proposition~3.4}.
This reduces the construction of the ambient space to the recursive determination of the Taylor series of $g_\rho$.

The construction of the extrinsic ambient space follows the same general strategy.
Our canonical form is as follows:

\begin{definition}
    \label{defn:extrinsic-straight-and-normal}
    An extrinsic pre-ambient space $\cjmath \colon \cmS \to (\cmG,\cg)$ for a conformal submanifold $j \colon \Sigma^k \to (M^n,\kc)$ is \defn{orthogonal} with respect to $g \in \kc$ if
    \begin{enumerate}
        \item $(\cmG,\cg)$ is straight and in normal form with respect to $g$, and
        \item there is a one-parameter family $\xi_\rho$ of sections of $N\Sigma$ such that $\xi_0=0$ and
        \begin{equation}
            \label{eqn:straight-and-normal-ambient-immersion}
            \cjmath(t,x,\rho) = \left( t , \exp^\perp \xi_\rho(x) , \rho \right) ,
        \end{equation}
        where $\exp^\perp$ is the normal exponential map~\eqref{eqn:normal-exponential-map}.
    \end{enumerate}
\end{definition}

Note that $\cjmath \circ \iota = \iota \circ \jmath$.
If $\cjmath \colon \cmS \to (\cmG,\cg)$ is orthogonal with respect to $g$, then $\cjmath^\ast\cg$ is straight, but it need not be in normal form with respect to $j^\ast g$;
see Remark~\ref{rk:straight-and-normal}.

Analogous to the Fefferman--Graham construction, the existence and uniqueness of orthogonal extrinsic ambient spaces is encoded in the Taylor series of $\xi_\rho$:

\begin{proposition}
    \label{prop:extrinsic-straight-and-normal}
    Let $j \colon \Sigma^k \to (M^n,\kc)$ be a conformal submanifold.
    Let $(\cmG,\cg)$ be a straight ambient space for $(M,\kc)$ that is in normal form with respect to $g \in \kc$.
    There is a one-parameter family $\xi_\rho$ of sections of $N\Sigma$ such that $\xi_0=0$ and Equation~\eqref{eqn:straight-and-normal-ambient-immersion} defines an extrinsic ambient space $\cjmath \colon \bR_+ \times \Sigma \times (-\varepsilon,\varepsilon) \to \cmG$ that is orthogonal with respect to $g$.
    Moreover, $\xi_\rho$ mod $O(\rho^s)$ is uniquely determined by $j$ and $g$, where
    \begin{enumerate}
        \item $s := k/2+1$ if $k$ is even;
        \item $s := n/2+1$ if $k$ is odd and $n$ is even;
        \item $s := \infty$ if $k$ and $n$ are odd.
    \end{enumerate}
\end{proposition}

\begin{proof}
    Set $\cmS := \bR_+ \times \Sigma \times (-\varepsilon,\varepsilon)$.
    Define $\iota \colon \mS \to \cmS$ by $\iota(x,t^2h_x) := (t,x,0)$ for $h := j^\ast g$.
    Let $\xi_\rho$ be a one-parameter family of sections of $N\Sigma$ such that $\xi_0=0$.
    Define $j_\rho \colon \Sigma \to M$ by $j_\rho(x) := \exp^\perp \xi_\rho(x)$ and define $\cjmath$ by Equation~\eqref{eqn:straight-and-normal-ambient-immersion}.
    Then $\cjmath(t,x,\rho) = (t,j_\rho(x),\rho)$.
    We recursively determine the Taylor series of $\xi_\rho$ at $\rho=0$ by the requirement that $\cjmath$ is asymptotically minimal.

    Pick Fermi coordinates $(z^a) = (x^\alpha,u^{\alpha'})$ near a point $j(p) \in j(\Sigma)$.
    Extend these to local coordinates $(x^A) = (t,x^\alpha,\rho)$ and $(z^A) = (t,z^a,\rho)$ on $\cmS$ and $\cmG$, respectively, with the convention $x^0=z^0=t$ and $x^\infty=z^\infty=\rho$.
    
    Denote by $\widetilde{\Gamma}_{AB}^C$ the Christoffel symbols of the Levi-Civita connection $\widetilde{\nabla}$ of $\cg$ with respect to $(z^A)$.
    Direct computation~\cite{FeffermanGraham2012}*{Equation~(3.16)} gives
    \begin{equation}
        \label{eqn:ambient-christoffel}
        \begin{split}
            \widetilde{\Gamma}_{AB}^0 & =
            \begin{pmatrix}
                0 & 0 & 0 \\
                0 & -\frac{t}{2}g'_{ab} & 0 \\
                0 & 0 & 0
            \end{pmatrix} , \\
            \widetilde{\Gamma}_{AB}^c & =
            \begin{pmatrix}
                0 & t^{-1}\delta_b^c & 0 \\
                t^{-1}\delta_a^c & \Gamma_{ab}^c & \frac{1}{2}g^{cd}g'_{ad} \\
                0 & \frac{1}{2}g^{cd}g'_{bd} & 0
            \end{pmatrix} , \\
            \widetilde{\Gamma}_{AB}^\infty & =
            \begin{pmatrix}
                0 & 0 & t^{-1} \\
                0 & -g_{ab} + \rho g'_{ab} & 0 \\
                t^{-1} & 0 & 0
            \end{pmatrix} ,
        \end{split}
    \end{equation}
    where $g_{ab}$ and $\Gamma_{ab}^c$ are the components of $g_\rho$ and the Christoffel symbols of the Levi-Civita connection of $g_\rho$, respectively, and $g'_{ab} := \partial_\rho g_{ab}$.
    It readily follows that 
    \begin{equation}
        \label{eqn:preserve-X}
        \begin{aligned}
            \cnabla_{\cjmath_\ast\partial_0} \cjmath_\ast\partial_0 & = 0 , \\
            \cnabla_{\cjmath_\ast\partial_0} \cjmath_\ast\partial_A & = t^{-1}\cjmath_\ast\partial_A , & \text{if $A \not= 0$} .
        \end{aligned}
    \end{equation}
    Therefore
    \begin{equation}
        \label{eqn:ambient-tfss-killed-by-0}
        \cL_{0A\gamma'} = 0 .
    \end{equation}

    Set $\ch := \cjmath^\ast\cg$.
    Then
    \begin{equation*}
        \widetilde{h} = 2\rho \, dt^2 + 2t \, dt \, d\rho + t^2\left( h_{\alpha\beta}\,dx^\alpha \, dx^\beta + 2h_{\alpha\infty} dx^\alpha \, d\rho + h_{\infty\infty} \, d\rho^2 \right) ,
    \end{equation*}
    where
    \begin{equation}
        \label{eqn:induced-metric-components}
        \begin{split}
            h_{\alpha\beta} & = g_{\alpha\beta} + 2j_{,(\alpha}^{\alpha'} g_{\beta)\alpha'} + j_{,\alpha}^{\alpha'}j_{,\beta}^{\beta'}g_{\alpha'\beta'} , \\ 
            h_{\alpha\infty} & = j_{,\infty}^{\beta'} (g_{\alpha\beta'} + j_{,\alpha}^{\alpha'}g_{\alpha'\beta'}) , \\
            h_{\infty\infty} & = j_{,\infty}^{\alpha'} j_{,\infty}^{\beta'} g_{\alpha'\beta'} .
        \end{split}
    \end{equation}
    In these formulas, the partial derivatives $j_{,\alpha}^{\alpha'},j_{,\infty}^{\alpha'}$ are evaluated at $x$, the components $g_{ab}$ are evaluated at $j_\rho(x)$, and $h_{\alpha\beta}$ are the components of $j_\rho^\ast g_\rho$.
    Denote by
    \begin{equation}
        \label{eqn:components-of-ambient-normal-projection}
        N\partial_{\gamma'} = \partial_{\gamma'} - u_{\gamma'}^0\cjmath_\ast\partial_0 - u_{\gamma'}^\alpha\cjmath_\ast\partial_\alpha - u_{\gamma'}^\infty\cjmath_\ast\partial_\infty 
    \end{equation}
    the normal projection of $\partial_{\gamma'}$.
    Direct computation yields
    \begin{equation}
        \label{eqn:normal-projection}
        \begin{split}
            0 & = \cg(\cjmath_\ast\partial_0,N\partial_{\gamma'}) = -2\rho u_{\gamma'}^0 - tu_{\gamma'}^\infty , \\
            0 & = \cg(\cjmath_\ast\partial_\alpha , N\partial_{\gamma'}) = t^2g_{\alpha\gamma'} + t^2j_{,\alpha}^{\alpha'}g_{\alpha'\gamma'} - t^2h_{\alpha\beta}u_{\gamma'}^\beta - t^2h_{\alpha\infty}u_{\gamma'}^\infty , \\
            0 & = \cg(\cjmath_\ast\partial_\infty , N\partial_{\gamma'}) = t^2j_{,\infty}^{\alpha'}g_{\alpha'\gamma'} - tu_{\gamma'}^0 - t^2u_{\gamma'}^\alpha h_{\alpha\infty} - t^2u_{\gamma'}^\infty h_{\infty\infty} .
        \end{split}
    \end{equation}

    We now determine the Taylor series of $\xi_\rho$.
    First, since $\xi_0=0$, there is a section $f$ of $N\Sigma$ such that $\xi_\rho = f\rho + O(\rho^2)$.
    Since $g_{\alpha\alpha'}=O(\rho)$, we see that
    \begin{equation*}
        \ch_{AB} =
        \begin{pmatrix}
            0 & 0 & t \\
            0 & t^2g_{\alpha\beta} & 0 \\
            t & 0 & t^2f^{\alpha'}f_{\alpha'}
        \end{pmatrix}
        + O(\rho) .
    \end{equation*}
    Denote by $\ch^{AB}$ the components of $\ch^{-1}$.
    It follows that
    \begin{equation}
        \label{eqn:inverse-induced-ambient-metric}
        \ch^{AB} =
        \begin{pmatrix}
            -f^{\alpha'}f_{\alpha'} & 0 & t^{-1} \\
            0 & t^{-2}g^{\alpha\beta} & 0 \\
            t^{-1} & 0 & 0
        \end{pmatrix}
        + O(\rho) .
    \end{equation}
    Combining Equations~\eqref{eqn:ambient-tfss-killed-by-0} and~\eqref{eqn:inverse-induced-ambient-metric} yields
    \begin{equation*}
        (k+2)\cH_{\gamma'} = t^{-2}g^{\alpha\beta}\cL_{\alpha\beta\gamma'} + O(\rho) .
    \end{equation*}
    It follows readily from Equation~\eqref{eqn:normal-projection} that $N\partial_{\gamma'} = \partial_{\gamma'} - tf_{\gamma'}\partial_0 + O(\rho)$.
    Combining this with Equation~\eqref{eqn:ambient-christoffel} yields
    \begin{equation}
        \label{eqn:ambient-mean-curvature-base-case}
        \cL_{\alpha\beta\gamma'} = t^2 \left( L_{\alpha\beta\gamma'} + f_{\gamma'}g_{\alpha\beta} \right) + O(\rho) .
    \end{equation}
    Therefore $\cH_{\gamma'}=O(\rho)$ if and only if $f_{\gamma'}=-H_{\gamma'}$.

    Suppose now that $\ell \geq 2$ is an integer such that $\xi_\rho^{(\ell-1)}$ has been uniquely determined modulo $O(\rho^\ell)$ by the requirement that $\cH_{\gamma'} = O(\rho^{\ell-1})$.
    Set $\xi_\rho^{(\ell)} = \xi_\rho^{(\ell-1)} + f\rho^\ell$ for some section $f$ of $N\Sigma$.
    We use the superscript ${}^{(\ell)}$ to denote quantities computed using the embedding $\cjmath^{(\ell)}$ determined by $\xi_\rho^{(\ell)}$, and omit the superscript when denoting quantities computed using the embedding $\cjmath$ determined by $\xi_\rho^{(\ell-1)}$.
    On the one hand, since $g_{\alpha\alpha'},j_{,\alpha}^{\alpha'} \in O(\rho)$, we compute that
    \begin{equation*}
        \ch_{AB}^{(\ell)} = \ch_{AB} +
        \begin{pmatrix}
            0 & 0 & 0 \\
            0 & 0 & 0 \\
            0 & 0 & 2\ell t^2 f_{\alpha'} j_{,\infty}^{\alpha'}
        \end{pmatrix} \rho^{\ell-1}
        + O(\rho^\ell) .
    \end{equation*}
    Therefore the components $\ch_{(\ell)}^{AB}$ of $(\ch^{(\ell)})^{-1}$ are given by
    \begin{equation*}
        \ch_{(\ell)}^{AB} = \ch^{AB} -
        \begin{pmatrix}
            2\ell f_{\alpha'} j_{,\infty}^{\alpha'} & 0 & 0 \\
            0 & 0 & 0 \\
            0 & 0 & 0
        \end{pmatrix} \rho^{\ell-1}
        + O(\rho^\ell) .
    \end{equation*}
    Combining this with Equation~\eqref{eqn:ambient-tfss-killed-by-0} yields
    \begin{equation*}
        (k+2)\cH_{\gamma'}^{(\ell)} = \ch^{AB}\cL_{AB\gamma'}^{(\ell)} + O(\rho^\ell) .
    \end{equation*}
    On the other hand, it readily follows from Equation~\eqref{eqn:normal-projection} that
    \begin{equation*}
        N^{(\ell)}\partial_{\gamma'} = N\partial_{\gamma'} - t\ell\rho^{\ell-1}f_{\gamma'}\partial_0 + O(\rho^\ell) .
    \end{equation*}
    Combining this with Equations~\eqref{eqn:ambient-christoffel} yields
    \begin{align*}
        \cL_{\alpha\beta\gamma'}^{(\ell)} & = \ell t^2\rho^{\ell-1} f_{\gamma'}h_{\alpha\beta} + \cL_{\alpha\beta\gamma'} + O(\rho^{\ell}) , \\
        \cL_{\infty\beta\gamma'}^{(\ell)} & = \cL_{\infty\beta\gamma'} + O(\rho^{\ell-1}) , \\
        \cL_{\infty\infty\gamma'}^{(\ell)} & = \ell(\ell-1)t^2 \rho^{\ell-2} f_{\gamma'} + \cL_{\infty\infty\gamma'} + O(\rho^{\ell-1}) .
    \end{align*}
    Since $\ch^{\alpha\infty} = O(\rho)$ and $\ch^{\infty\infty} = -2t^{-2}\rho + O(\rho^2)$, we deduce that
    \begin{equation}
        \label{eqn:ambient-mean-curvature-inductive-step}
        (k+2)\cH_{\gamma'}^{(\ell)} = \ell(k + 2 - 2\ell)\rho^{\ell-1} f_{\gamma'} + (k+2)\cH_{\gamma'} + O(\rho^{\ell}) .
    \end{equation}
    We conclude that, unless $\ell=k/2+1$, there is a unique choice of $f$ such that $\cH_{\gamma'}^{(\ell)} = O(\rho^{\ell})$.

    Finally, suppose that $\xi_\rho$ is given.
    We claim that $\widetilde{H}_{\gamma'}$ mod $O(\rho^{n/2})$ is locally determined by $g_\rho$ mod $O(\rho^{n/2})$.
    If true, then the claimed dependence of $\xi_\rho$ mod $O(\rho^s)$ on $j$ and $g$ follows from Equation~\eqref{eqn:ambient-mean-curvature-inductive-step}.

    We now verify our claim.
    By Equations~\eqref{eqn:ambient-tfss-killed-by-0} and~\eqref{eqn:inverse-induced-ambient-metric}, it suffices to show that $\cL_{\alpha\beta\alpha'}$ mod $O(\rho^{n/2})$ and $\cL_{\infty\alpha\alpha'},\cL_{\infty\infty\alpha'}$ mod $O(\rho^{(n-2)/2})$ are locally determined by $g_\rho$ mod $O(\rho^{n/2})$.
    Equation~\eqref{eqn:normal-projection} implies that $u_{\gamma'}^0 = tj_{,\infty}^{\alpha'}g_{\alpha'\gamma'} + O(\rho)$, that $u_{\gamma'}^\beta,u_{\gamma'}^\infty = O(\rho)$, and that $u_{\gamma'}^\infty$ mod $O(\rho^{(n+2)/2})$ and $u_{\gamma'}^\beta,u_{\gamma'}^0$ mod $O(\rho^{n/2})$ are locally determined by $g_\rho$ mod $O(\rho^{n/2})$.
    By definition,
    \begin{equation}
        \label{eqn:christoffel-for-ambiguity}
        \begin{aligned}
            \cnabla_{\cjmath_\ast\partial_\alpha}\cjmath_\ast\partial_\beta & = \left( \cGamma_{\alpha\beta}^C + j_{,\alpha}^{\alpha'}\cGamma_{\alpha'\beta}^C + j_{,\beta}^{\alpha'}\cGamma_{\alpha\alpha'}^C + j_{,\alpha}^{\alpha'}j_{,\beta}^{\beta'}\cGamma_{\alpha'\beta'}^C \right) \partial_C + j_{,\alpha\beta}^{\alpha'}\partial_{\alpha'} , \\
            \cnabla_{\cjmath_\ast\partial_\infty}\cjmath_\ast\partial_\beta & = \left( \cGamma_{\infty\beta}^C + j_{,\infty}^{\alpha'}\cGamma_{\alpha'\beta}^C + j_{,\beta}^{\beta'}\cGamma_{\infty\beta'}^C + j_{,\infty}^{\alpha'}j_{,\beta}^{\beta'}\cGamma_{\alpha'\beta'}^C \right) \partial_C + j_{,\infty\beta}^{\alpha'}\partial_{\alpha'} , \\
            \cnabla_{\cjmath_\ast\partial_\infty}\cjmath_\ast\partial_\infty & = \left( 2j_{,\infty}^{\alpha'}\cGamma_{\alpha'\infty}^C + j_{,\infty}^{\alpha'}j_{,\infty}^{\beta'}\cGamma_{\alpha'\beta'}^C \right) \partial_C + j_{,\infty\infty}^{\alpha'}\partial_{\alpha'} .
        \end{aligned}
    \end{equation}
    Equation~\eqref{eqn:ambient-christoffel} implies that $\cGamma_{ab}^0,\cGamma_{a\infty}^c,\cGamma_{\infty a}^c$ mod $O(\rho^{(n-2)/2})$, and that all other Christoffel symbols $\cGamma_{AB}^C$ mod $O(\rho^{n/2})$, are locally determined by $g_\rho$ mod $O(\rho^{n/2})$.
    It follows immediately that $\cL_{\infty\alpha\alpha'},\cL_{\infty\infty\alpha'}$ mod $O(\rho^{(n-2)/2})$ are locally determined by $g_\rho$ mod $O(\rho^{n/2})$.
    Consider finally
    \begin{equation*}
        \cL_{\alpha\beta\alpha'} = \cg\bigl( \cnabla_{\cjmath_\ast\partial_\alpha} \cjmath_\ast\partial_\beta , N\partial_{\alpha'} \bigr) .
    \end{equation*}
    Since $\cg_{00},\cg_{0\gamma'} = O(\rho)$, the discussion above implies that $\cL_{\alpha\beta\alpha'}$ mod $O(\rho^{n/2})$ is locally determined by $g_\rho$ mod $O(\rho^{n/2})$.
    This verifies our claim.
\end{proof}

\begin{remark}
    \label{rk:straight-and-normal}
    Equation~\eqref{eqn:preserve-X} implies that if $\cjmath \colon \cmS \to (\cmG,\cg)$ is an orthogonal extrinsic ambient space, then $(\cmS,\cjmath^\ast\cg)$ is straight.
    Since the components $h_{\alpha\infty}$ and $h_{\infty\infty}$ need not vanish, $(\cmS,\cjmath^\ast\cg)$ need not be in normal form (cf.\ \cite{FeffermanGraham2012}*{Lemma~3.1}).
\end{remark}

The last statement of Proposition~\ref{prop:extrinsic-straight-and-normal} allows us to prove the uniqueness of extrinsic ambient spaces:

\begin{proposition}
    \label{prop:extrinsic-ambient-uniqueness}
    Suppose that $\cjmath \colon \cmS \to (\cmG,\cg)$ is an extrinsic ambient space for a conformal submanifold $j \colon \Sigma^k \to (M^n,\kc)$.
    Pick $g \in \kc$ and let $\cjmath' \colon \cmS' \to (\cmG',\cg')$ be an extrinsic ambient space for $j$ that is orthogonal with respect to $g$.
    Then $\cjmath$ and $\cjmath'$ are extrinsic ambient equivalent.
\end{proposition}

\begin{proof}
    Fefferman and Graham~\cite{FeffermanGraham2012}*{Theorem~2.3} proved that there is an ambient equivalence $\Phi \colon (\cmG,\cg) \to (\cmG',\cg')$.
    Proposition~\ref{prop:extrinsic-straight-and-normal} implies that $(\Id,\Phi)$ is an extrinsic ambient equivalence.
    Hence we may assume that $(\cmG,\cg) = (\cmG',\cg')$.
    
    Use $g$ to identify $\mS \cong \bR_+ \times \Sigma$ via $(x,t^2j_x^\ast g) \cong (t,x)$.
    Then $\jmath \colon \mS \to \mG$ is given by $\jmath(t,x) = \bigl(t,j(x)\bigr)$.
    Set $\cSigma := (t \circ \cjmath)^{-1}(\{1\}) \subset \cmS$ and $\cM := t^{-1}(\{ 1 \}) \subset \cmG$.

    Let $\cX$ be the infinitesimal generator of dilations.
    Set $\crho := \rho \circ \cjmath \colon \cmS \to \bR$.
    Since $\cjmath$ is $\cdelta_s$-equivariant, we see that if $\widetilde{Y} \in T_z\cmS$, then
    \begin{equation*}
        \cjmath^\ast\cg(\cX_z,\widetilde{Y}) = \cg\bigl(\cX_{\cjmath(z)},\cjmath_\ast\widetilde{Y}\bigr) = d(t^2\rho)\bigl( \cjmath_\ast\widetilde{Y} \bigr) = d\cjmath^\ast(t^2\rho)(\widetilde{Y}) = d(t^2\crho)(\widetilde{Y}) .
    \end{equation*}
    The nondegeneracy of $\cjmath$ then implies that $d\crho$ is nowhere vanishing along $\iota(\mS)$.
    Using $\cdelta_s$-equivariance and applying Proposition~\ref{prop:tubular-neighborhood} to the restriction $\cjmath\rv_{\cSigma} \colon \cSigma \to \cM$ yields a $\cdelta_s$-invariant neighborhood $\cmS'' \subset \bR_+ \times \Sigma \times (-\varepsilon,\varepsilon)$, a $\cdelta_s$-equivariant diffeomorphism $\Psi \colon \cmS'' \to \cmS$ such that $\Psi(t,x,0) = \iota(t,x)$ for all $(t,x) \in \mS$, and a one-parameter family $\xi_\rho$ of sections of $N\Sigma$ such that
    \begin{equation*}
        (\cjmath \circ \Psi)(t,x,\rho) = \bigl( t , \exp^\perp \xi_\rho(x) , \rho \bigr) .
    \end{equation*}
    By shrinking $\cmS'$ and $\cmS''$ if necessary, we may assume that $\cmS'=\cmS''$.
    We conclude from Proposition~\ref{prop:extrinsic-straight-and-normal} that $(\Psi,\Id)$ is an extrinsic ambient equivalence.
\end{proof}

The proof of the main result of this section is now straightforward:

\begin{proof}[Proof of Theorem~\ref{thm:extrinsic-ambient-space}]
    Proposition~\ref{prop:extrinsic-straight-and-normal} establishes the existence of an orthogonal extrinsic ambient space.
    Proposition~\ref{prop:extrinsic-ambient-uniqueness} establishes its uniqueness.
\end{proof}

Case, Graham, and Kuo~\cite{CGK23}*{Section~6} carried out a careful study of extrinsic ambient spaces for minimal submanifolds of Einstein manifolds.
One of their results, which we require for our study of straightenable natural submanifold tensors, is the existence of a canonical extrinsic ambient space for such submanifolds:

\begin{lemma}
    \label{lem:minimal-in-einstein-ambient-space}
    Let $j \colon \Sigma^k \to (M^n,g)$ be a minimal submanifold of an Einstein manifold with $\Ric = (n-1)\lambda g$.
    Define $\widetilde{\jmath} \colon \widetilde{\mathcal{S}} \to \widetilde{\mathcal{G}}$ by
    \begin{align*}
        \widetilde{\mathcal{G}} & := (0,\infty)_t \times M \times (-\varepsilon,\varepsilon)_\rho , \\
        \widetilde{\mathcal{S}} & := (0,\infty)_t \times \Sigma \times (-\varepsilon,\varepsilon)_\rho , \\
        \widetilde{\jmath}(t,x,\rho) & = \bigl( t,j(x),\rho\bigr) ,
    \end{align*}
    for some $\varepsilon>0$ sufficiently small.
    Set
    \begin{align*}
        \widetilde{g} & := 2\rho \, dt^2 + 2t \, dt \, d\rho + \tau^2 g , \\
        \tau & := t(1+\lambda\rho/2) .
    \end{align*}
    Then $\widetilde{\jmath} \colon \widetilde{\mathcal{S}} \to (\widetilde{\mathcal{G}},\widetilde{g})$ is an extrinsic ambient space for $j \colon \Sigma \to (M,[g])$ for which $\widetilde{\Ric}=0$ and $\widetilde{H}=0$.
\end{lemma}

\begin{proof}
    Fefferman and Graham~\cite{FeffermanGraham2012}*{p.\ 67} showed that $(\widetilde{\mathcal{G}},\widetilde{g})$ is Ricci flat.
    Direct computation~\cite{CGK23}*{Equation~(6.7)} shows that $\widetilde{\jmath}$ is minimal.
\end{proof}

The \defn{canonical extrinsic ambient space} $\cjmath \colon \cmS \to (\cmG,\cg)$ of a minimal submanifold $j \colon \Sigma^k \to (M^n,g)$ of an Einstein manifold is the one constructed by Lemma~\ref{lem:minimal-in-einstein-ambient-space}.
In this case, we denote by $\varpi \colon \cmS \to \Sigma$ and $\varpi \colon \cmG \to M$ the canonical projections, and denote by
\begin{equation*}
    \varpi^\ast := (\varpi,\varpi)^\ast \colon (j^{-1}T^\ast M)^{\otimes r} \to ( \cjmath^{-1}T^\ast\cmG)^{\otimes r}
\end{equation*}
the pullback as in Equation~\eqref{eqn:submanifold-pullback}.
The sense in which $\cjmath$ is canonical is explained by Case, Graham, and Kuo~\cite{CGK23}*{Theorem~4.10}.

\section{Conformal submanifold scalars}
\label{sec:invariants}

In this section we use the extrinsic ambient space to construct a large class of conformal submanifold scalars.
The main result of this section, which proves the second part of Theorem~\ref{thm:rough-ambient-space}, gives a sufficient condition for a natural submanifold scalar on an extrinsic ambient space to descend to a conformal submanifold scalar.
These results and our presentation parallel the treatment of scalar conformal invariants by Fefferman and Graham~\cite{FeffermanGraham2012}*{Chapters~6 and~9}.
We conclude this section with an independent construction of the obstruction field, first studied in general by Graham and Reichert~\cite{GrahamHirachi2005}, for a conformal submanifold.

The main idea in our construction is as follows:
Let $j \colon \Sigma^k \to (M^n,\kc)$ be a conformal submanifold.
Given $w \in \mathbb{R}$, denote by
\begin{equation*}
    \mE[w] := \left\{ u \in C^\infty(\mS) \suchthat \delta_s^\ast u = s^wu \right\}
\end{equation*}
the set of functions on the metric bundle $\mS$ of $(\Sigma,j^\ast\kc)$ that are homogeneous of degree $w$ with respect to dilations.
A choice of metric $h \in j^\ast\kc$ determines a section $h$ of $\pi \colon \mS \to \Sigma$ by $h(x) := (x,h_x)$.
Denote by $h^\ast \colon \mE[w] \to C^\infty(\Sigma)$ the restriction to $\mE[w]$ of the pullback by $h$.
Direct computation implies that if $\Upsilon \in C^\infty(\Sigma)$, then
\begin{equation*}
    (e^{2\Upsilon}h)^\ast = e^{w\Upsilon}h^\ast 
\end{equation*}
on $\mE[w]$.
Thus elements of $\mE[w]$ pull back via $h^\ast$ to functions that transform like conformal submanifold scalars.

Given an extrinsic ambient space $\cjmath \colon \cmS \to (\cmG,\cg)$ for $j$, denote by
\begin{equation*}
    \cmE[w] := \left\{ \widetilde{u} \in C^\infty(\cmS) \suchthat \cdelta_s^\ast\widetilde{u} = s^w\widetilde{u} \right\}
\end{equation*}
the set of functions on $\cmS$ that are homogeneous of degree $w$ with respect to dilations.
Then $\iota^\ast \colon \cmE[w] \to \mE[w]$ is a surjective linear map.
Suppose that $(\Psi,\Phi)$ is an extrinsic ambient equivalence from $\widetilde{\jmath}$ to $\cjmath' \colon \cmS' \to (\cmG',\cg')$.
If $\widetilde{u}' \in C^\infty(\cmS')$ has homogeneity $w$, then so does $\Psi^\ast\widetilde{u}'$.
Moreover,
if $\widetilde{I}$ is a natural submanifold scalar of $(k+2)$-submanifolds of $(n+2)$-manifolds, then
\begin{equation*}
    (\iota')^\ast \cI^{\cjmath',\cg'} = \iota^\ast \Psi^\ast \cI^{\cjmath',\cg'} = \iota^\ast \cI^{\Phi^{-1} \circ \cjmath' \circ \Psi,\Phi^\ast\cg'} .
\end{equation*}
It follows that the pullbacks of homogeneous natural submanifold scalars are well-defined, and hence determine conformal submanifold scalars, so long as they are independent of the ambiguities of $\widetilde{\jmath}$ and $\cg$.
In this section we give a condition on the homogeneity that guarantees this independence.

Given a nonnegative integer $r \geq 0$, denote by
\begin{equation*}
    \cL^{(r)} := \widetilde{\overline{\nabla}}{}^r\widetilde{L}
\end{equation*}
the $r$-th covariant derivative of the second fundamental form of $\cjmath$ with respect to the induced connections on $T^\ast\cmS$ and $N^\ast\cmS$.
Our first objective is to compute the components
\begin{equation*}
    \cL^{(r)}_{A\alpha'} := \widetilde{\overline{\nabla}}_{A_1} \dotsm \widetilde{\overline{\nabla}}_{A_r}\cL_{A_{r+1}A_{r+2}\alpha'}
\end{equation*}
of $\cL^{(r)}$ when at least one component of the multi-index $A \in \{ 0, 1, \dotsc, k, \infty\}^{r+2}$ is $0$.
This was done by Case, Graham, and Kuo~\cite{CGK23}*{Proposition~6.4}, though we state and prove the result needed here to avoid possible misinterpretation of the setting.
To that end, recall from Equation~\eqref{eqn:ambient-tfss-killed-by-0} that $\cL_{0A\alpha'}=0$.
The remaining cases are computed from the formula for $\cX^{A_1}\cL^{(r)}_{A_1 \dotsm A_{r+2}\alpha'}$ and differentiation.

\begin{lemma}
    \label{lem:second-fundamental-form-homogeneity}
    Let $\cjmath \colon \cmS^{k+2} \to (\cmG^{n+2},\cg)$ be an orthogonal extrinsic ambient space and let $\cX$ denote the infinitesimal generator of dilations.
    Let $r \geq 1$ be an integer and let $\ell \in \{ 1, \dotsc, r \}$.
    Then
    \begin{align}
        \label{eqn:sffh-inner} \MoveEqLeft \cX^E\cL^{(r)}_{A_1\dotsm A_rBE\alpha'} = -\sum_{i=1}^{r} \cL^{(r-1)}_{A_1 \dotsm \widehat{A_i} \dotsm A_r BA_i\alpha'} , \\
        \label{eqn:sffh-outer} \MoveEqLeft \cX^E\cL^{(r)}_{A_1 \dotsm A_{\ell-1} E A_{\ell}\dotsm A_{r-1}BC\alpha'} = -(r-\ell+1)\cL^{(r-1)}_{A_1\dotsm A_{r-1}BC\alpha'} \\
            & \quad - \sum_{i=1}^{\ell-1} \cL^{(r-1)}_{A_1 \dotsm \widehat{A_i} \dotsm A_{\ell-1} A_i A_{\ell} \dotsm A_{r-1} BC\alpha'} , \notag
    \end{align}
    where hats denote omitted indices and the empty sum equals zero.
\end{lemma}

\begin{proof}
    Equations~\eqref{eqn:preserve-X} and~\eqref{eqn:ambient-tfss-killed-by-0} imply that $\widetilde{\overline{\nabla}}\cX=\Id$ and $\cX^E\widetilde{L}_{EA\alpha'}=0$, respectively.
    Differentiating the second equation using the first yields Equation~\eqref{eqn:sffh-inner}. 
    
    Direct computation using the $\cdelta_s$-equivariance of the extrinsic ambient space and the naturality and homogeneity of the second fundamental form yields
    \begin{equation*}
        (\cdelta_s,\cdelta_s)^\ast(\cL^{(r-1)})^{\cjmath,\cg} = (\cL^{(r-1)})^{\cjmath,\cdelta_s^\ast\cg} = s^2(\cL^{(r-1)})^{\cjmath,\cg} .
    \end{equation*}
    Since the conclusion is local, we may assume that $\widetilde{\jmath}$ is an embedding.
    Pick a section $\cU$ of $(T^\ast\cmG)^{\otimes(r+2)}$ that restricts to $\cL^{(r-1)}$ on $\cjmath(\cmS)$;
    by the above computation, we may assume that $\cdelta_s^\ast\cU = s^2\cU$.
    Then the Lie derivative of $\cU$ is
    \begin{equation*}
        \mathcal{L}_{\cX}\cU = 2\cU .
    \end{equation*}
    Since $\widetilde{\nabla}\cX = \Id$, it holds that $\mathcal{L}_{\cX}\widetilde{\alpha} = \cnabla_{\cX}\widetilde{\alpha} + \widetilde{\alpha}$ for any one-form $\widetilde{\alpha}$ on $\cmG$.
    Hence
    \begin{equation*}
        \cnabla_{\cX}\cU = \mathcal{L}_{\cX}\cU - (r+2)\cU = -r\cU .
    \end{equation*}
    Projecting to $(T^\ast\cmS)^{\otimes(r+1)} \otimes N^\ast\cmS$ yields Equation~\eqref{eqn:sffh-outer} in the case $\ell=1$.
    The remaining cases follow by differentiating as in the first paragraph.
\end{proof}

The next step in our construction of conformal submanifold scalars is to find a sufficient condition on a multi-index $A \in \{ 0, 1, \dotsc, k , \infty\}^{r+2}$ for the component $\cL^{(r)}_{A\alpha'}$ to be independent of the ambiguities of an extrinsic ambient space.
The \defn{strength} of $A$ is
\begin{equation*}
    \lVert A \rVert := \#\left\{ i \suchthat A_i \in \{ 1, \dotsc, n \} \right\} + 2\#\left\{ i \suchthat A_i = \infty \right\} .
\end{equation*}
This notion, introduced by Fefferman and Graham~\cite{FeffermanGraham2012}*{Chapter~6}, provides a useful way to determine when a natural submanifold tensor is independent of the ambiguities of an extrinsic ambient space.
More precisely:

\begin{proposition}
    \label{prop:second-fundamental-form-strength}
    Let $\cjmath \colon \cmS \to (\cmG,\cg)$ be an extrinsic ambient space that is orthogonal with respect to a nondegenerate submanifold $j \colon \Sigma^k \to (M^n,g)$.
    For each multi-index $A \in \{ 0, 1, \dotsc, k , \infty\}^{r+2}$, $r \geq 0$, the component $\cL^{(r)}_{A\alpha'}$ mod $O(\rho^{(k+2-\lVert A \rVert)/2})$ depends only on $j_\rho$ mod $O(\rho^{(k+2)/2})$ and on $g_\rho$ mod $O(\rho^{n/2})$.
\end{proposition}

\begin{proof}
    The proof is by induction in $r$.
    For brevity, we say that an equivalence class $T$ mod $O(\rho^s)$ is \emph{independent of the ambiguities of $j_\rho$ and $g_\rho$} if it depends only on $j_\rho$ mod $O(\rho^{(k+2)/2})$ and $g_\rho$ mod $O(\rho^{n/2})$.

    Consider the base case $r=0$.
    Since $\cL_{0A\alpha'}=0$, it suffices to assume that $0 \not\in A$.
    Write the normal projection $N$
    as in Equation~\eqref{eqn:components-of-ambient-normal-projection}.
    It follows from Equation~\eqref{eqn:normal-projection} that $u_{\gamma'}^0 = tj_{,\infty}^{\alpha'}g_{\alpha'\gamma'} + O(\rho)$, that $u_{\gamma'}^\beta,u_{\gamma'}^\infty = O(\rho)$, and that
    \begin{equation}
        \label{eqn:normal-projection-dependencies}
        \begin{split}
        u_{\gamma'}^\infty & \mod O(\rho^{(k+2)/2}) , \\
        u_{\gamma'}^\beta & \mod O(\rho^{(k+2)/2}) \cap O(\rho^{n/2}) , \\
        u_{\gamma'}^0 & \mod O(\rho^{k/2}) ,
        \end{split}
    \end{equation}
    are independent of the ambiguities of $j_\rho$ and $g_\rho$.
    Equations~\eqref{eqn:ambient-christoffel} imply that $\cGamma_{ab}^0,\cGamma_{a\infty}^c,\cGamma_{\infty a}^c$ mod $O(\rho^{(n-2)/2})$, and all other Christoffel symbols mod $O(\rho^{n/2})$, are independent of the ambiguities of $j_\rho$ and $g_\rho$.
    Write
    \begin{align*}
        \cL_{\alpha\beta\alpha'} & = \cg\bigl( \cnabla_{\cjmath_\ast\partial_\alpha} \cjmath_\ast\partial_\beta , N\partial_{\alpha'} \bigr) , \\
        \cL_{\infty\alpha\alpha'} & = \cg\bigl( \cnabla_{\cjmath_\ast\partial_\infty} \cjmath_\ast\partial_\alpha , N\partial_{\alpha'} \bigr) , \\
        \cL_{\infty\infty\alpha'} & = \cg\bigl( \cnabla_{\cjmath_\ast\partial_\infty} \cjmath_\ast\partial_\infty , N\partial_{\alpha'} \bigr) .
    \end{align*}
    Since $\cg_{00},u_{\gamma'}^\infty = O(\rho)$, Equation~\eqref{eqn:christoffel-for-ambiguity} implies that $\cL_{\alpha\beta\alpha'}$ mod $O(\rho^{k/2})$, $\cL_{\infty\alpha\alpha'}$ mod $O(\rho^{k/2}) \cap O(\rho^{(n-2)/2})$, and $\cL_{\infty\infty\alpha'}$ mod $O(\rho^{(k-2)/2})$ are independent of the ambiguities of $j_\rho$ and $g_\rho$.
    This establishes the base case.

    Suppose that $r \geq 0$ is such that $\cL^{(r)}_{A\alpha'}$ mod $O(\rho^{(k+2-\lVert A \rVert)/2})$ is independent of the ambiguities of $j_\rho$ and $g_\rho$ for all multi-indices $A$ of length $r+2$.
    Let $A$ be a multi-index of length $r+3$.
    Write $A = (A_1,A')$, where $A_1 \in \{ 0, 1, \dotsc, k, \infty \}$.
    
    If $A_{1}=0$, then Lemma~\ref{lem:second-fundamental-form-homogeneity} gives the required independence of $\cL^{(r+1)}_{A\alpha'}$ mod $O(\rho^{(k+2-\lVert A \rVert)/2})$ from the ambiguities of $g_\rho$ and $j_\rho$.
    
    Suppose now that $A_{1}\not=0$.
    
    If $\lVert A \rVert \leq 2$, then at most $2$ of the components $A_2,\dotsc,A_{r+3}$ are nonzero.
    Iteratively applying Lemma~\ref{lem:second-fundamental-form-homogeneity} implies that $\cL^{(r+1)}_{A\alpha'}$ mod $O(\rho^{(k+2-\lVert A \rVert)/2})$ is independent of the ambiguities of $g_\rho$ and $j_\rho$.
    
    Suppose now that $\lVert A \rVert \geq 3$.
    Then
    \begin{equation}
        \label{eqn:weight-to-strength}
        (k+2-\lVert A \rVert)/2 \leq (k-1)/2 .
    \end{equation}
    Write
    \begin{equation}
        \label{eqn:expand-via-christoffol}
        \cL^{(r+1)}_{A\alpha'} = \partial_{A_{1}}\cL^{(r)}_{A'\alpha'} - \sum_{i=2}^{r+3} \widetilde{\overline{\Gamma}}{}_{A_{1}A_i}^B\cL^{(r)}_{A_2 \dotsm B \dotsm A_{r+3}\alpha'} - \cD_{A_{1}\alpha'}^{\beta'}\cL^{(r)}_{A'\beta'} ,
    \end{equation}
    where $\widetilde{\overline{\Gamma}}{}_{BC}^E$ denotes the Christoffel symbols of the metric $\cjmath^\ast\cg$ and $\cD_{B\alpha'}^{\beta'}$ denotes the connection coefficients of the normal connection;
    i.e.
    \begin{equation*}
        \cg \bigl( \cnabla_{\cjmath_\ast\partial_B}N\partial_{\alpha'} , N\partial_{\gamma'} \bigr) = \cD_{B\alpha'}^{\beta'}\cg(N\partial_{\beta'},N\partial_{\gamma'}) .
    \end{equation*}
    
    First, the inductive hypothesis implies that $\partial_{A_{1}}\cL^{(r)}_{A'\alpha'}$ mod $O(\rho^{(k+2-\lVert A \rVert)/2})$ is independent of the ambiguities of $j_\rho$ and $g_\rho$.
    
    Second, Equations~\eqref{eqn:induced-metric-components} imply that, with the exception of $\widetilde{\overline{\Gamma}}{}_{\infty\infty}^0$, all of the Christoffel symbols $\widetilde{\overline{\Gamma}}{}_{BC}^E$ mod $O(\rho^{(k-1)/2})$ are independent of the ambiguities of $j_\rho$ and $g_\rho$;
    instead, $\widetilde{\overline{\Gamma}}{}_{\infty\infty}^0$ mod $O(\rho^{(k-2)/2})$ is independent of the ambiguities of $j_\rho$ and $g_\rho$.
    The former Christoffel symbols do not contribute to the ambiguity of $\cL^{(r+1)}_{A\alpha'}$ by Inequality~\eqref{eqn:weight-to-strength}.
    The latter Christoffel symbol only arises if $\lVert A \rVert \geq 4$, in which case $(k+2-\lVert A \rVert)/2 \leq (k-2)/2$.
    Hence $\widetilde{\overline{\Gamma}}{}_{\infty\infty}^0$ does not contribute to the ambiguity of $\cL^{(r+1)}_{A\alpha'}$.
    
    Third, it follows from Equations~\eqref{eqn:normal-projection-dependencies} and the facts $\cg_{00},\cg_{0\alpha'},u_{\gamma'}^\infty = O(\rho)$ that $\cD_{B\alpha'}^{\beta'}$ mod $O(\rho^{k/2})$ is independent of the ambiguities of $j_\rho$ and $g_\rho$, and hence does not contribute to the ambiguity of $\cL^{(r+1)}_{A\alpha'}$.
    
    By the induction hypothesis, these observations imply the desired conclusion except if there are terms in Equation~\eqref{eqn:expand-via-christoffol} for which
    \begin{equation*}
        \lVert (A_2 , \dotsc, A_{i-1} , B , A_{i+1} , \dotsc , A_{r+3}) \rVert > \lVert A \rVert .
    \end{equation*}
    This can only happen if $B=\infty$, $A_i=0$, and $A_{1}=a$.
    However, Equation~\eqref{eqn:preserve-X} implies that $\widetilde{\overline{\Gamma}}{}_{0a}^\infty = 0$, so this case does not contribute to Equation~\eqref{eqn:expand-via-christoffol}.
\end{proof}

We now prove the second statement of Theorem~\ref{thm:rough-ambient-space}, which constructs conformal submanifold scalars as pullbacks $\mathcal{I} := \iota^\ast\cI$ of natural submanifold scalars in the extrinsic ambient space.
While our statement is not optimal (cf.\ \cite{FeffermanGraham2012}*{Proposition~9.1}), it covers all homogeneities that arise in our computations of renormalized extrinsic curvature integrals (cf.\ Theorem~\ref{thm:straight}).

\begin{theorem}
    \label{thm:construction-of-scalars}
    Fix integers $2 \leq k < n$.
    Let $\cI$ be a natural submanifold scalar of homogeneity $w \geq -k$ on $(k+2)$-submanifolds of $(n+2)$-manifolds.
    For each conformal submanifold $j \colon \Sigma^k \to (M^n,\kc)$, the function
    \begin{equation*}
        \mathcal{I} := \iota^\ast \cI^{\cjmath,\cg} \in \mE[w]
    \end{equation*}
    is independent of the choice of extrinsic ambient space $\cjmath \colon \cmS \to (\cmG,\cg)$ for $j$.
    Moreover, $\mathcal{I}^h := h^\ast\mathcal{I}$ defines a conformal submanifold scalar of weight $w$ on $k$-sub\-man\-i\-folds of $n$-manifolds.
\end{theorem}

\begin{proof}
    It suffices to show that $\cI$ mod $O(\rho)$ depends only on $j_\rho$ mod $O(\rho^{(k+2)/2})$ and on $g_\rho$ mod $O(\rho^{n/2})$.
    As a natural submanifold scalar, $\cI$ can be written as a linear combination of complete contractions of
    \begin{equation*}
        \pi_1(\cnabla^{P_1}\cRm) \otimes \dotsm \otimes \pi_p(\cnabla^{P_p}\cRm) \otimes (\widetilde{\overline{\nabla}}{}^{P'_1}\cL) \otimes \dotsm \otimes (\widetilde{\overline{\nabla}}{}^{P'_q}\cL) .
    \end{equation*}
    Each summand has
    \begin{equation*}
        2K := 4p + 3q + \sum_{a=1}^p P_a + \sum_{b=1}^q P'_b
    \end{equation*}
    pairwise contracted indices.
    The homogeneities of $\cRm$, $\cnabla$, and $\cL$ imply that
    \begin{equation*}
        w = -2K + 2p + 2q .
    \end{equation*}
    Hence, the assumption $w \geq -k$ yields
    \begin{equation*}
        2K \leq k + 2p + 2q .
    \end{equation*}
    Write $\pi_a(\cnabla^{P_a}\cRm)$ in terms of the normal projection $N_A^B \colon \cjmath^{-1}T^\ast\cmG \to N^\ast\cmS$, the tangential projection $T_A^B \colon \cjmath^{-1}T^\ast\cmG \to T^\ast\cmS$, and $\cnabla^{P_a}\cRm$.
    Denote by $S_a$, $a \in \{ 1, \dotsc, p \}$, and $S_b'$, $b \in \{ 1, \dotsc, q \}$, the strengths of the factors $\cnabla^{P_a}\cRm$ and $\cL^{(P'_b)}$, respectively.
    Since $j_\ast\partial_0=\partial_0$, we see that $N_0^A=0$.
    Equations~\eqref{eqn:components-of-ambient-normal-projection} and~\eqref{eqn:normal-projection} imply that $N_{\alpha'}^\infty \equiv 0$, $N_{\alpha'}^{b} \equiv \delta_{\alpha'}^b$, and $N_{\alpha'}^0 \equiv -tj_{,\infty}^{\beta'}g_{\alpha'\beta'}$ mod $O(\rho)$.
    A similar computation shows that $N_\infty^\infty \equiv 0$, $N_\infty^a \equiv 0$, and $N_\infty^0 \equiv -tj_{,\infty}^{\alpha'}j_{,\infty}^{\beta'}g_{\alpha'\beta'}$ mod $O(\rho)$.
    In particular, $N$ is independent of the ambiguities of $j_\rho$ and $g_\rho$, and normal projections do not decrease the strength.
    The same is true of the tangential projection $T_A^B$ mod $O(\rho)$.
    Since $\cg^{AB}$ mod $O(\rho)$ is nonzero only when $\lVert AB \rVert = 2$, we deduce that
    \begin{equation*}
        \sum_{a=1}^p S_a + \sum_{b=1}^q S_b' \leq 2K .
    \end{equation*}
    The facts $\cR_{0ABC}=0$ and $\cL_{0A\alpha'}=0$ imply that $S_a \geq 4$ and $S_b' \geq 3$, respectively.
    
    Suppose first that $p \geq 1$.
    Let $a_0 \in \{ 1 , \dotsc , p \}$.
    Then
    \begin{equation*}
        S_{a_0} + 4(p-1) + 3q \leq \sum_{a=1}^p S_a + \sum_{b=1}^q S_b' \leq k + 2p + 2q .
    \end{equation*}
    Therefore $S_{a_0} \leq k + 2$.
    A result of Fefferman and Graham~\cite{FeffermanGraham2012}*{Proposition~6.2} implies that the contribution of $\cnabla^{P_{a_0}}\cRm$, and hence of $\pi(\cnabla^{P_{a_0}}\cRm)$, to the complete contraction is independent of the ambiguities of $j_\rho$ and $g_\rho$.
    
    Suppose next that $q \geq 1$.
    Since $\cL{}^{(r)}$ has only one normal component and $\cI$ is a complete contraction, we see that $p \geq 1$ or $q \geq 2$.
    Let $b_0 \in \{ 1, \dotsc, q \}$.
    Then
    \begin{equation*}
        S_{b_0}' + 4p + 3(q-1) \leq \sum_{a=1}^p S_a + \sum_{b=1}^q S_b' \leq k + 2p + 2q.
    \end{equation*}
    Therefore $S_{b_0}' \leq k+1$.
    Proposition~\ref{prop:second-fundamental-form-strength} implies that the contribution of $\cL{}^{(P'_{b_0})}$ to the complete contraction is independent of the ambiguities of $j_\rho$ and $g_\rho$.

    The above paragraphs show that $\mathcal{I}$ is independent of the ambiguities of the extrinsic ambient space.
    Since $\mathcal{I}^h$ is locally defined, we can therefore evaluate it with respect to an extrinsic ambient space $\cjmath \colon \cmS \to (\cmG,\cg)$ that is orthogonal with respect to an extension $g \in \kc$ of $h$.
    On the one hand, the inductive procedure used to prove Proposition~\ref{prop:extrinsic-straight-and-normal} implies that, in Fermi coordinates, $\mathcal{I}^h = \widetilde{I}^{\cjmath,\cg}\rv_{t=1,\rho=0}$ can be expressed as a universal polynomial in $h^{\alpha\beta}$, $h^{\alpha'\beta'}$, and $\partial^k_{a_1\dotsm a_k}g_{bc}$.
    On the other hand, if $\psi \colon \Sigma \to \Sigma'$ and $\phi \colon M \to M'$ are diffeomorphisms, then naturality implies that
    \begin{equation*}
        \cjmath'(t,x',\rho) := \bigl( t , \exp^\perp \phi_\ast\xi_\rho(\psi^{-1}(x')),\rho \bigr) = (1 \times \phi \times 1) \circ \cjmath \circ (1 \times \psi \times 1)^{-1}
    \end{equation*}
    defines an extrinsic ambient space for $\phi \circ j \circ \psi^{-1}$ that is orthogonal with respect to $(\phi^{-1})^\ast g$.
    Hence $\mathcal{I}^h = \psi^\ast I^{(\psi^{-1})^\ast h}$.
    We conclude from Proposition~\ref{prop:natural} that $\mathcal{I}$ is an extrinsic scalar invariant.
    Its conformal invariance follows from homogeneity.
\end{proof}

We conclude this section by studying the obstruction $\mathcal{H}_{\alpha'}$ to the existence of an extrinsic ambient space for $j \colon \Sigma^k \to (M^n,\kc)$ that is formally minimal to infinite order when $k$ is even.
This obstruction field was first studied by Graham and Reichert~\cite{GrahamReichert2020} via Poincar\'e spaces, analogous to the treatment of the Fefferman--Graham obstruction tensor $\mathcal{O}_{ab}$ by Graham and Hirachi~\cite{GrahamHirachi2005}.
We instead give the ambient treatment of $\mathcal{H}_{\alpha'}$, analogous to the treatment of $\mathcal{O}_{ab}$ by Fefferman and Graham~\cite{FeffermanGraham2012}.
We also compute the leading-order term of $\mathcal{H}_{\alpha'}$.

\begin{theorem}
    \label{thm:obstruction}
    Let $j \colon \Sigma^k \to (M^n,\kc)$, $k$ even, be a conformal submanifold.
    Given an extrinsic ambient space $\cjmath \colon \cmS \to (\cmG,\cg)$ for $j$, define the section $\mathcal{H}_{\alpha'}$ of $N^\ast\mS$ by
    \begin{equation*}
        \mathcal{H}_{\alpha'} := c_k\iota^\ast\bigl( \rho^{-k/2}\cH_{\alpha'} \bigr) , \qquad c_k := 2^{k/2-1}(k/2-1)!(k/2)! .
    \end{equation*}
    Then
    \begin{enumerate}
        \item $\mathcal{H}_{\alpha'}$ is independent of the choice of extrinsic ambient space and is homogeneous of degree $-k$ with respect to dilations;
        \item $\mathcal{H}_{\alpha'}^h := h^\ast\mathcal{H}_{\alpha'}$ defines a natural submanifold tensor of bi-rank $(0,1)$, and
        \begin{equation}
            \label{eqn:obstruction}
            \mathcal{H}_{\alpha'}^h = \overline{\Delta}^{k/2}H_{\alpha'} + \mathrm{lots} ,
        \end{equation}
        where $\mathrm{lots}$ denotes terms that involve at most $k-2$ derivatives of the second fundamental form;
        \item if $\Upsilon \in C^\infty(\Sigma)$, then $\mathcal{H}_{\alpha'}^{e^{2\Upsilon}h} = e^{-k\Upsilon}\mathcal{H}_{\alpha'}^h$; and
        \item if there is a $g \in \kc$ such that $j \colon \Sigma^k \to (M^n,g)$ is a minimal submanifold of an Einstein manifold, then $\mathcal{H}_{\alpha'} = 0$.
    \end{enumerate}
\end{theorem}

\begin{proof}
    Let $\cX$ be the infinitesimal generator of dilations in $\cmS$ and set $\widetilde{Q} := \cjmath^\ast\cg(\cX,\cX)$.
    Then $\widetilde{Q} = 2t^2\rho + O(\rho^{k/2+1})$ is homogeneous of degree $2$ with respect to dilations.
    Thus $\widetilde{Q}^{-k/2}\cH_{\alpha'}$ is homogeneous of degree $-k$ where defined.
    It follows from naturality and Equation~\eqref{eqn:ambient-mean-curvature-inductive-step} that $\iota^\ast(\widetilde{Q}^{-k/2}\cH_{\alpha'})$, and hence $\mathcal{H}_{\alpha'}$, is independent of the choice of extrinsic ambient space.
    We deduce from Equations~\eqref{eqn:ambient-mean-curvature-base-case} and~\eqref{eqn:ambient-mean-curvature-inductive-step} that $\mathcal{H}_{\alpha'}^h$ is a conformal submanifold tensor of bi-rank $(0,1)$ and weight $-k$.
    Therefore $\mathcal{H}_{\alpha'}^{e^{2\Upsilon}h} = e^{-k\Upsilon}\mathcal{H}_{\alpha'}^h$ for all $h \in j^\ast\kc$ and all $\Upsilon \in C^\infty(\Sigma)$.

    Next we compute the leading-order term of $\mathcal{H}_{\alpha'}$.
    To that end, we compute as in the proof of Proposition~\ref{prop:extrinsic-straight-and-normal}, but modulo terms that involve the Riemann curvature tensor of $g$ or are at least quadratic in the second fundamental form.
    Recalling Equation~\eqref{eqn:ambient-tfss-killed-by-0}, we may ignore all derivatives of $g_\rho$ and all terms at least quadratic in $j_{,\infty}^{\alpha'}$ and its derivatives in Equations~\eqref{eqn:ambient-christoffel}, \eqref{eqn:induced-metric-components}, and~\eqref{eqn:normal-projection} to deduce that
    \begin{equation}
        \label{eqn:ambient-mean-curvature-lot}
        \cH_{\alpha'} \equiv (k-2\rho\partial_\rho)\partial_\rho j_{\alpha'} + kH_{\alpha'} .
    \end{equation}
    Combining the variational formula $\frac{\partial}{\partial s}kH^{j_s} \equiv -\overline{\Delta}\partial_s j_s$ with a straightforward induction argument yields
    \begin{equation*}
        \partial_\rho^\ell \vert_{\rho=0} j_{\alpha'} \equiv \frac{(k/2-\ell)!}{2^{\ell-1}(k/2-1)!} \overline{\Delta}^{\ell-1}\partial_\rho j_{\alpha'} \equiv -\frac{(k/2-\ell)!}{2^{\ell-1}(k/2-1)!} \overline{\Delta}^{\ell-1}H_{\alpha'}
    \end{equation*}
    for all positive integers $\ell \leq k/2$.
    Applying $\partial_\rho^{k/2}\vert_{\rho=0}$ to Equation~\eqref{eqn:ambient-mean-curvature-lot} yields
    \begin{equation*}
        \partial_\rho^{k/2} \vert_{\rho=0} \cH_{\alpha'} \equiv \frac{1}{2^{k/2-1}(k/2-1)!}\overline{\Delta}^{k/2}H_{\alpha'} .
    \end{equation*}
    Equation~\eqref{eqn:obstruction} readily follows from the fact that there are no nonzero $N^\ast\Sigma$-valued partial contractions of~\eqref{eqn:tensor-to-be-contracted} of homogeneity $-k$ with a factor $\overline{\nabla}{}^{k/2-1}L$.

    Finally, the canonical extrinsic ambient space is minimal to infinite order.
    Therefore $\mathcal{H}_{\alpha'} = 0$ for conformally minimal submanifolds of an Einstein manifold.
\end{proof}

\section{Straightenable extrinsic invariants}
\label{sec:straightenable}

In this section we develop the notion of straight submanifold tensors and their associated straightenable submanifold tensors.
The main result of this section produces a large class of conformal submanifold scalars that are readily computed on minimal submanifolds of Einstein manifolds.
Our approach is analogous to that used to study invariants of conformal manifolds~\cite{CKLTY2024}.

Straight invariants are defined in terms of their behavior at canonical extrinsic ambient spaces as constructed by Lemma~\ref{lem:minimal-in-einstein-ambient-space}.

\begin{definition}
    \label{defn:stragT}
    A natural submanifold tensor $\cT$ of bi-rank $(r,s)$ and homogeneity $w \in \bR$ on $(k+2)$-submanifolds of $(n+2)$-manifolds is \defn{straight} if there is a natural submanifold tensor $T$ of bi-rank $(r,s)$ and homogeneity $w$ on $k$-submanifolds of $n$-manifolds such that
    \begin{equation*}
        \cT^{\cjmath,\cg} = \tau^w\varpi^\ast T^{j,g}
    \end{equation*}
    whenever $\cjmath \colon \cmS \to (\cmG,\cg)$ is the canonical extrinsic ambient space of a minimal submanifold $j \colon \Sigma^k \to (M^n,g)$ of an Einstein manifold.
    In this case we call $T$ a \defn{straightenable} invariant associated to $\cT$.
\end{definition}

The set of straight (resp.\ straightenable) submanifold tensors of bi-rank $(r,s)$ and homogeneity $w$ on $(k+2)$-submanifolds of $(n+2)$-manifolds (resp.\ $k$-submanifolds of $n$-manifolds) is a real vector space.
We emphasize that the properties of being straight or straightenable are defined in reference to the canonical extrinsic ambient space, and hence do not uniquely determine the natural submanifold tensor itself.
For example, every element of the differential ideal of submanifold tensors generated by the ambient Ricci tensor and the ambient mean curvature is a straight tensor to which the zero tensor field is associated.

We produce many examples of straight submanifold tensors via two constructions.
These constructions both begin with two fundamental straight invariants:

\begin{lemma}
    \label{lem:sff-riem-straight}
    The second fundamental form $\cL$ and projections of the Riemann curvature tensor $\widetilde{\Rm}$ are straight submanifold tensors of homogeneity $2$.
    Moreover, the trace-free part $\intl$ of the second fundamental form and projections of the Weyl tensor, respectively, are associated straightenable submanifold tensors.
\end{lemma}

\begin{proof}
    Let $\cjmath \colon \cmS \to (\cmG,\cg)$ be the canonical extrinsic ambient space of a minimal submanifold $j \colon \Sigma^k \to (M^n,g)$ of an Einstein manifold.
    We compute in Fermi coordinates as in Section~\ref{sec:ambient}.
    Since $\cjmath(t,x,\rho) = (t,j(x),\rho)$, we see that $\partial_0,\partial_\alpha,\partial_\infty$ are sections of $T\cmS$.
    Set $h := j^\ast g$.
    Equation~\eqref{eqn:induced-metric-components} implies that
    \begin{equation}
        \label{eqn:pullback-metric-is-straightenable}
        \cjmath^\ast\cg = 2\rho \, dt^2 + 2t \, dt \, d\rho + \tau^2\varpi^\ast h .
    \end{equation}
    Equations~\eqref{eqn:normal-projection} then imply that $N\partial_{\alpha'}=\partial_{\alpha'}$.
    On the one hand, the fact that the Weyl tensor is straightenable and associated to the ambient Riemann curvature tensor~\cite{CKLTY2024}*{Lemma~3.4} yields our claims about projections of $\widetilde{\Rm}$ and $W$.
    On the other hand, Equations~\eqref{eqn:christoffel-for-ambiguity} readily yield
    \begin{equation*}
        \cL^{\cjmath,\cg} = \tau^2\varpi^\ast \intl^{j,g} . \qedhere
    \end{equation*}
\end{proof}

Our first construction of straight submanifold tensors is via tensor products and contractions.
Explaining this requires two pieces of terminology.

Suppose that $\widetilde{T}_i$ (resp.\ $T_i$), $i \in \{ 1, 2 \}$, are natural submanifold tensors of bi-rank $(r_i,s_i)$ on $(k+2)$-submanifolds of $(n+2)$-manifolds (resp.\ $k$-submanifolds of $n$-manifolds).
We say that two partial contractions of $\widetilde{T}_1 \otimes \widetilde{T}_2$ and $T_1 \otimes T_2$ are the \defn{same} if they are obtained by contracting the same pairs of indices and listing free indices in the same order;
e.g.\ $\widetilde{S}_{AD'EB'}\widetilde{T}^{D'}{}_C{}^E$ and $S_{\alpha\delta'\epsilon\beta'}T^{\delta'}{}_\gamma{}^\epsilon$ are the same partial contraction.

Let $S$ be a natural submanifold tensor of bi-rank $(r,s)$ and homogeneity $w$ on $k$-submanifolds of $n$-manifolds.
The \defn{tensor weight} of $S$ is $w - r - s$.
This invariant has three fundamental properties.
First, the tensor weight equals the homogeneity on scalars.
Second, the tensor weight is additive with respect to tensor products:
if $S$ has tensor weight $w_1$ and if $T$ has tensor weight $w_2$, then $S \otimes T$ has tensor weight $w_1 + w_2$.
Third, the tensor weight is unchanged by contraction;
e.g.\ if $S_{\alpha\beta\gamma'}$ has tensor weight $w$, then so does $S_\alpha{}^\alpha{}_{\alpha'}$.

Together these properties allow us to consider partial contraction of tensor products of straight submanifold tensors.

\begin{lemma}
    \label{lem:contraction-straight}
    Let $\widetilde{T}_i$, $i \in \{ 1, 2 \}$, be straight submanifold tensors of tensor weight $w_i$ on $(k+2)$-submanifolds of $(n+2)$-manifolds.
    Then any partial contraction $\widetilde{U}$ of $\widetilde{T}_1 \otimes \widetilde{T}_2$ is a straight submanifold tensor of tensor weight $w_1+w_2$.
    Moreover, if $T_i$, $i \in \{ 1 , 2 \}$, are straightenable tensor invariants associated to $\widetilde{T}_i$, then the same partial contraction of $T_1 \otimes T_2$ is a straightenable tensor invariant associated to $\widetilde{U}$.
\end{lemma}

\begin{proof}
    It follows immediately from Definition~\ref{defn:stragT} that $\widetilde{T}_1 \otimes \widetilde{T}_2$ is a straight submanifold tensor of tensor weight $w_1+w_2$ and that $T_1 \otimes T_2$ is a straightenable submanifold tensor associated to $\widetilde{T}_1 \otimes \widetilde{T}_2$.

    Let $\cjmath \colon \cmS \to (\cmG,\cg)$ be the canonical extrinsic ambient space of a minimal submanifold $j \colon \Sigma^k \to (M^n,g)$ of an Einstein manifold.
    Lemma~\ref{lem:minimal-in-einstein-ambient-space} implies that $\widetilde{g}^{\alpha\beta} = \tau^{-2}g^{\alpha\beta}$ and $\widetilde{g}^{\alpha'\beta'} = \tau^{-2}g^{\alpha'\beta'}$.
    It follows that any partial contraction $\widetilde{U}$ of $\widetilde{T}_1 \otimes \widetilde{T}_2$ is straight, and that the same partial contraction $T_1 \otimes T_2$ is straightenable and associated to $\widetilde{U}$.
    The final conclusion follows from the fact that the tensor weight is unchanged by contraction.
\end{proof}

Our second construction, which applies only to scalars, is via the ambient Laplacian.
Note that the associated straightenable invariants can be chosen to be extrinsic scalar invariants in this construction.

\begin{proposition} 
    \label{prop:straightIell}
    Let $\widetilde{I}$ be a straight submanifold scalar of homogeneity $w$ on $(k+2)$-submanifolds of $(n+2)$-manifolds.
    Let $\ell \in \mathbb{N}$.
    Then $\cDelta^\ell\widetilde{I}$ is a straight submanifold scalar of homogeneity $w-2\ell$.
    Additionally, if $I$ is a straightenable natural submanifold scalar associated to $\widetilde{I}$, then
    \begin{equation}
        \label{I_ell-definition}
        I_\ell := \left( \prod_{s=0}^{\ell-1} \left( \overline{\Delta} + \frac{2(2s-w)(k+w-2s-1)}{k}\mathcal{P}_{\alpha}{}^{\alpha} \right) \right) I
    \end{equation}
    is a straightenable natural submanifold scalar associated to $\cDelta^\ell\widetilde{I}$.
    Moreover, if $I$ is an extrinsic scalar invariant, then $I_\ell$ is an extrinsic scalar invariant.
\end{proposition}

\begin{proof}
    Let $\widetilde{\jmath} \colon \cmS \to (\cmG,\cg)$ be the canonical extrinsic ambient space of a minimal submanifold $j \colon \Sigma^k \to (M^n,g)$ of an Einstein manifold with $\Ric = (n-1)\lambda g$.
    Equation~\eqref{eqn:pullback-metric-is-straightenable} implies that if $u \in C^\infty(\Sigma)$ and $w \in \bR$, then
    \begin{equation*}
        \cDelta(\tau^w\varpi^\ast u) = \tau^{w-2}\varpi^\ast\bigl( (\overline{\Delta} - w(k+w-1)\lambda)u \bigr)
    \end{equation*}
    (cf.\ \cite{CLY23}*{Lemma~5.1}).
    It immediately follows that $\cDelta^\ell\cI$ is straight and
    \begin{equation}
        \label{eqn:I_ell-definition-with-lambda}
        \cDelta^\ell\widetilde{I} = \tau^{w-2\ell}\varpi^\ast\prod_{s=0}^{\ell-1} \left( \overline{\Delta} + (2s-w)(k+w-2s-1)\lambda \right) I .
    \end{equation}
    Finally, Equation~\eqref{eqn:defn-mP} implies that $\mathcal{P}_\alpha{}^\alpha = k\lambda/2$.
    This yields Equation~\eqref{I_ell-definition}.
    The final conclusion follows from the fact that $\mathcal{P}_{\alpha\beta}$ is an extrinsic tensor invariant.
\end{proof}

The above constructions produce straight submanifold scalars of high order in the metric that are easily computed modulo natural divergences.

Another key point of straight invariants is that they give rise to easily computable conformal submanifold scalars:

\begin{lemma}
    \label{lem:straight-is-almost-conformal}
    Let $\widetilde{I}$ be a straight submanifold scalar of homogeneity $w \geq -k$ on $(k+2)$-submanifolds of $(n+2)$-manifolds.
    Set $\mathcal{I} := \iota^\ast\widetilde{I}$.
    If $I$ is a straightenable submanifold scalar associated to $\widetilde{I}$ and if $j \colon \Sigma^k \to (M^n,g)$ is a minimal submanifold of an Einstein manifold, then
    \begin{equation*}
        I^{j,g} = \mathcal{I}^{j^\ast g} .
    \end{equation*}
\end{lemma}

\begin{proof}
    Let $\cjmath \colon \cmS \to (\cmG,\cg)$ be the canonical extrinsic ambient space for $j$ and set $h := j^\ast g$.
    Theorem~\ref{thm:construction-of-scalars} implies that $\mathcal{I}$ is well-defined.
    The definition of the canonical extrinsic ambient space yields $\tau \iota h = 1$ and $\varpi \iota h = \Id_\Sigma$.
    We deduce that
    \begin{equation*}
        \mathcal{I}^h = h^\ast \iota^\ast \widetilde{I}\,^{\cjmath,\cg} = h^\ast \iota^\ast \left( \tau^w \varpi^\ast I^{j,g} \right) = I^{j,g} . \qedhere
    \end{equation*}
\end{proof}

Proposition~\ref{prop:straightIell} and Lemma~\ref{lem:straight-is-almost-conformal} give an effective way to compute a large class of conformal submanifold scalars:

\begin{corollary}
    \label{cor:straight-mod-divergence}
    Let $\widetilde{I}$ be a straight submanifold scalar of homogeneity $w$ on $(k+2)$-submanifolds of $(n+2)$-manifolds.
    Let $\ell \in \mathbb{N}_0$ and suppose that $w - 2\ell \geq -k$.
    If $I$ is a straightenable conformal submanifold scalar associated to $\widetilde{I}$ and if $j \colon \Sigma^k \to (M^n,g)$ is a minimal submanifold of an Einstein manifold with $\Ric = (n-1)\lambda g$, then
    \begin{equation*}
        \bigl( \iota^\ast \cDelta^\ell\widetilde{I} \, \bigr)^{j^\ast g} \equiv (2\lambda)^\ell\frac{(-w/2+\ell-1)!(k+w-1)!!}{(-w/2-1)!(k+w-2\ell-1)!!} I^{j,g} 
    \end{equation*}
    modulo natural divergences.
\end{corollary}

\begin{proof}
    This follows immediately from Equation~\eqref{eqn:I_ell-definition-with-lambda} and Lemma~\ref{lem:straight-is-almost-conformal}.
\end{proof}

We conclude this section by deriving those formulas from the introduction that rely on straight invariants but do not involve renormalization.

First, we systematically compute conformal submanifold scalars at minimal submanifolds of Einstein manifolds:

\begin{proof}[Proof of Theorem~\ref{thm:compute-straight}]
    Lemmas~\ref{lem:sff-riem-straight} and~\ref{lem:contraction-straight} imply that $\cmP_{a,b}$ is a straight submanifold scalar of homogeneity $w := -a-2b$.
    The conclusion follows from Equation~\eqref{eqn:I_ell-definition-with-lambda} and Lemma~\ref{lem:straight-is-almost-conformal}.
\end{proof}

Second, we derive a Gauss--Bonnet--Chern-type formula for compact minimal submanifolds of Einstein manifolds:

\begin{proof}[Proof of Corollary~\ref{cor:compact-area}]
    Let $j \colon \Sigma^k \to (M^n,g)$, $k$ even, be a minimal submanifold of an Einstein manifold with $\Ric = (n-1)\lambda g$.

    First, we compute the intrinsic Pfaffian $\overline{\Pf}{}^h$ of $h := j^\ast g$.
    The Gauss equation~\eqref{eqn:gauss} yields
    \begin{align*}
        \overline{\Rm} & = j^\ast\Rm + \frac{1}{2}\intl \wedge \intl = \widehat{W} + \frac{\lambda}{2} h \wedge h , \\
        \widehat{W} & := j^\ast W + \frac{1}{2}\intl \wedge \intl .
    \end{align*}
    On the one hand, Lemmas~\ref{lem:sff-riem-straight} and~\ref{lem:contraction-straight} imply that $\cjmath^\ast\widetilde{\Rm} + \frac{1}{2}\widetilde{L} \wedge \widetilde{L}$ is a straight submanifold tensor to which $\widehat{W}$ is associated.
    Moreover, Equation~\eqref{eqn:gauss} yields
    \begin{equation*}
        \widetilde{\overline{\Rm}} = \cjmath^\ast\widetilde{\Rm} + \frac{1}{2} \widetilde{L} \wedge \widetilde{L} .
    \end{equation*}
    On the other hand, the Binomial Theorem and Equation~\eqref{eqn:Pf-with-g} yield
    \begin{equation}
        \label{eqn:evaluate-pfaffian-general}
        \begin{split}
            \overline{\Pf}{}^h & = \sum_{r=0}^{k/2} \binom{k/2}{r}\left( \frac{\lambda}{2}\right)^{k/2-r}\Pf_{k/2}\Bigl( \widehat{W}^{\otimes r} \otimes \bigl( h \wedge h \bigr)^{\otimes(k/2-r)} \Bigr) \\
            & = \sum_{r=0}^{k/2} (k-2r-1)!!\lambda^{k/2-r}\Pf_r\bigl( \widehat{W} \bigr) .
        \end{split}
    \end{equation}
    Set $\widetilde{\mathcal{P}}_r := \Pf_r( \widetilde{\overline{\Rm}} )$.
    Lemma~\ref{lem:contraction-straight} implies that $\widetilde{\mathcal{P}}_r$ is a straight invariant to which $\Pf_r(\widehat{W})$ is associated.
    Lemma~\ref{lem:straight-is-almost-conformal} then yields
    \begin{equation*}
        \bigl( \iota^\ast \widetilde{\mathcal{P}}_{r} \bigr)^h = \Pf_{r}( \widehat{W} ) ,
    \end{equation*}
    while Corollary~\ref{cor:straight-mod-divergence} yields
    \begin{equation*}
        \bigl( \iota^\ast \cDelta^{k/2-r}\widetilde{\mathcal{P}}_{r} \bigr)^h \equiv (2\lambda)^{k/2-r}\frac{(k/2-1)!(k-2r-1)!!}{(r-1)!} \bigl( \iota^\ast \widetilde{\mathcal{P}}_{r} \bigr)^h
    \end{equation*}
    modulo natural divergences.
    Combining these with Equation~\eqref{eqn:evaluate-pfaffian-general} yields
    \begin{equation}
        \label{eqn:evaluate-pfaffian-mod-div-general}
        \overline{\Pf}{}^h \equiv (k-1)!!\lambda^{k/2} + \sum_{r=1}^{k/2} 2^{r-k/2}\frac{(r-1)!}{(k/2-1)!}\iota^\ast\bigl( \widetilde{\Delta}^{k/2-r} \widetilde{\mathcal{P}}_r \bigr)
    \end{equation}
    modulo natural divergences.
    Integrating this over a compact manifold via the Divergence Theorem yields the final conclusion.
\end{proof}

\section{Renormalized extrinsic curvature integrals}
\label{sec:albin}

In this section we generalize results of Albin~\cite{Albin2009} to the setting of conformally compact minimal submanifolds of conformally compact Einstein manifolds.
Indeed, as in Albin's work, the results of this section depend only on the formal asymptotics of such spaces below the order of the respective nonlocal terms.
Since geodesic defining functions do not pull back to geodesic defining functions on submanifolds, we renormalize using the larger class of even defining functions.
Our approach is inspired by that of Graham and his coauthors~\cites{EptaminitakisGraham2021,GrahamReichert2020,GrahamWitten1999,Graham2000}, though the focus on even defining functions is new.

\subsection{Even asymptotically hyperbolic manifolds}

We begin by defining even asymptotically hyperbolic manifolds and computing the asymptotic expansions of natural Riemannian tensors thereon.
Our presentation mostly follows Albin~\cite{Albin2009}, though we compute with $\mathbb{Z}_2$-gradings on covariant tensors, rather than just functions, and exclusively employ Hadamard regularization.

A \defn{collar neighborhood} for a manifold-with-boundary $\overline{M}$ is a diffeomorphism $F \colon [0,\varepsilon_0) \times \partial\overline{M} \to U$ onto a neighborhood $U \subset \overline{M}$ of $\partial\overline{M}$ with the property that $F(0,\cdot)$ is the inclusion map.
We say that $\overline{M}$ is \defn{collared} if it is has been equipped with a fixed collar neighborhood, and in this case we always denote by $\rho$ the coordinate on the $[0,\varepsilon_0)$ factor.

Let $\overline{M}{}^n$ be a collared manifold-with-boundary.
A section $T$ of a vector bundle $E \to \overline{M}$ is \defn{polyhomogeneous} if its restriction to the interior $M$ of $\overline{M}$ is smooth and there are a strictly increasing sequence $(m_j)_{j=0}^\infty$ of integers and a double sequence $(T_{(i,j)})_{j \geq 0 , i \geq m_j}$ of smooth sections of the pullback bundle $E\rvert_{\partial\overline{M}} \to \partial\overline{M}$ such that $T$ has an asymptotic expansion
\begin{equation}
    \label{eqn:Cphg}
    T = \sum_{j = 0}^\infty\sum_{i \geq m_j} T_{(i,j)}\rho^i(\log\rho)^j
\end{equation}
near $\{ \rho = 0 \}$.
We say that $T$ is of \defn{class $\Cphg^m$} if its asymptotic expansion~\eqref{eqn:Cphg} is valid with $m_0=0$ and $m_1=m$.
Such $T$ has an asymptotic expansion
\begin{equation}
    \label{eqn:phg-expansion-capped}
    T = T_{(0,0)} + \dotsm + T_{(m-1,0)}\rho^{m-1} + T_{(m,1)}\rho^m\log\rho + T_{(m,0)}\rho^m + o(\rho^m) .
\end{equation}
We say that $T$ is of \defn{class $C^m$} if additionally $T_{(m,1)}=0$.
Since manifolds-with-boundary and collar neighborhoods are smooth, the classes of $\Cphg^{m}$ and $C^m$ sections are independent of the choice of collar neighborhood;
see Grieser's lecture notes~\cite{Grieser2001} for additional details.

Our results for conformally compact Einstein $n$-manifolds (resp.\ conformally compact minimal $k$-submanifolds in conformally compact Einstein $n$-manifolds) only require the validity of the expansion~\eqref{eqn:phg-expansion-capped} with $m=n-1$ (resp.\ $m=k+1$), but we find it convenient to work in the class of polyhomogeneous sections.
This is no restriction for conformally compact Einstein manifolds~\cite{ChruscielDelayLeeSkinner2005} or for conformally compact, graphical, minimal hypersurfaces~\cite{MarxKuo2025}.

A choice of collar neighborhood determines an even structure near the boundary~\cite{EptaminitakisGraham2021}.
We exploit this by introducing $\mathbb{Z}_2$-gradings\footnote{
    An algebra $A$ is $\mathbb{Z}_2$-graded if it decomposes $A = A(1) \oplus A(-1)$ as vector spaces and $A\bigl((-1)^s\bigr)A\bigl((-1)^t\bigr) \subseteq A\bigl((-1)^{s+t}\bigr)$ for all $s,t \in \mathbb{Z}_2$.
    A linear operator $D$ on $A$ has degree $k$ if $DA\bigl((-1)^s\bigr) \subseteq A\bigl((-1)^{s+k}\bigr)$ for all $s \in \mathbb{Z}_2$.
}
on the spaces of polyhomogeneous covariant tensors on a collared manifold-with-boundary.

Denote by $\mFphg(1)$ the vector space of polyhomogeneous functions $f$ of class $\Cphg^{n-1}$ on a collared manifold-with-boundary $\overline{M}{}^n$ such that $f_{(i,0)}=0$ in Equation~\eqref{eqn:Cphg} whenever $i \leq n-2$ is odd.
Denote by $\mF(1) \subset \mFphg(1)$ the subspace whose elements $f$ also satisfy $f_{(n-1,1)}=0$ and, if $n$ is even, $f_{(n-1,0)}=0$.
Denote by $\mFphg(-1)$, or equivalently $\mF(-1)$, the vector space of polyhomogeneous functions $f$ of class $C^{n-1}$ such that $f_{(i,0)}=0$ whenever $i \leq n-1$ is even.
Denote
\begin{equation*}
    \Cphg^{n-1}(\overline{M}) := \mFphg := \mFphg(1) + \mFphg(-1)
\end{equation*}
and observe that
\begin{align*}
    \mFphg\bigl((-1)^s\bigr)\mFphg\bigl((-1)^t\bigr) & \subseteq \mFphg\bigl((-1)^{s+t}\bigr) , \\
    \mF\bigl((-1)^s\bigr)\mF\bigl((-1)^t\bigr) & \subseteq \mF\bigl((-1)^{s+t}\bigr) ,
\end{align*}
for all $s,t \in \mathbb{Z}$.
Moreover, since $\mF(1) \cap \mF(-1) = \{ 0 \}$, we see that $\mF(\pm1)$ gives $\mF := \mF(1) \oplus \mF(-1)$ the structure of a $\mathbb{Z}_2$-graded algebra.
We call $\mF(1)$ (resp.\ $\mF(-1)$) the set of \defn{even} (resp.\ \defn{odd}) functions.
We say that a (different) collar neighborhood $F'$ for $\overline{M}$ is \defn{even} if whenever $(x^i)_{i=1}^{n-1}$ are local coordinates on $\partial\overline{M}$, the functions $\rho \circ (F')^{-1}$ and $x^i \circ (F')^{-1}$ are odd and even, respectively.
This defines an equivalence class of collared neighborhoods on $\overline{M}$, and the spaces $\mFphg(\pm1)$ and $\mF(\pm1)$ depend only on this equivalence class.

Denote by $\mFphg^1(\pm 1)$ the vector space of polyhomogeneous one-forms $\omega$ on $\overline{M}$ with the property that if $(x^i)_{i=1}^{n-1}$ are local coordinates on $\partial\overline{M}$, then
\begin{equation*}
    \omega = \rho^{-1} \bigl( \omega_0 \, d\rho + \omega_i\,dx^i \bigr)
\end{equation*}
for local functions $\omega_0 \in \mFphg(\pm1)$ and $\omega_i \in \mFphg(\mp1)$, $i \in \{ 1, \dotsc, n-1 \}$.
Informally, the parity is determined by the requirement that $\rho^{-1}\,d\rho$ be even and $\rho^{-1}\,dx^i$ be odd.
The choice to divide by $\rho$ is consistent both with the asymptotic behavior of $g_+$ and, by duality, with the use of the vector fields $\rho\partial_\rho$ and $\rho\partial_{x^i}$ for analysis on conformally compact manifolds (cf.\ \cite{MazzeoMelrose1987}).
The spaces $\mF^1(\pm 1)$ are defined similarly.
Denote $\mFphg^1 := \mFphg^1(1) + \mFphg^1(-1)$.
The space $\mF^1 := \mF^1(1) \oplus \mF^1(-1)$ of one-forms of class $\rho^{-1}C^{n-1}$ has the structure of a $\mathbb{Z}_2$-graded module over $\mF$.
Indeed,
\begin{align*}
    \mFphg\bigl((-1)^s\bigr)\mFphg^1\bigl((-1)^t\bigr) & \subseteq \mFphg^1\bigl((-1)^{s+t}\bigr) , \\
    \mF\bigl((-1)^s\bigr)\mF^1\bigl((-1)^t\bigr) & \subseteq \mF^1\bigl((-1)^{s+t}\bigr) .
\end{align*}

Similarly, we denote by $\mFphg^k(\pm1)$ (resp.\ $\mF^k(\pm1)$) the vector spaces of polyhomogeneous sections of $\otimes^kT^\ast M$ that, near $\partial\overline{M}$, can be expressed as linear combinations of tensor products of the even one-form $\rho^{-1}d\rho$ and the odd one-form $\rho^{-1}dx^i$, with coefficients in $\mFphg(\pm1)$ (resp.\ $\mF(\pm1)$) and $\mFphg(\mp1)$ (resp.\ $\mF(\mp1)$), respectively.
We set $\mFphg^k := \mFphg^k(1) + \mFphg^k(-1)$ and observe that $\mF^k := \mF^k(1) \oplus \mF^k(-1)$ is a graded $\mathbb{Z}_2$-module over $\mF$.
If $T \in \mFphg^k$ (resp.\ $T \in \mF^k$), then $\rho^kT$ extends to a section of class $\Cphg^{n-1}$ (resp.\ of class $C^{n-1}$) of $\otimes^kT^\ast\overline{M}$.

A \defn{conformally compact manifold} is a complete pseudo-Riemannian manifold $(M^n,g_+)$ together
with a compact collared manifold-with-boundary $\overline{M}$ such that $M$ is the interior of $\overline{M}$ and $g_+ \in \mFphg^2$.
Note that the \defn{conformal compactification} $\bigl(\overline{M},[\rho^2g_+]\bigr)$ and the \defn{conformal infinity} $(\partial_\infty M,\kc) := (\partial\overline{M},[\rho^2g_+\rvert_{\partial_\infty M}])$ are independent of the choice of collar neighborhood.
We emphasize that $\kc$ is smooth, but that $[\rho^2g_+]$, as a conformal class on $\overline{M}$, need not be smooth.

An \defn{asymptotically hyperbolic manifold} is a conformally compact manifold $(M^n,g_+)$ such that $\lvert d\rho \rvert_{\rho^2g_+} = 1$ along $\partial_\infty M$.
This is independent of the choice of collar neighborhood.
This terminology is explained by the conformal transformation law for the Riemann curvature tensor, which implies~\cite{Mazzeo1986}*{Proposition~1.10} that
\begin{equation*}
    \Rm^{g_+} = -\lvert d\rho \rvert_{\rho^2g_+}^2 g_+ \wedge g_+ + O(\rho^{-3}) .
\end{equation*}
We say that $(M^n,g_+)$ is \defn{even} if whenever $(x^i)_{i=1}^{n-1}$ are local coordinates on $\partial_\infty M$, it holds that
\begin{equation}
    \label{eqn:ah-metric-components}
    \rho^2g_+ = g_{00} \, d\rho^2 + 2g_{0i} \, d\rho \, dx^i + g_{ij} \, dx^i \, dx^j
\end{equation}
for local functions $g_{00} \in \mF(1)$, $g_{0i} \in \mF(-1)$, $g_{ij} \in \mFphg(1)$, and, moreover, 
\begin{align*}
    (g^{ij})^{(0,0)}(g_{ij})_{(n-1,1)} & = 0 , \\
    (g^{ij})^{(0,0)}(g_{ij})_{(n-1,0)} & = 0 , && \text{if $n$ is even} ,
\end{align*}
where $(g^{ij})^{(0,0)}$ are the components of the inverse of $(g_{ij})_{(0,0)}$.
Chru\'sciel, Delay, Lee, and Skinner~\cite{ChruscielDelayLeeSkinner2005}*{Theorem~A} and Fefferman and Graham~\cite{FeffermanGraham2012}*{Chapter~4} showed that conformally compact Einstein manifolds are even asymptotically hyperbolic manifolds.
We say that $(M^n,g_+)$ is \defn{strongly even} if $g_+ \in \mF^2(1)$.

Let $(M^n,g_+)$ be an even asymptotically hyperbolic manifold.
An \defn{even defining function} for $\partial_\infty M$ is a polyhomogeneous function $r$ such that $r/\rho \in \mF(1)$ and $r/\rho \rvert_{\partial_\infty M}$ is positive.
\defn{Geodesic defining functions}, which are nonnegative functions on $\overline{M}$ such that $\partial\overline{M} = \{ r = 0 \}$ and $\lvert dr \rvert_{r^2g_+} = 1$ in a neighborhood of $\partial_\infty M$, are even (cf.\ \cite{Guillarmou2005}*{Lemma~2.1}).

There are two key points to the definitions above.
First, as we will see in the remainder of this section, they are sufficiently general to apply to conformally compact Einstein manifolds and to conformally compact minimal submanifolds therein.
Second, we have a general renormalization result that recovers properties known for renormalized volumes~\cites{Graham2000,GoverWaldron2017} and renormalized curvature integrals~\cites{Albin2009,CKLTY2024}.
Our proof draws heavily from Graham's study~\cite{Graham2000} of the renormalized volume.

\begin{proposition}
    \label{prop:fundamental-renormalization}
    Let $(M^n,g_+)$ be an even asymptotically hyperbolic manifold and let $f \in \mF(1)$.
    Let $r$ be an even defining function for $\partial_\infty M$.
    Then there is an asymptotic expansion
    \begin{equation}
        \label{eqn:integral-expansions}
        \begin{aligned}
        \int_{\{ r > \varepsilon \}} f\dvol_{g_+} & = \sum_{i = 0 }^{(n-2)/2} \varphi_{(2i)}\varepsilon^{2i+1-n} + \mathscr{V}_f + o(1) , && \text{if $n$ is even} , \\
        \int_{\{ r > \varepsilon \}} f \dvol_{g_+} & = \sum_{i = 0}^{(n-3)/2} \varphi_{(2i)}\varepsilon^{2i+1-n} + \mathscr{L}_f \log\varepsilon + \mathscr{V}_f + o(1) , && \text{if $n$ is odd} ,
        \end{aligned}
    \end{equation}
    as $\varepsilon \to 0$.
    Moreover,
    \begin{enumerate}
        \item if $n$ is even, then $\mathscr{V}_f$ is independent of the choice of $r$, and $\mathscr{V}_f = 0$ if $f$ is the $g_+$-divergence of a one-form $\omega \in \mF^1(1)$; and
        \item if $n$ is odd, then $\mathscr{L}_f$ is independent of the choice of $r$, and $\mathscr{L}_f=0$ if $f$ is the $g_+$-divergence of a one-form $\omega \in \mF^1(1)$.
    \end{enumerate}
\end{proposition}

\begin{proof}
    Denote by $F \colon [0,\varepsilon_0) \times \partial_\infty M \to \overline{M}$ the collar neighborhood of $\overline{M}$.
    By shrinking $\varepsilon_0$ if necessary, we may assume that $F^\ast dr$ is nowhere-vanishing.

    Set $h := \rho^2g_+\rv_{T\partial_\infty M}$.
    Our assumptions imply that
    \begin{equation*}
        f \dvol_{g_+} = \sum_{i=0}^{\lfloor (n-1)/2 \rfloor} \phi_{(2i)}\rho^{2i-n} \, d\rho \dvol_h \mathop{+} o(\rho^{-1})
    \end{equation*}
    as $\rho \to 0$, where $\phi_{(2i)} \in C^\infty(\partial_\infty M)$ for $i \leq (n-1)/2$.
    Since $r$ is an even defining function, there is a positive function $b = b(r,x)$ of class $C^{n-1}$ on some product $[0,\delta_0) \times \partial_\infty M$ such that $\rho = br$ and $b(\cdot,x)$ mod $o(r^{n-1})$ has an even expansion.
    Let $\varepsilon > 0$ be sufficiently small.
    Set $\epsilon(x) := \varepsilon b(\varepsilon,x)$, so that $\{ r > \varepsilon \} = \{ \rho > \epsilon \}$.
    Then
    \begin{equation}
        \label{eqn:plug-in-asymptotics}
        \int_{\{ r > \varepsilon \}} f \dvol_{g_+} = \sum_{i=0}^{\lfloor (n-1)/2 \rfloor} \int_{\partial_\infty M} \int_\epsilon^{\varepsilon_0} \phi_{(2i)}(x)\rho^{2i-n} \, d\rho \dvol_h(x) + O(1) .
    \end{equation}
    Integrating Equation~\eqref{eqn:plug-in-asymptotics} in $\rho$ yields the expansion~\eqref{eqn:integral-expansions}.
    Equation~\eqref{eqn:plug-in-asymptotics} and our definition of $\epsilon$ also imply that
    \begin{multline*}
        \int_{\{ r > \varepsilon \}} f \dvol_{g_+} - \int_{\{ \rho > \varepsilon \}} f \dvol_{g_+} = -\int_{\partial_\infty M} \phi_{(n-1)}\log b(\varepsilon,\cdot) \dvol_h \\
        + \sum_{i=0}^{\lfloor (n-2)/2 \rfloor} \frac{\varepsilon^{2i-n+1}}{n-2i-1}\int_{\partial_\infty M} \bigl( b(\varepsilon,\cdot)^{2i-n+1} - 1 \bigr)\phi_{(2i)} \dvol_h \mathop{+} o(1) ,
    \end{multline*}
    where $\phi_{(n-1)}:=0$ if $n$ is even.
    Since $b(0,\cdot)$ is positive, we see that if $n$ is odd, then $\mathscr{L}_f$ is independent of the choice of $r$.
    Since $b$ is even in $\varepsilon$, we see that if $n$ is even, then $\mathscr{V}_f$ is independent of the choice of $r$.

    Finally, suppose that $f = \mathrm{div}^{g_+}\,\omega$ for some $\omega \in \mF(1)$.
    By the above, it suffices to compute the finite (resp.\ logarithmic) term in the expansion~\eqref{eqn:integral-expansions} when $r=\rho$ in the case when $n$ is even (resp.\ $n$ is odd).
    The Divergence Theorem yields
    \begin{equation*}
        \int_{\{ \rho > \varepsilon \}} f \dvol_{g_+} = \int_{\{ \rho = \varepsilon\}} \omega(\mu) \, \mu \lrcorner \dvol_{g_+} ,
    \end{equation*}
    where
    \begin{equation*}
        \mu := \bigl( g_{00} - g_{0i}g_{0j}g^{ij} \bigr)^{-1/2} \bigl( \rho\partial_\rho - \rho g_{0j}g^{ij}\partial_i \bigr)
    \end{equation*}
    is the inward-pointing unit normal with respect to $g_+$ along $\{ \rho = \varepsilon \}$ and $g^{ij}$ denotes the one-parameter family of inverses of $g_{ij}$.
    Our assumptions on $(M^n,g_+)$ and $\omega$ imply that $\omega(\mu) = a$ and $\mu \lrcorner \dvol_{g_+} = \varepsilon^{1-n}b \dvol_h$ for functions $a,b \in \mF(1)$.
    The conclusion readily follows.
\end{proof}

Albin~\cite{Albin2009} and Case, Khaitan, et al.~\cite{CKLTY2024} showed that the evaluations of natural Riemannian scalars and one-forms, respectively, are even on even asymptotically hyperbolic manifolds of even dimension.
We rederive their results in general dimensions, as our study of renormalized extrinsic curvature integrals imposes no assumptions on the dimension of the target manifold.
The key fact is that covariant derivatives of the Riemann curvature tensor are even (cf.\ \cite{Albin2009}*{Corollary~3.3}):

\begin{lemma}
    \label{lem:albin-expansion}
    Let $(M^n,g_+)$ be an even asymptotically hyperbolic manifold and let $\ell \geq 0$ be an integer.
    Then $\nabla^\ell\Rm \in \mFphg^{\ell+4}(1)$ for each integer $\ell \geq 0$.
    Moreover, $\bigl( \Rm^{g_+} \mathop{+} \frac{1}{2}g_+ \wedge g_+\bigr)_{(0,0)} = 0$ and
    \begin{enumerate}
        \item if $n$ is even, then the components of $(\nabla^\ell\Rm)_{(n-1,0)}^{g_+}$ are linear combinations of partial contractions of the tensors $\mathcal{K} \otimes h^{\otimes s}$, $s \in \mathbb{N}_0$, where $h_{ij} := (g_{ij})_{(0,0)}$ and $\mathcal{K}_{ij} := (g_{ij})_{(n-1,0)}$;
        \item if $n$ is odd, then the components of $(\nabla^\ell\Rm)_{(n-1,1)}^{g_+}$ are linear combinations of partial contractions of the tensors $\mathcal{K} \otimes h^{\otimes s}$, $s \in \mathbb{N}_0$, where $h_{ij} := (g_{ij})_{(0,0)}$ and $\mathcal{K}_{ij} := (g_{ij})_{(n-1,1)}$.
    \end{enumerate}
    In particular, if $n$ is even (resp.\ $n$ is odd), then $(R_{abcd;e_1 \dotsm e_\ell})_{(n-1,0)} = 0$ (resp.\ $(R_{abcd;e_1 \dotsm e_\ell})_{(n-1,1)} = 0$) whenever an odd number of $a,b,c,d,e_1,\dotsc,e_\ell$ is nonzero.
\end{lemma}

\begin{proof}
    Pick local coordinates $(x^i)_{i=1}^{n-1}$ on $\partial_\infty M$ and extend these, via the given collar neighborhood, to local coordinates $(x^a)_{a=0}^{n-1}$ on $\overline{M}$ with $x^0:=\rho$.
    Throughout this proof, indices $i,j,k$ take values in $\{ 1, \dotsc, n-1 \}$, indices $a,b,c$ take values in $\{ 0, \dotsc, n-1 \}$, and indices $s,t \in \mathbb{Z}$ record the $\mathbb{Z}_2$-gradings.
    
    Set $X_0 := \rho\partial_\rho$ and $X_i := \rho\partial_i$.
    Since $(M^n,g_+)$ is even,
    \begin{equation}
        \label{eqn:albin-metric-coefficients}
        \begin{aligned}
            g_{00} & \in \mF(1) , & g_{0i} & \in \mF(-1), & g_{ij} & \in \mFphg(1) .
        \end{aligned}
    \end{equation}
    Moreover,
    \begin{align*}
            X_0\mFphg\bigl((-1)^s\bigr) & \subseteq \mFphg\bigl((-1)^s\bigr) , \\
            X_i\mFphg\bigl((-1)^s\bigr) & \subseteq \mF\bigl((-1)^{s+1}\bigr) ;
    \end{align*}
    the second observation follows from the identity $\mFphg(-1)=\mF(-1)$.
    
    Consider the $\mathbb{Z}_2$-grading on polyhomogeneous vector fields determined by
    \begin{equation}
        \label{eqn:albin-X-preserve-grading}
        \mathcal{X}\bigl((-1)^s\bigr) := \mF\bigl((-1)^s\bigr)X_0 + \mF\bigl((-1)^{s+1}\bigr)X_1 + \dotsm + \mF\bigl((-1)^{s+1}\bigr)X_{n-1} .
    \end{equation}
    Direct computation gives
    \begin{equation}
        \label{eqn:commute-0-frame}
        [ X_0 , X_i ] = X_i ,
    \end{equation}
    and all other inequivalent commutators vanish.
    It follows that $\mathcal{X} := \mathcal{X}(1) \oplus \mathcal{X}(-1)$ is a $\mathbb{Z}_2$-graded Lie algebra;
    i.e.
    \begin{equation}
        \label{eqn:albin-X-graded-lie-algebra}
        \bigl[ \mathcal{X}\bigl((-1)^s\bigr) , \mathcal{X}\bigl((-1)^t\bigr) \bigr] \subseteq \mathcal{X}\bigl((-1)^{s+t}\bigr) .
    \end{equation}
    We use $\{ X_a \}_{a=0}^{n-1}$ to compute components of tensors.
    Thus, a tensor $T$ of rank $\ell$ is in $\mFphg^\ell\bigl((-1)^s\bigr)$ if and only if
    \begin{equation*}
        T_{a_1 \dotsm a_\ell} := T( X_{a_1} , \dotsc , X_{a_\ell} ) \in \mFphg\bigl( (-1)^{s+t} \bigr) , \quad t := \# \bigl\{ i \in \{ 1 , \dotsc , \ell \} \mathrel{}:\mathrel{} a_i \not= 0 \bigr\} .
    \end{equation*}
    We characterize $\mF^\ell\bigl((-1)^s\bigr)$ similarly.

    Define connection coefficients $\Gamma_{ab}^c$ by
    \begin{equation}
        \label{eqn:defn-0-christoffel}
        \nabla^{g_+}_{X_a}X_b := \Gamma_{ab}^cX_c .
    \end{equation}
    Equation~\eqref{eqn:commute-0-frame} implies that $2\Gamma_{[ab]}^c = 2\delta_{[a}^0\delta_{b]}^c$.
    Combining Equations~\eqref{eqn:albin-metric-coefficients}, \eqref{eqn:albin-X-preserve-grading}, and~\eqref{eqn:albin-X-graded-lie-algebra} with the Koszul formula implies that the Levi-Civita connection has degree zero with respect to the $\mathbb{Z}_2$-grading;
    i.e.
    \begin{equation*}
        \Gamma_{ab}^c \in \mFphg\bigl((-1)^s\bigr) , \quad s := \# \bigl\{ i \in \{ a , b , c \} \mathrel{}:\mathrel{} i \not= 0 \bigr\} .
    \end{equation*}
    It follows that $\nabla^\ell\Rm \in \mFphg^{\ell+4}(1)$ for each integer $\ell \geq 0$.

    We conclude by computing the critical coefficients of $(\nabla^\ell\Rm)_{(n-1,1)}$;
    the case of $(\nabla^\ell\Rm)_{(n-1,0)}$ when $n$ is even is similar.
    Define $g_{00},g_{0i},g_{ij}$ as in Equation~\eqref{eqn:ah-metric-components}.
    Since $g_+$ is even and asymptotically hyperbolic, $(g_{00})_{(0,0)}=1$ and $(g_{0i})_{(0,0)}=0$;
    moreover, $h_{ij} := (g_{ij})_{(0,0)}$ defines an invertible matrix $(h_{ij})_{i,j=1}^{n-1}$.
    Let $(h^{ij})_{i,j=1}^{n-1}$ denote its inverse.
    Set $\mathcal{K}_{ij} := (g_{ij})_{(n-1,1)}$;
    since $g_+$ is even, $(g_{00})_{(n-1,1)} = 0$, $(g_{0i})_{(n-1,1)}=0$, and $h^{ij}\mathcal{K}_{ij} = 0$.
    Direct computation (cf.\ \cite{CKLTY2024}*{Proof of Lemma~4.1}) yields $(\Rm^{g_+} + \frac{1}{2}g_+ \wedge g_+)_{(0,0)}=0$ and
    \begin{equation}
        \label{eqn:leading-terms-of-ah-rm}
        \begin{aligned}
        (\Gamma_{ij}^0)_{(0,0)} & = h_{ij} , & (\Gamma_{ij}^0)_{(n-1,1)} & = -\frac{n-3}{2}\mathcal{K}_{ij} , \\
        (\Gamma_{0i}^j)_{(0,0)} & = 0 , & (\Gamma_{0i}^j)_{(n-1,1)} & = \frac{n-1}{2}h^{jk}\mathcal{K}_{ik} , \\
        (R_{0i0j})_{(0,0)} & = -h_{ij} , & (R_{0i0j})_{(n-1,1)} & = -\frac{n^2-4n+5}{2}\mathcal{K}_{ij} , \\
        (R_{ijkl})_{(0,0)} & = -2h_{i[k}h_{l]j} , & (R_{ijkl})_{(n-1,1)} & = \frac{n-3}{2}\bigl( \mathcal{K} \wedge h \bigr)_{ijkl} ,
        \end{aligned}
    \end{equation}
    and all other components of not obtained from these by symmetry vanish.
    The conclusion follows by differentiation.
\end{proof}

The evenness of natural Riemannian $k$-forms follows:

\begin{corollary}
    \label{cor:parity-of-natural-riemannian-objects}
    Let $(M^n,g_+)$ be an even asymptotically hyperbolic manifold.
    If $\omega$ is a natural Riemannian $k$-form, then $\omega^{g_+} \in \mF^k(1)$.
\end{corollary}

\begin{proof}
    By definition, $\omega^{g_+}$ is a linear combination of partial contractions of tensors
    \begin{equation*}
        \nabla^{I_1}\Rm \otimes \dotsm \otimes \nabla^{I_p}\Rm \mathop{\otimes} g^{\otimes J} .
    \end{equation*}
    Since $g_+$ is even, we deduce from Lemma~\ref{lem:albin-expansion} that $\omega^{g_+} \in \mFphg^k(1)$.

    We now show that $(\omega^{g_+})_{(n-1,1)}=0$;
    the proof that $(\omega^{g_+})_{(n-1,0)}=0$ if $n$ is even is similar.
    
    Since no log terms appear in the expansion of $\omega^{g_+}$ below order $\rho^{n-1}\log\rho$, we see that $(\omega^{g_+})_{(n-1,1)}$ is a linear combination of partial contractions of
    \begin{align*}
        (\nabla^{I_1}\Rm)_{(n-1,1)} \otimes (\nabla^{I_2}\Rm)_{(0,0)} \otimes \dotsm \otimes (\nabla^{I_p}\Rm)_{(0,0)} \otimes (g)_{(0,0)} \otimes \dotsm \otimes (g)_{(0,0)} , \\
        (\nabla^{I_1}\Rm_{(0,0)}) \otimes \dotsm \otimes (\nabla^{I_p}\Rm)_{(0,0)} \otimes (g)_{(n-1,1)} \otimes (g)_{(0,0)} \otimes \dotsm \otimes (g)_{(0,0)} .
    \end{align*}
    We deduce from Lemma~\ref{lem:albin-expansion} that the components of $(\omega^{g_+})_{(n-1,1)}$ are linear combinations of partial contractions of
    \begin{equation*}
        \mathcal{K} \otimes h \otimes \dotsm \otimes h .
    \end{equation*}
    Since $\mathcal{K}$ and $h$ are symmetric, the skew symmetry of $\omega$ yields $(\omega^{g_+})_{(n-1,1)}=0$.
\end{proof}

\subsection{Conformally compact minimal submanifolds}
\label{subsec:submanifold}

We now study asymptotically minimal submanifolds of even asymptotically hyperbolic manifolds and compute the asymptotic expansions of natural submanifold tensors thereon.
These spaces include the conformally compact minimal submanifolds of conformally compact Einstein manifolds discussed in the Introduction.
Our presentation is heavily inspired by that of Graham and his coauthors~\cites{GrahamWitten1999,GrahamReichert2020,CGKTW24}, though our discussion of asymptotic expansions of natural submanifold tensors is new.

A nondegenerate submanifold $j \colon \Sigma^k \to (M^n,g_+)$ is \defn{conformally compact} if
\begin{enumerate}
    \item $(M^n,g_+)$ and $(\Sigma^k,j^\ast g_+)$ are conformally compact with conformal infinities $(\partial_\infty M,\kc)$ and $(\partial_\infty\Sigma,\kc_\Sigma)$, respectively,
    \item there is a conformal submanifold $j_\infty \colon \partial_\infty\Sigma \to (\partial_\infty M,\kc)$ such that $j^\ast\kc = \kc_\Sigma$, and
    \item there is a polyhomogeneous section $U$ of $N\partial_\infty\Sigma \to \overline{\Sigma}$ of class $\Cphg^{k+1}$ such that $U(0,\cdot)=0$ and if $\varrho > 0$, then
    \begin{equation*}
        ( F^{-1} \circ j \circ G)(\varrho , x ) = \bigl( \varrho , \exp^\perp U(\varrho,x) \bigr) ,
    \end{equation*}
    where $F$ and $G$ are the collar neighborhoods of the compactifications $\overline{M}$ and $\overline{\Sigma}$, respectively, and $\exp^\perp$ is defined using $\rho^2g_+\vert_{T\partial_\infty M}$.
\end{enumerate}
Throughout this section, $\rho$ and $\varrho$ denote the coordinates on the first factor of the collar neighborhoods of $\overline{M}$ and $\overline{\Sigma}$, respectively.
In this case we call $j_\infty$ the \defn{conformal infinity} of $j$.
We say that $j$ is \defn{asymptotically minimal} if its mean curvature, regarded as a section of $N^\ast\Sigma$, satisfies $H = O(\varrho^{k-1})$.
This is equivalent to the requirement that $H^\sharp = O(\varrho^{k+1})$ as a section of $N\Sigma$ (cf.\ \cite{GrahamReichert2020}*{Theorem~3.1}), where $\sharp$ is defined via $g_+$.

Let $(M^n,g_+)$ be an even asymptotically hyperbolic manifold.
A conformally compact submanifold $j \colon \Sigma^k \to (M^n,g_+)$ is \defn{even} if for each local frame $\{ e_{\alpha'} \}_{\alpha'=k}^{n-1}$ for $N\partial_\infty\Sigma$, the normal bundle of $j_\infty$, we have that
\begin{equation}
    \label{eqn:components-of-U}
    U(\varrho,x) = \sum_{\alpha'=k}^{n-1} U^{\alpha'}(\varrho,x)e_{\alpha'}(x) 
\end{equation}
for functions $U^{\alpha'}$ of class $\Cphg^{k+1}$ satisfying $(U^{\alpha'})_{(2i+1,0)}=0$ if $2i < k$.
Asymptotically minimal submanifolds of even asymptotically hyperbolic manifolds are even.
We prove this by modifying an argument of Graham and Witten~\cite{GrahamWitten1999}.

\begin{lemma}
    \label{lem:even-submanifolds-facts}
    Let $j \colon \Sigma^k \to (M^n,g_+)$ be an asymptotically minimal submanifold of an even asymptotically hyperbolic manifold.
    Then $j$ is even and $(\Sigma^k,j^\ast g_+)$ is a strongly even asymptotically hyperbolic manifold.
\end{lemma}

\begin{proof}
    Fix $p \in \partial_\infty\Sigma$.
    Let $(x^\alpha,u^{\alpha'})$ be Fermi coordinates near $j_\infty(p) \in \partial_\infty M$.
    Lift these to coordinates $(\rho,x^\alpha,u^{\alpha'})$ and $(\varrho,x^\alpha)$ on $M$ and $\Sigma$, respectively, via the appropriate collar neighborhoods.
    Set
    \begin{align*}
        X_a & := \varrho\partial_{x^a} , && \text{on $\Sigma$, where $a \in \{ 0, \dotsc k-1 \}$} , \\
        Z_A & := \rho\partial_{z^A} , && \text{on $M$, where $A \in \{ 0, \dotsc, n-1 \}$} ,
    \end{align*}
    with the conventions $x^0=\varrho$ and $z^0 = \rho$.
    Define a local frame of $T\Sigma \subset j^{-1}TM$ by
    \begin{equation*}
        Y_a := dj(X_a) = Z_a + U_{,a}^{\alpha'}Z_{\alpha'} ,
    \end{equation*}
    where $U^{\alpha'}$ is as in Equation~\eqref{eqn:components-of-U} and $U_{,a}^{\alpha'} := \partial_{x^a} U^{\alpha'}$.
    Let $g_{AB}$ denote the components of $g_+$ with respect to the local frame $\{ Z_A \}$.
    Denote by
    \begin{equation}
        \label{eqn:ah-induced-metric-components}
        \begin{aligned}
            h_{00} & := g_+(Y_0,Y_0) = g_{00} + 2U_{,0}^{\alpha'}g_{0\alpha'} + U_{,0}^{\alpha'}U_{,0}^{\beta'}g_{\alpha'\beta'} , \\
            h_{0\alpha} & := g_+(Y_0,Y_\alpha) = g_{0\alpha} + U_{,\alpha}^{\alpha'}g_{0\alpha'} + U_{,0}^{\beta'}g_{\alpha\beta'} + U_{,\alpha}^{\alpha'}U_{,0}^{\beta'}g_{\alpha'\beta'} , \\
            h_{\alpha\beta} & := g_+(Y_\alpha,Y_\beta) = g_{\alpha\beta} + 2U_{,(\alpha}^{\alpha'}g_{\beta)\alpha'} + U_{,\alpha}^{\alpha'}U_{,\beta}^{\beta'}g_{\alpha'\beta'} ,
        \end{aligned}
    \end{equation}
    the components of $j^\ast g_+$ with respect to the local frame $\{ X_a \}$.
    Let $h^{ab}$ denote the components of the inverse of $(h_{ab})$.
    Set
    \begin{equation}
        \label{eqn:ah-Yalphaprime}
        \begin{aligned}
            Y_{\alpha'} & := \varrho\partial_{x^{\alpha'}} - f_{\alpha'}^0 Y_0 - f_{\alpha'}^\alpha Y_\alpha , \\
            f_{\alpha'}^0 & := h^{00}(g_{0\alpha'} + U^{\beta'}_{,0}g_{\alpha'\beta'}) + h^{\alpha0}(g_{\alpha\alpha'} + U^{\beta'}_{,\alpha}g_{\alpha'\beta'}) , \\
            f_{\alpha'}^\alpha & := h^{0\alpha}(g_{0\alpha'} + U^{\beta'}_{,0}g_{\alpha'\beta'}) + h^{\alpha\beta}(g_{\beta\alpha'} + U^{\beta'}_{,\beta}g_{\alpha'\beta'}) .
        \end{aligned}
    \end{equation}
    It is straightforward to check that $\{ Y_{\alpha'} \}$ is a local frame for $N\Sigma$.
    
    We now compute the components
    \begin{equation*}
        L_{ab\gamma'} := g_+( \nabla^{g_+}_{Y_a} Y_b , Y_{\gamma'} )
    \end{equation*}
    of the second fundamental form of $j$.
    Direct computation gives
    \begin{equation}
        \label{eqn:ah-L-components}
        \begin{aligned}
        L_{ab\gamma'} & = \left( \Gamma_{ab}^C + U_{,b}^{\beta'}\Gamma_{a\beta'}^C + U_{,a}^{\alpha'}\Gamma_{\alpha'b}^C + U_{,a}^{\alpha'}U_{,b}^{\beta'}\Gamma_{\alpha'\beta'}^C \right)\tilde g_{C\gamma'} + \varrho U_{,ab}^{\alpha'}\tilde g_{\alpha'\gamma'} , \\
        \tilde g_{A\gamma'} & := g_+(Z_A,Y_{\gamma'}) = g_{A\gamma'} - f_{\gamma'}^0(g_{A0} + U_{,0}^{\alpha'}g_{A\alpha'}) - f_{\gamma'}^\alpha(g_{A\alpha} + U_{,\alpha}^{\alpha'}g_{A\alpha'}) ,
        \end{aligned}
    \end{equation}
    where the connection coefficients $\Gamma_{AB}^C$ are defined by $\nabla_{Z_A}^{g_+}Z_B = \Gamma_{AB}^CZ_C$ as in Equation~\eqref{eqn:defn-0-christoffel}.
    Moreover, since $j$ is asymptotically minimal,
    \begin{equation}
        \label{eqn:ah-H-components}
        h^{ab}L_{ab\gamma'} = kH(Y_{\gamma'}) \in O(\varrho^{k}) .
    \end{equation}

    We first show that $U^{\alpha'} = O(\varrho^2)$.
    Since $j$ is conformally compact, $U^{\alpha'} = O(\varrho)$.
    Recall that $g_{00} = 1 + O(\rho^2)$ and $g_{0\alpha} = O(\rho)$.
    Thus
    \begin{align*}
        h_{00} & = 1 + U_{,0}^{\alpha'}U_{,0}^{\beta'}g_{\alpha'\beta'} + O(\varrho) , & h_{0\alpha} & = O(\varrho), & h_{\alpha\beta} & = g_{\alpha\beta} + O(\varrho) , \\
        h^{00} & = (1 + U_{,0}^{\alpha'}U_{,0}^{\beta'}g_{\alpha'\beta'})^{-1} + O(\varrho) , & h^{0\alpha} & = O(\varrho), & h^{\alpha\beta} & = g^{\alpha\beta} + O(\varrho) .
    \end{align*}
    The coefficients $f_{\alpha'}^a$ and $\tilde g_{A\gamma'}$ in Equations~\eqref{eqn:ah-Yalphaprime} and~\eqref{eqn:ah-L-components}, respectively, that are nonzero mod $O(\varrho)$ are
    \begin{align*}
        f_{\alpha'}^0 & = h^{00}U_{,0}^{\beta'}g_{\alpha'\beta'} + O(\varrho) , \\
        \tilde g_{0\gamma'} & = -h^{00}U_{,0}^{\alpha'}g_{\alpha'\gamma'} + O(\varrho) , \\
        \tilde g_{\alpha'\gamma'} & = g_{\alpha'\gamma'} - h^{00}U_{,0}^{\beta'}U_{,0}^{\delta'}g_{\alpha'\beta'}g_{\gamma'\delta'} + O(\varrho) .
    \end{align*}
    Equations~\eqref{eqn:commute-0-frame} and~\eqref{eqn:leading-terms-of-ah-rm} imply that
    \begin{equation*}
        \Gamma_{ij}^0 = h_{ij} + O(\rho) , \quad \Gamma_{i0}^j = -\delta_i^j + O(\rho) ,
    \end{equation*}
    and all other connection coefficients $\Gamma_{AB}^C$ vanish mod $O(\rho)$.
    Therefore
    \begin{align*}
        L_{00\gamma'} & = -U_{,0}^{\alpha'}g_{\alpha'\gamma'} + O(\varrho) , \\
        L_{\alpha\beta\gamma'} & = -h^{00}U_{,0}^{\alpha'}g_{\alpha'\gamma'}g_{\alpha\beta} + O(\varrho) .
    \end{align*}
    We deduce from Equation~\eqref{eqn:ah-H-components} that $U^{\alpha'} = O(\varrho^2)$.

    We next show that $(U^{\alpha'})_{(2i+1,0)}=0$ for all integers $0 \leq i \leq (k-2)/2$.
    The case $i=0$ is done.
    Suppose that $0 \leq \ell \leq (k-4)/2$ is an integer such that $(U^{\alpha'})_{(2i+1,0)}=0$ for all integers $0 \leq i \leq \ell$.
    Set $\xi^{\alpha'} := (U^{\alpha'})_{(2\ell+3,0)}$.
    Since $g_+$ is even, Equations~\eqref{eqn:ah-induced-metric-components} imply that $h_{00},h_{\alpha\beta}$ mod $O(\varrho^{2\ell+3})$ are even and that $h_{0\alpha}$ mod $O(\varrho^{2\ell+4})$ is odd.
    Combining this with Equations~\eqref{eqn:ah-Yalphaprime} and~\eqref{eqn:ah-L-components} yields
    \begin{equation*}
        (f_{\alpha'}^0)_{(2\ell+2,0)} = (2\ell+3)\xi_{\alpha'} \quad\text{and}\quad (\tilde g_{0\gamma'})_{(2\ell+2,0)} = -(2\ell+3)\xi_{\gamma'} .
    \end{equation*}
    Combining this with Equations~\eqref{eqn:leading-terms-of-ah-rm} and the evenness of $g_+$ and of $U^{\alpha'}$ yields
    \begin{align*}
        (L_{00\gamma'})_{(2\ell+2,0)} & = (2\ell+1)(2\ell+3)\xi^{\beta'}g_{\beta'\gamma'} , \\
        (L_{\alpha\beta\gamma'})_{(2\ell+2,0)} & = -(2\ell+3)\xi^{\beta'}g_{\beta'\gamma'}g_{\alpha\beta} .
    \end{align*}
    Equation~\eqref{eqn:ah-H-components} then yields
    \begin{equation*}
        k(H_{\gamma'})_{(2\ell+2,0)} = (2\ell+3)(2\ell+2-k)\xi^{\beta'}g_{\beta'\gamma'} .
    \end{equation*}
    Since $2\ell + 2 \leq k - 2 < k$, we conclude that $\xi^{\alpha'}=0$.

    Finally, since $U^{\alpha'}$ mod $O(\varrho^{k+1})$ has an even expansion, we conclude from Equations~\eqref{eqn:ah-induced-metric-components} that $(\Sigma^k , j^\ast g_+)$ is strongly even.
\end{proof}

In order to prove Theorem~\ref{thm:albin}, it now suffices to show that natural submanifold scalars are necessarily even when evaluated at asymptotically minimal submanifolds of even asymptotically hyperbolic manifolds.
To that end, denote by $\overline{\mathcal{F}}{}^\ell(\pm1)$ the $\mathbb{Z}_2$-grading on tensors of rank $\ell$ on the strongly even asymptotically hyperbolic manifold $(\Sigma,j^\ast g_+)$.
Denote by $j^{\ast}\mFphg^\ell(\pm1) \subseteq \overline{\mF}{}^\ell(\pm1)$ the image of $\mFphg^\ell(\pm1)$ under pullback.
We have the following extrinsic analogue of Lemma~\ref{lem:albin-expansion}:

\begin{lemma}
    \label{lem:expansion}
    Let $j \colon \Sigma^k \to (M^n,g_+)$ be an asymptotically minimal submanifold of an even asymptotically hyperbolic manifold.
    Let $\omega$ be a natural submanifold tensor of bi-rank $(\ell,0)$.
    Then $\omega^{j^\ast g_+} \in \overline{\mF}{}^\ell(1)$.
\end{lemma}

\begin{proof}
    Set $\overline{\mF}{}^{(1,0)}(\pm1) := \overline{\mF}{}^1(\pm1)$ and denote by $\overline{\mF}{}^{(0,1)}(\pm1)$ the vector space of sections $\xi$ of $N^\ast\Sigma$ of class $C^k$ such that
    \begin{equation*}
        \xi = \varrho^{-1}\xi_{\alpha'} \, dx^{\alpha'}
    \end{equation*}
    for some $\xi_{\alpha'} \in \overline{\mF}(\mp1)$.
    Denote by $\overline{\mF}{}^{(r,s)}(\pm1)$ the analogous vector space of sections of $(T^\ast\Sigma)^{\otimes r} \otimes (N^\ast\Sigma)^{\otimes s}$.
    We first show that if $\ell\geq0$ and $r,s\geq0$ are such that $r+s=\ell+4$, then $\pi\nabla^\ell\Rm \in \overline{\mF}{}^{(r,s)}(1)$ for appropriate projections $\pi$, and that if $\ell \geq 0$, then $\overline{\nabla}{}^\ell L \in \overline{\mF}{}^{(\ell+2,1)}(1)$.

    In the notation in the proof of Lemma~\ref{lem:even-submanifolds-facts}, it follows from the evenness of $U$ and $g_+$ that $Y_0$ is an even section of $T\Sigma$ and that $Y_\alpha$ and $Y_{\alpha'}$ are odd sections of $T\Sigma$ and $N\Sigma$, respectively.
    In particular, the projections $\pi \colon \mFphg^1 \to \overline{\mF}{}^{(1,0)}$ and $\pi \colon \mFphg^1 \to \overline{\mF}{}^{(0,1)}$ have degree zero.
    Lemma~\ref{lem:albin-expansion} then implies that $\pi\nabla^\ell\Rm$ is even.
    Moreover, the proof of Lemma~\ref{lem:albin-expansion} shows that $\nabla$ has degree zero, from which we deduce that $\overline{\nabla}$ and $L$ have degree zero.
    Therefore $\overline{\nabla}{}^\ell L$ is even.

    Now let $\omega$ be a natural submanifold tensor of bi-rank $(\ell,0)$.
    Then $\omega$ is a linear combination of partial contractions of tensors of the form~\eqref{eqn:tensor-to-be-contracted}.
    Since $j^\ast g_+$ is strongly even, we deduce from the previous paragraph that $\omega \in \overline{\mF}{}^\ell(1)$.
\end{proof}

One consequence of Lemma~\ref{lem:expansion} is our main result about renormalized extrinsic curvature integrals:

\begin{proof}[Proof of Theorem~\ref{thm:albin}]
    Recall that conformally compact Einstein manifolds are even asymptotically hyperbolic manifolds~\cite{ChruscielDelayLeeSkinner2005}.
    The conclusion follows from Proposition~\ref{prop:fundamental-renormalization} and Lemmas~\ref{lem:even-submanifolds-facts} and~\ref{lem:expansion}.
\end{proof}

Lemma~\ref{lem:expansion} also implies that the renormalized extrinsic curvature integral of  a natural extrinsic divergence is zero (cf.\ \cite{CKLTY2024}*{Lemma~4.1}):

\begin{lemma}
    \label{lem:renormalized-divergence-is-zero}
    Let $j \colon \Sigma^k \to (M^n,g_+)$ be an asymptotically minimal submanifold of an even asymptotically hyperbolic manifold.
    Let $\omega$ be a natural submanifold one-form.
    Then
    \begin{equation*}
        \intr \odivsymbol\omega \darea = 0 .
    \end{equation*}
\end{lemma}

\begin{proof}
    This follows immediately from Proposition~\ref{prop:fundamental-renormalization} and Lemmas~\ref{lem:even-submanifolds-facts} and~\ref{lem:expansion}.
\end{proof}

We conclude by deriving a Gauss--Bonnet--Chern-type formula:

\begin{lemma}
    \label{lem:renormalized-pfaffian}
    Let $j \colon \Sigma^k \to (M^n,g_+)$ be an asymptotically minimal submanifold of an even asymptotically hyperbolic manifold.
    Then
    \begin{equation*}
        (2\pi)^{k/2}\chi(\Sigma) = \intr \overline{\Pf} \darea ,
    \end{equation*}
    where $\overline{\Pf}$ is the Pfaffian of $j^\ast g_+$.
\end{lemma}

\begin{proof}
    Lemma~\ref{lem:even-submanifolds-facts} implies that $(\Sigma^k,j^\ast g_+)$ is an even asymptotically hyperbolic manifold.
    The conclusion follows from a result of Albin~\cite{Albin2009}*{Theorem~4.5}.
\end{proof}

\section{Computing renormalized extrinsic curvature integrals}
\label{sec:renormalize}

In this section we compute a large class of renormalized extrinsic curvature integrals, including the renormalized area.
Our approach is analogous to that used to compute renormalized curvature integrals~\cite{CKLTY2024}.

A basic fact is that the integral of a conformal submanifold scalar of the appropriate weight is automatically convergent on conformally compact submanifolds:

\begin{lemma}
    \label{lem:basic-renormalization}
    Fix positive integers $k < n $ with $k$ even.
    Let $I$ be a conformal submanifold scalar of weight $-k$ on $k$-submanifolds of $n$-manifolds.
    If $j \colon \Sigma^k \to (M^n,g_+)$ is an asymptotically minimal submanifold of an even asymptotically hyperbolic manifold, then
    \begin{equation*}
        \intr I \darea = \int_\Sigma I^{j^\ast g_+} \darea_{j^\ast g_+} .
    \end{equation*}
\end{lemma}

\begin{proof}
    Let $\varrho$ be as in Subsection~\ref{subsec:submanifold} and set $h := \varrho^2j^\ast g_+$.
    The conformal invariance of $I$ implies that
    \begin{equation*}
        \int_{\{ \varrho > \varepsilon\}} I^{j^\ast g_+}\darea_{j^\ast g_+} = \int_{\{ \varrho > \varepsilon \}} I^h \darea_h
    \end{equation*}
    for any $\varepsilon > 0$.
    Lemma~\ref{lem:even-submanifolds-facts} implies that $h$ is of class $C^k$, and hence $I^h$ is of class $C^0$.
    The conclusion readily follows.
\end{proof}

The main result of this section is a general formula for the renormalized extrinsic curvature integral of a straightenable submanifold scalar.

\begin{theorem}
    \label{thm:straight}
    Let $k \in \mathbb{N}$ be an even integer and let $I$ be a straightenable natural submanifold scalar of homogeneity $-2\ell \geq -k$ on $k$-submanifolds of $n$-manifolds.
    Let $j \colon \Sigma^k \to (M^n,g_+)$ be a conformally compact minimal submanifold of a conformally compact Einstein manifold.
    Then
    \begin{equation*}
        \intr I \darea = \frac{(-2)^{\ell-k/2}(\ell-1)!}{(k/2-1)!(k-2\ell-1)!!}\int_\Sigma \iota^\ast\left( \widetilde{\Delta}^{k/2-\ell}\cI \right) \darea ,
    \end{equation*}
    where $\widetilde{I}$ is a straight natural submanifold scalar to which $I$ is associated.
\end{theorem}

\begin{proof}
    Let $I_\ell$ be the straightenable submanifold scalar defined in Proposition~\ref{prop:straightIell}.
    Set $\mathcal{I}_\ell := \iota^\ast\cDelta^{k/2-\ell}\widetilde{I}$.
    Then $\mathcal{I}_\ell$ has weight $-k$.
    Theorem~\ref{thm:albin} implies that renormalized extrinsic curvature integrals are well-defined.
    Applying Lemma~\ref{lem:straight-is-almost-conformal} and then Corollary~\ref{cor:straight-mod-divergence} and Lemma~\ref{lem:renormalized-divergence-is-zero} yields
    \begin{equation*}
        \intr \mathcal{I}_\ell \darea = \intr I_\ell \darea = \frac{(-2)^{k/2-\ell}(k/2-1)!(k-2\ell-1)!!}{(\ell-1)!}\intr I \darea .
    \end{equation*}
    The conclusion follows from Lemma~\ref{lem:basic-renormalization}.
\end{proof}

Our Gauss--Bonnet--Chern-type formula for conformally compact minimal submanifolds follows similarly.

\begin{proof}[Proof of Corollary~\ref{cor:renormalized-area}]
    Combine Equation~\eqref{eqn:evaluate-pfaffian-mod-div-general} with Lemmas~\ref{lem:renormalized-divergence-is-zero}, \ref{lem:renormalized-pfaffian}, and~\ref{lem:basic-renormalization}.
\end{proof}

\section{A rigidity result in hyperbolic space}
\label{sec:applications}

In this section we prove our rigidity result for conformally compact minimal submanifolds of a conformally compact hyperbolic manifold with umbilic conformal infinity.
The key fact is that the graphing function is locally determined to a given order by the conformal infinity~\cite{GrahamWitten1999}, from which we deduce improved order of vanishing of the length of the second fundamental form near the boundary:

\begin{lemma}
    \label{lem:second-fundamental-form-asymptotics}
    Let $j \colon \Sigma^k \to (M^n,g_+)$ be a conformally compact minimal submanifold of a conformally compact hyperbolic manifold.
    Suppose that $j_\infty$ is umbilic.
    Then the second fundamental form of $j$ satisfies $\lvert L \rvert = O(\varrho^k)$.
\end{lemma}

\begin{proof}
    Since the result is local near conformal infinity, we may assume that $(M^n,g_+)$ is hyperbolic space
    \begin{equation*}
        \bigl( M^n , g_+ \bigr) = \bigl( (0,\infty)_t \times \mathbb{R}^{n-1} , t^{-2}(dt^2+dx^2) \bigr)
    \end{equation*}
    and that $j_\infty$ is a piece of the embedding $\mathbb{R}^{k-1} \hookrightarrow \mathbb{R}^{k-1} \times \{ 0 \} \subset \mathbb{R}^{n-1}$.
    A computation of Graham and Witten~\cite{GrahamWitten1999}*{Equations~(2.14) and~(2.15)} implies that
    \begin{equation*}
        j(x) = \bigl( \varrho , x , u(\varrho,x) \bigr)
    \end{equation*}
    for some function $u \colon (0,\infty) \times \mathbb{R}^{k-1} \to \mathbb{R}^{n-k}$ such that $u(\varrho,x) = O(\varrho^{k+1})$.
    Write the components of $u$ as $u^{\alpha'}(\varrho,x) = \varrho^{k+1}f^{\alpha'}(x) + O(\varrho^{k+2})$.
    Direct computation yields
    \begin{align*}
        L(\varrho\partial_0,\varrho\partial_0,\varrho\partial_{\gamma'}) & = (k^2-1)\varrho^{k}f^{\gamma'} + O(\varrho^{k+1}) , \\
        L(\varrho\partial_0,\varrho\partial_{\alpha},\varrho\partial_{\gamma'}) & = O(\varrho^{k+1}) , \\
        L(\varrho\partial_{\alpha},\varrho\partial_{\beta},\varrho\partial_{\gamma'}) & = -(k+1)\varrho^{k}f^{\gamma'}\delta_{\alpha\beta} + O(\varrho^{k+1}) .
    \end{align*}
    Therefore $\lvert L \rvert_{g_+}^2 = k(k-1)(k+1)^2\varrho^{2k}\lvert f^{\alpha'}\rvert^2 + O(\varrho^{2k+1})$.
\end{proof}

We now prove our rigidity result:

\begin{proof}[Proof of Theorem~\ref{thm:rigidity}]
    Lemma~\ref{lem:second-fundamental-form-asymptotics} implies $\lvert L \rvert \in L^p(\Sigma)$ for all $p \in [1,\infty]$.
    Theorem~\ref{thm:straight} then implies that
    \begin{align}
        \label{eqn:renorm-L2ell} \int_\Sigma \iota^\ast\left( (-\cDelta)^{k/2-\ell}\lvert\widetilde{L}\rvert^{2\ell} \right) \darea & = 2^{k/2-\ell}\frac{(k/2-1)!(k-2\ell-1)!!}{(\ell-1)!} \int_\Sigma \lv L \rv^{2\ell} \darea , \\
        \label{eqn:renorm-L2-2} \int_\Sigma \iota^\ast\left( (-\cDelta)^{k/2-2}\lvert\widetilde{L}^2\rvert^2 \right) \darea & = 2^{k/2-2}(k/2-1)!(k-5)!! \int_\Sigma \lv L^2 \rv^2 \darea .
    \end{align}

    We now deduce our inequalities.

    Inequality~\eqref{eqn:totally-geodesic} and the characterization of its equality case follows immediately from Equation~\eqref{eqn:renorm-L2ell}.

    The contracted Gauss equation~\eqref{eqn:gauss} yields
    \begin{equation}
        \label{eqn:tf-ric-pw}\overline{E}_{\alpha\beta} = -L^2_{\alpha\beta} + \frac{1}{k}\lv L \rv^2g_{\alpha\beta} ,
    \end{equation}
    where $\overline{E}_{\alpha\beta} := \overline{R}_{\alpha\beta} - \frac{1}{k}\overline{R}g_{\alpha\beta}$ is the trace-free part of the Ricci tensor of $j^\ast g_+$.
    Therefore
    \begin{equation*}
        \lvert\overline{E}\rvert^2 = \lvert L^2\rvert^2 - \frac{1}{k}\lvert L \rvert^4 .
    \end{equation*}
    Combining this with Equations~\eqref{eqn:renorm-L2ell} and~\eqref{eqn:renorm-L2-2} yields Inequality~\eqref{eqn:einstein} with equality if and only if $(\Sigma,j^\ast g_+)$ is Einstein.
    We now show that $(\Sigma,j^\ast g_+)$ is Einstein if and only if $j$ is totally geodesic.
    Equation~\eqref{eqn:tf-ric-pw} implies that if $j$ is totally geodesic, then $j^\ast g_+$ is Einstein.
    Conversely, if $j^\ast g_+$ is Einstein, then it has constant scalar curvature.
    Contracting the Gauss equation~\eqref{eqn:gauss} twice implies that $\lvert L \rvert^2$ is constant.
    Since $\lvert L \rvert = O(\varrho^{k})$, we conclude that $j$ is totally geodesic.

    Finally, the Gauss equation~\eqref{Gausseun} yields
    \begin{equation*}
        \frac{k-2}{2}\lvert\overline{W}\rvert^2 = -k\lvert L^2\rvert^2 + \frac{k^2-3k+3}{k-1}\lvert L \rvert^4
    \end{equation*}
    in codimension one~\cite{CaseTyrrell2023}*{Section~3}.
    Combining this with Equations~\eqref{eqn:renorm-L2ell} and~\eqref{eqn:renorm-L2-2} yields Inequality~\eqref{eqn:lcf} and its characterization of equality.
\end{proof}

\section*{Acknowledgements}
This project is a part of the AIM SQuaRE ``Global invariants of Poincar\'e--Einstein manifolds and applications''.
We thank the American Institute for Mathematics for their support.

JSC was partially supported by a Simons Foundation Collaboration Grant for Mathematicians and by the National Science Foundation under Award No.\ DMS-2505606. YJL was partially supported by the National Science Foundation-LEAPS grant under Award No.\ DMS-2418740. WY was partially supported by NSFC (Grant No.12571065)

\bibliographystyle{plainnat}
\bibliography{bib}

\end{document}